\documentclass[a4paper,11pt,reqno]{amsart}
\usepackage[utf8]{inputenc}
\usepackage[english]{babel}
\usepackage{amsthm,amsmath,amsfonts,amssymb}
\usepackage{array,enumerate,float,moreverb,appendix,mathrsfs, mathabx, eucal}
\usepackage[svgnames]{xcolor}
\usepackage[final]{graphicx}
\usepackage[colorlinks=true,citecolor=Blue,linkcolor=FireBrick]{hyperref}
\usepackage[lofdepth,lotdepth]{subfig}
\usepackage[left=2.5 cm,right=2.5cm,top=2.5cm,bottom=2.5cm]{geometry}
\numberwithin{equation}{section}

\theoremstyle{plain}
\newtheorem{theorem}{Theorem}
\newtheoremstyle{thm}
  {12pt}
  {7pt}
  { \slshape}
  {}
  {\bfseries \scshape}
  {. }
  { }
  {}
\theoremstyle{thm} 
\newtheorem{prop}{Proposition}
\newtheorem{lem}[prop]{Lemma}
\newtheorem{corol}[prop]{Corollary}

\newtheorem{remarque}[prop]{Remark}
\newtheorem{exemple}[prop]{Example}
\newtheorem{defin}[prop]{Definition}
\newtheorem{notation}[prop]{Notation}
\numberwithin{prop}{section}
\usepackage{etoolbox}
\patchcmd{\section}{\scshape}{\scshape\large}{}{}
\patchcmd{\section}{.7}{1.4}{}{}
\patchcmd{\section}{.5}{1.2}{}{}
\title{The spectral gap of sparse random digraphs.}
\author{Simon Coste}
\date{\today}
\newcommand{\A}{50}
\begin{document}
\bibliographystyle{alpha}

\begin{abstract}

The second largest eigenvalue of a transition matrix $P$ has connections with many properties of the underlying Markov chain, and especially its convergence rate towards the stationary distribution. In this paper, we give an asymptotic upper bound for the second eigenvalue when $P$ is the transition matrix of the simple random walk over a random directed graph with given degree sequence. This is the first result concerning the asymptotic behavior of the spectral gap for sparse non-reversible Markov chains with an unknown stationary distribution. An immediate consequence of our result is a proof of the Alon conjecture for directed regular graphs. 
\end{abstract}

\maketitle
\tableofcontents

\section{Introduction and statement of the results.}

\subsection{Directed configurations.}\label{section_modele}

Given two $n$-tuples of positive integers, say $(d_1^+, \dotsc , d_n^+)$ and $(d_1^-, \dotsc , d_n^-)$, we build a sequence of directed multigraphs $G_1$, $G_2$, \ldots using the configuration model: at each of the $n$ vertices (labeled from $1$ to $n$), we glue tails and heads. The vertex $i$ has $d_i^+$ heads and $d_i^-$ tails. For consistency we ask the total number of tails to be equal to the total number of heads: 
\begin{equation}\label{eq:condition_halfedges}
\sum_{i=1}^n d_i^+ = \sum_{i=1}^n d_i^- := M.
\end{equation}

We then choose uniformly at random a matching of the tails into the heads, that is a random permutation $\sigma_n \in \mathfrak{S}_M$. If $\mathbf{e}$ is a head attached to vertex $x$, we glue it to the tail $\sigma_n(\mathbf{e})=\mathbf{f}$. If $\mathbf{f}$ is attached to vertex $y$, this gives rise to an oriented edge from $x$ to $y$. The whole construction leads to a directed multigraph $G_n$ (we will often say \emph{digraph})  on $n$ vertices called the \textit{directed configuration graph} associated with the so-called \emph{degree sequence} $d_1^+, d_1^-, \dotsc,  d_n^+, d_n^-$. The permutation $\sigma_n$ will sometimes be called the \emph{environment}. 

The random graph $G_n$ will simply be noted $G$, the $n$-dependence being implicit through all this paper. We are interested in properties of $G$ in the asymptotic regime $n \to \infty$: we say that an event depending on $n$ holds \emph{with high probability} if its probability tends to $1$ as $n \to \infty$. 

If $u$ is a vertex, we will adopt the following notations: $E^+(u) $ is the set of all heads attached to $u$, and $E^-(u)$ is the set of all tails attached to $u$. Therefore, $\#E^+(u) = d_u^+$ and $\#E^-(u) = d_u^-$. Through all this paper, and unless specified otherwise, heads will be denoted by the bold letter $\mathbf{e}$ and tails by $\mathbf{f}$.

\subsection{Statement of the theorem and illustrations.}

The transition probability matrix $P$ on the graph $G$ is defined as follows: 
\begin{equation}\label{definP}
P(u,v) = \frac{\# \{\mathbf{e} \in E^+(u): \sigma(\mathbf{e}) \in E^-(v) \}}{d_u^+}.
\end{equation}

The matrix $P$ is thus a random stochastic matrix. The \emph{eigenvalues} of $P$ are the $n$ complex roots (counted with multiplicity) of its characteristic polynomial $\mathrm{det}(P-z\mathrm{I})$. We order them by decreasing modulus: 
\[
|\lambda_n|\leqslant |\lambda_{n-1}|\leqslant \dotsb \leqslant |\lambda_2|\leqslant \lambda_1 = 1.
\]  Recall that all those eigenvalues are random variables depending implicitly on $n$ and on the degree sequence $(d_i^+, d_i^-)_{i \leqslant n}$. We will impose that all the degrees are bounded independently on $n$, meaning that there are two constants $\delta \geqslant 2$ and $\Delta \geqslant \delta$ such that for every $n$, 
\begin{equation}\label{H}\tag{H1}
\delta \leqslant \min \{d_1^+, d_1^-, \dotsc , d_n^+, d_n^- \}  \qquad \text{and} \qquad  \max \{d_1^+, d_1^-, \dotsc , d_n^+, d_n^- \} \leqslant \Delta. 
\end{equation}

Under the first assumption, the minimal degree is greater than two (which means there are no dead-ends) and the graph $G$ is strongly connected with high probability as shown in \cite{cooper-frieze}. Let us introduce a  central parameter of this model:
\begin{equation}\label{defin:rho}
\rho := \sqrt{ \frac{1}{M}\sum_{i=1}^n \frac{d_i^-}{d_i^+} } .
\end{equation}

Our goal is to link the modulus of the second eigenvalue with $\rho$. The main result is the following theorem.

\begin{theorem}\label{mainthm}
Let $P$ be the transition matrix \eqref{definP} of the random digraph associated with the degree sequence $(d_1^+, d_n^-, \dotsc, d_n^+, d_n^-)$ satisfying \eqref{H}. Let $\rho$ be as in \eqref{defin:rho} and define $\tilde{\rho} = \rho \vee \delta^{-1}$. Then, as $n$ goes to infinity, we have for every $\varepsilon>0$: 
\begin{equation}\label{claim}
\lim_{n \to \infty} \mathbf{P}\left( |\lambda_2|> \tilde{\rho} + \varepsilon \right) = 0.
\end{equation}
\end{theorem}

Hence, for every $\varepsilon>0$, with high probability as $n$ goes to infinity, the second eigenvalue satisifies 
\[
|\lambda_2|\leqslant \max \left\lbrace \frac{1}{\delta}, \sqrt{ \frac{1}{M}\sum_{i=1}^n \frac{d_i^-}{d_i^+} }. \right\rbrace + \varepsilon. 
\]
This theorem only provides an upper bound for $|\lambda_2|$; knowing if the bound is optimal and having a symmetric lower bound are questions not adressed in this paper. The following figure shows an illustration of \eqref{claim}.

\begin{remarque}\label{rk:h2}
When $\delta^{-1}$ is  smaller than $\rho$, the bound of theorem \ref{mainthm} is equal to $\rho$. This happens when 
\begin{equation}\label{H2}\delta\rho>1\end{equation} and this is not always verified as shown in the following example:
\[
\begin{cases}d_i^+ = d_i^- = 2 \qquad \forall i \in \{1, ..., 100\}\\d_i^+ = d_i^- = 8 \qquad \forall i \in \{101, \dotsc , 200\}.\end{cases}
\]
This degree sequence satisfies $\rho = \sqrt{n/M} = \sqrt{200/1000} \simeq 0.45 $ and in this case we have $\delta \rho <1$.  In fact, using Jensen's inequality, one can give a slightly stronger form of \eqref{H2}. Let $\pi^-$ be the so-called \emph{in-degree distribution} on vertices $\{1, \dotsc , n\}$, that is $\pi^-(i)  =d_i^-/M$. Let $U$ be a random variable with probability distribution $\pi^-$: we have 
$$\rho^2 = \mathbf{E}\left[ \frac{1}{d^+_U}\right].$$
Using Jensen's inequality for the convex function $x \mapsto 1/x$, we get $\mathbf{E}[d^+_U]^{-1} \leqslant \rho^2$. A direct consequence of hypothesis \eqref{H} is $\delta \leqslant \mathbf{E}[d_U^+] \leqslant \Delta$, so \eqref{H2} is fulfilled when $\mathbf{E}[d_U^+] < \delta^2$. This hypothesis can be interpreted as a concentration hypothesis in the sense that the out-degree of a $\pi^-$-distributed random vertex has an expectation not far from the minimum out-degree. 
\end{remarque}

\begin{figure}[H]\centering
  \begin{center}
    \subfloat[Case with $\tilde{\rho} = \delta^{-1}$.]{\includegraphics[scale=0.4]{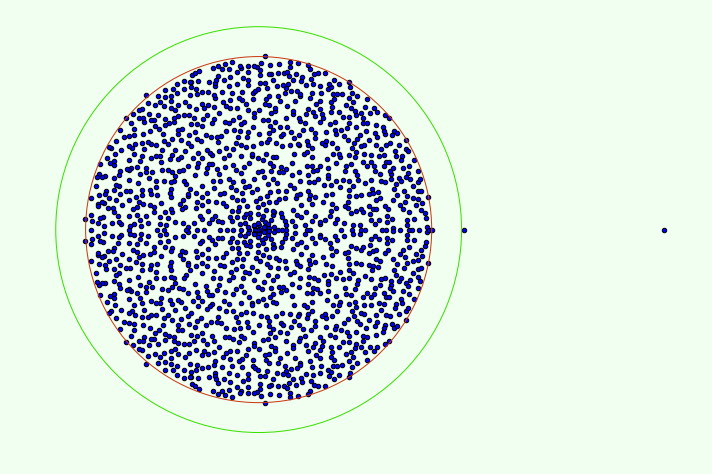}\label{sub:a}}\hfill
    \subfloat[Case with $\tilde{\rho} = \rho$.]{\includegraphics[scale=0.4]{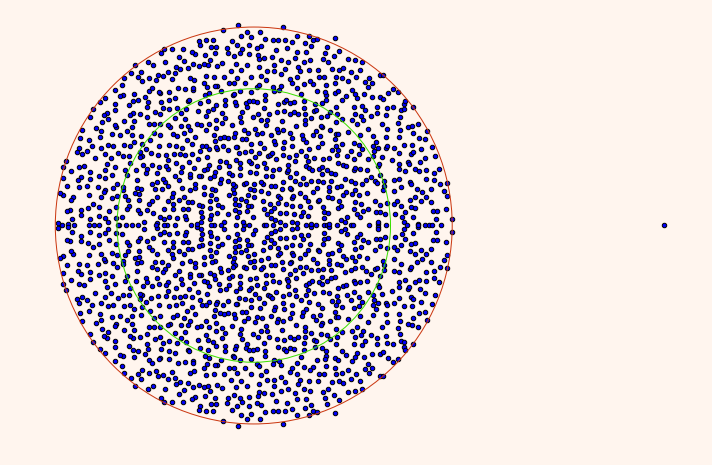}\label{sub:b}}
\caption{Two spectra of the transition matrix on a random configuration digraph. We drew in red the circle with radius $\rho$; in green, the circle with radius $\delta^{-1}$. The rightmost outlier is the Perron eigenvalue $\lambda_1=1$. \\$\bullet$ In figure \protect\subref{sub:a} there are $n=1600$ vertices: $700$ of them have type $(2,2)$ and $800$ have type $(9,9)$. In this case we have $\tilde{\rho} = \delta^{-1}=1/2$. Notice that there are very few outliers outside the circle of radius $\rho$: only one in this case. \\ $\bullet$  In figure \protect\subref{sub:b}, there are $n=1800$ vertices, $600$ of them have type $(5,6)$, $600$ of type $(3,7)$ and $600$ of type $(9,4)$. Here we have $\tilde{\rho} = \rho$.}    \label{fig:spectre}
  \end{center}
\end{figure}

\subsection{Ramanujan digraphs and the Alon conjecture.}

A $d$-regular undirected graph is said to be \emph{Ramanujan} if every eigenvalue $\lambda$ of its transition matrix has $|\lambda|=1$ or $|\lambda|\leqslant 2\sqrt{d-1}/d$. Those graphs have been very well studied, notably for their optimal expansion properties (\cite{ramanujan-book, hoory2006expander}). The reason why the value $2\sqrt{d-1}/d$ appears here is because the universal cover of every $d$-regular graph is the infinite $d$-regular tree $\mathbb{T}_d$, and its transition operator has spectrum $[-2\sqrt{d-1}/d, 2\sqrt{d-1}/d]$, a classical result of Kesten~\cite{kesten59}; Ramanujan graphs are the regular graphs whose non-trivial eigenvalues are included in the spectrum of their universal cover. 

A recent line of research generalized this to digraphs, as recently\footnote{The survey \cite{2018arXiv180408028P} appeared on the ArXiv after the first version of this paper.} surveyed in \cite{2018arXiv180408028P}: the universal cover of a $d$-regular digraph is the infinite $d$-regular tree $\vec{\mathbb{T}}_d$ obtained from the infinite $2d$-regular tree $\mathbb{T}_{2d}$ by assigning a direction for $d$ edges at every vertex and the other direction for the $d$ other edges at this vertex. The spectrum of the transition operator $\vec{\mathbb{T}}_d$ is precisely $\{z \in \mathbb{C} : |z|\leqslant 1/\sqrt{d} \}$ as proven in \cite{MR1231179}. By analogy, a $d$-regular digraph is called Ramanujan if every eigenvalue $\lambda$ of its adjacency matrix has $|\lambda|=1$ or $|\lambda|\leqslant 1/\sqrt{d}$. 

Explicit constructions of Ramanujan graphs have been a challenging problem with a rich history, but one of the most striking phenomenon in the domain is that \emph{most} regular graphs are \emph{nearly} Ramanujan. More precisely, Alon conjectured in \cite{Alon1986} that for every $d,\varepsilon$, the second eigenvalue $\lambda_2$ of the transition matrix of a uniform $d$-regular graph on $n$ vertices is smaller than $2\sqrt{d-1}/d + \varepsilon$ with high probability when $n \to \infty$. The question remained open for two decades and was solved by Friedman in his celebrated 2004 paper \cite{friedman}. In fact, the bound was optimal due to a simple inequality already shown by Alon, sometimes referred to as the Alon-Boppana inequality (\cite{NILLI1991207}). This is now called \emph{Friedman's second eigenvalue theorem}:

\begin{theorem}[\cite{friedman, bordenave2015}]Fix an integer $d >2$. For every $\varepsilon>0$, as $n\to \infty$ we have
\begin{equation}
\mathbf{P}\left( \left| |\lambda_2| - \frac{2\sqrt{d-1}}{d} \right|>\varepsilon \right) \to 0.
\end{equation}
\end{theorem}

This solved the first-order asymptotic behaviour of the second eigenvalue for regular graphs; we refer the reader to the introductions of \cite{Alon1986,bordenave2015,ramanujan-book,hoory2006expander} for further reference. When it comes to regular digraphs, our main theorem settles the Alon conjecture for \emph{digraphs} (see \cite[section 5.5]{2018arXiv180408028P}). In fact, in a $d$-regular digraph, we have $d_i^+ = d_i^- = d$, hence $\tilde{\rho}$ is equal to $\frac{1}{d} \vee \frac{1}{\sqrt{d}} = \frac{1}{\sqrt{d}}$. We state this as a corollary. 

\begin{corol}\label{corol:friedman}
Let $d\geqslant 2$ be a fixed integer and $P$ be the transition matrix of a random $d$-regular digraph. Note $|\lambda_n|\leqslant \dotsb \leqslant |\lambda_2|\leqslant \lambda_1=1$ the eigenvalues of $P$, ordered by decreasing modulus. Fix $\varepsilon>0$. Then, as $n$ goes to infinity, the following holds with high probability: 
\begin{equation}
|\lambda_2| \leqslant \frac{1}{\sqrt{d}}+\varepsilon. 
\end{equation}
\end{corol}

 \subsection{Motivation, background and related work.}

 \subsubsection*{Random digraphs.}

In this paper, we consider random directed (multi)graphs with a specified sequence of in-degrees and out-degrees; when all the degrees are equal to $d$, this model reduces to the directed $d$-regular case. Our construction with half-edges is a directed variant of the classical configuration model (see \cite{bollobas-rg}). When the degrees are bounded independently of the size of the graph, such multigraphs are sparse, meaning they have few edges. Even if digraphs are much more difficult to handle than undirected graphs, they are also one step closer to reality when modelling real-life situtations: see \cite{newman,cooper2011random} and references for (many) examples of graph-modelling that go beyond the Internet graph. 

\subsubsection*{Eigenvalues of Markov chains.}
 
 Many strong connections exist between the second eigenvalue of a transition matrix and the convergence properties of the corresponding Markov chain. The following proposition is the most known result:
 
 \begin{prop}[\cite{peres}, \cite{montenegro}]\label{prop:MCCV}Let $P$ be the transition matrix of an irreducible, aperiodic Markov chain on the finite state space $S=\{1, \dotsc, n\}$ with stationary distribution $\pi_\star$. Let $1 = |\lambda_1|\geqslant |\lambda_2| \geqslant \dotsb \geqslant |\lambda_n|$ be the eigenvalues of $P$ ordered by decreasing modulus and $d(n)$ be the \emph{distance to equilibrium} at time $n$, defined as $d(n) = \max_{x \in S}\Vert P^n(x, \cdot) - \pi_\star \Vert_\mathrm{TV}$, with $\Vert \cdot\Vert_{\mathrm{TV}}$ the usual total variation distance. Then, 
\begin{equation}
\lim_{n \to \infty} d(n)^\frac{1}{n} = |\lambda_2|.
\end{equation}
\end{prop}

In other words, large values of the \emph{spectral gap} $\gamma_\star :=1 - |\lambda_2|$ are linked with fast convergence. For random walks on graphs, $\lambda_2$ is also known to be strongly linked with expansion properties of the underlying graph (see~\cite{hoory2006expander} for an excellent survey). It is thus of special interest to study the spectrum of transition matrices; however, instead of focusing on a fixed chain $P$, researchers now study ``generic" models of transition matrices. Most of the time, the transition matrix is chosen at random among a certain type of matrices and its properties are studied in a probabilistic setting. In this line, random walks on random graphs have attracted an extraordinary attention during the last decades.  

Another very important aspect of Markov chains linked with $|\lambda_2|$ is \emph{mixing}, and especially the cutoff phenomenon (\cite{diaconis1996cutoff, peres}). Proving cutoffs for large classes of random walks is an active line of research. In the context of random graphs, cutoff had been proven with high probability in the $d$-regular model (\cite{cutoff_dreg}),  but it was recently shown by Lubetzky and Peres in their influential paper \cite{cutoff_ramanujan} that \emph{every} Ramanujan graph exhibits cutoff, suggesting that optimality of the second eigenvalue is linked with optimal mixing. Our paper gives the first upper bound for the second eigenvalue for a non-reversible model of Markov chains. The cutoff phenomenon for our model has been established whp in the inspiring paper \cite{caputo_salez_bordenave}, with a logarithmic mixing time (see Theorems 1 and 2 in \cite{caputo_salez_bordenave}).   Note that our main result (Theorem \ref{mainthm}) immediately implies Theorem 3 in \cite{caputo_salez_bordenave}, as a consequence of Proposition \ref{prop:MCCV}.

\subsubsection*{Random transition matrices.}

While we are interested in the spectral gap of a special kind of those matrices, some serious advances on global asymptotics of the spectrum have recently been made. In a series of papers \cite{cook-srrd,cook_rrd,cook_dreg2}, Nicholas Cook and coauthors established convergence towards the circular law of the empirical spectrum of matrices related to the adjacency matrix of $d$-regular directed graphs, when $d$ grows to infinity with $n$. In another series of papers (\cite{BCC_rev, BCC_revex,BCC,BCCP}), Bordenave, Caputo and Chafaï considered the spectra of a transition matrix $P$ constructed by row-normalizing a random matrix with nonnegative iid entries $X_{i,j}$, that is $P(i,j) := X_{i,j} \rho(i)^{-1}$ where $\rho(i):=X_{i,1}+...+X_{i,n}$. 
A key result is formulated in \cite{BCCP} where the authors prove the convergence towards the circular law in the sparse case where the $X_{i,j}$ are heavy-tailed with index $\alpha \in ]0,1[$. They also conjecture (\cite{BCCP} remark 1.3) that in this case, whp the second eigenvalue $|\lambda_2|$  will be smaller than $\sqrt{1-\alpha}$. We believe that our method could be adapted to tackle this conjecture.

\subsubsection*{Non-reversible chains.} A key feature of random walks on random unoriented graphs is \emph{reversibility} of the Markov chain. When the walk is reversible, the transition matrix $P$ has a known stationary distribution $\pi_\star$ and is self-adjoint relatively to the hilbert product $\langle \cdot, \cdot\rangle_\star$ defined by 
\[
\langle x, y \rangle_\star = \sum_{x \in V} x_i y_i \pi_\star(i) \qquad (x,y \in \mathbb{R}^n ).
\]

In this reversible case, all the classical tools from hermitian algebra can be used to study the spectrum of $P$. When $P$ is not reversible but when its stationary distribution $\pi_\star$ is known, we can still use the \emph{reversibilization trick} introduced by Fill (\cite{fill1991eigenvalue}; see also \cite{montenegro}): if $P^*$ denotes the time-reversibilization of $P$, defined as $P^*(i,j) = P(j,i)\pi_\star(j) \pi_\star(i)^{-1}$, then $PP^*$ is self-adjoint for $\langle \cdot, \cdot \rangle_\star$. All the eigenvalues $1=\mu_1 \geqslant \mu_2 \geqslant ... \geqslant \mu_n \geqslant 0$ of $PP^*$ are real and positive, and $\mu_2 \geqslant |\lambda_2|^2$, thus giving informations about $|\lambda_2|$. However, in any model where $\pi_\star$ is not explicitly known, those techniques are useless. 

Our method is the first one to efficiently deal with the top eigenvalue of non-hermitian matrices with no information on the eigenvectors; we strongly believe this method could prove extremely useful in other problems within the random matrix theory, especially in the non-hermitian setting. 

In fact, after the first version of this paper was put on the ArXiv, other results on the spectral gap of random matrix models have been proven with this method, such as the spectral gap for random biregular bipartite graphs \cite{dumitriu_bireg}, and for sparse bistochastic matrices \cite{bordenave_qiu}.

\bigskip

We finally mention some related questions and conjectures. 
\begin{enumerate}
\item What is the link between $|\lambda_2|$ and the cutoff phenomenon for the Markov chain ? Do \emph{all} graphs in our model having $|\lambda_2|\leqslant \rho$ exibit cutoff ?
\item Is the upper bound \eqref{claim} optimal ? In the Friedman theorem, the difficult part was to prove the upper bound while the lower bound had been proven very early (\cite{NILLI1991207}) using the full strength of the symmetric nature of $P$. We have proven an upper bound for our model, but no lower bound is known yet. 
\item This paper deals with the second eigenvalue of random digraphs in general. In the specific case of $d$-regular digraphs, it is conjectured in \cite[Section 7]{bordenave_chafai_survey} that the whole empirical spectral measure of the adjacency matrix of a $d$-regular digraph converges almost surely in distribution to $\mu_{\textsc{OKMC}}$, a complex version of the Kesten-McKay distribution, namely
\[
\mu_{\textsc{OKMC}}(\mathrm{d}z) = \pi^{-1} \frac{d^2(d-1)}{(d^2 - |z|^2)^2}\mathbf{1}_{|z|\leqslant \sqrt{d}} \mathrm{d}z.
\]
\end{enumerate}

\subsection{Conventions and notations.}\label{nota:deg}
The operator norm of a real square matrix $A \in \mathcal{M}_n(\mathbb{R})$ is  $$\Vert A\Vert  = \sup_{x \neq 0} \frac{\Vert Ax\Vert }{\Vert x\Vert }$$
where $\Vert x \Vert = (x_1^2+ \dotsb +x_n^2)^\frac{1}{2}$ is the standard euclidean norm. If $M$ is any matrix, $A^\top$ is its usual transpose. We will also note $\mathbf{1}$ the column vector $\mathbf{1} = (1,\dotsc ,1)^\top$.
If $(a_n)$ and $(b_n)$ are two real sequences, we use the classical Landau notations $a_n \sim b_n, a_n = o(b_n)$ and $a_n = O(b_n)$. 

\bigskip

We will also adopt the following notations for half-edges in our model. Formally, a half-edge will be coded by a triple $(u,i,\varepsilon)$, where 
\begin{itemize}
\item $u$ is a vertex, 
\item $\varepsilon \in \{-, +\}$ is a sign indicating the nature of the half-edge: a $+$ symbol denotes a head, a $-$ denotes a tail, 
\item $i$ is an integer in $\{1, \dotsc,  d_u^\varepsilon \}$. 
\end{itemize}
With this notation, we have $E^+(u) = \{(u,i,+): i = 1, \dotsc,  d^+_u) \}$ and also $E^-(u) = \{(u,i,-): i = 1, \dotsc, d^-_u) \}$. These notations will specifically be used in the combinatorial section \ref{sec:combi}. In general, it will be more convenient to adopt the following conventions, much easier to read: heads will be denoted by the bold letter $\mathbf{e}$ and tails will be denoted by the bold letter $\mathbf{f}$. If a half-edge $\mathbf{e}$ is attached to vertex $u$, we will write $d_\mathbf{e}^\pm$ instead of $d_u^\pm$.

For example, a 2-step path in the graph between vertices $a$ and $b$ is a sequence of the form $(\mathbf{e}_1, \mathbf{f}_1, \mathbf{e}_2, \mathbf{f}_2 )$ with $\mathbf{e}_1$ attached to $a$, $\mathbf{f}_2$ attached to $b$, $\mathbf{e}_2$ and $\mathbf{f}_1$ attached to the same vertex and $\sigma(\mathbf{e}_1) = \mathbf{f}_1, \sigma(\mathbf{e}_2) = \mathbf{f}_2$. We will give a complete and precise definition of \emph{paths} further in the paper.

In the rest of the paper, we will denote all universal constants by $C>0$. 

\subsection{Acknowledgement.}

The author is grateful to his advisors Charles Bordenave and Justin Salez for their valuable help and advice during the writing of this paper, from preliminary discussions about the problem and the understanding of \cite{bordenave2015} to the final remarks on the manuscript. 

\section{Proof of the main theorem.}

\subsection{Outline.}We give a motivated sketch of the main difficulties in the proof of our theorem and the core ideas to overcome them.

As mentionned in the beginning of \cite{friedman} or \cite{bordenave2015}, the standard trace method for bounding $|\lambda_2|$ is doomed to fail: the main obstruction comes from the fact that with small probability, some very small graphs with many cycles (``tangled graphs") are present in the graph, and they drastically perturb the expectation of the trace of $P^t$. To tackle the problem, a powerful idea is to use a \emph{selective trace}. 

Recall that the coefficient $(i,j)$ of $P^t$ is the sum over all paths of length $t$ from $i$ to $j$ of the probability that the  simple random walk follows this path. Instead of taking all those paths, we are going to select only those that are not ``too much tangled" and replace the matrix $P^t$ with a ``tangle-free" matrix $P^{(t)}$ --- all proper definitions will be stated in Section \ref{sec:defintangle} --- and use the fact that with high probability, when $t$ is not too large, there are no tangles in the original graph (Proposition \ref{lemme_pas_de_noeuds}). This idea was introduced in \cite{friedman} for the proof of the Friedman theorem and was refined in \cite{bordenave2015} and \cite{bordenave_lelarge_massoulie}. 

In the models studied in these papers, it was easier to study paths that are non-backtracking, i.e. that do not take the same edge twice in a row. In our own model of directed graphs, no edge can be crossed twice in a row except self-loops --- which are rare --- hence we can concentrate on the transition matrix $P^t$ or its tangle-free analog $P^{(t)}$ instead of resorting to non-backtracking matrices. 

The next step will be to relate the second eigenvalue of $P^{(t)}$ with the matrix norm of different other related matrices, namely $\underline{P}^{(t)}$ and $R^{t,\ell}$, defined in \ref{sec:tangled_remainders}. Those matrices are easier to study, because their components are nearly centered. Their norms are given in Propositions \ref{control1} and \ref{control2}. 

The key difficulty of our model, compared to the regular case studied in \cite{bordenave2015}, lies in the fact that the stationary distribution is unknown. In the regular case, the stationary distribution --- i.e., the top left-eigenvector --- is known to be $(1/n, \dotsc,  1/n)$, which could be used in Lemma 3 of \cite{bordenave2015} for deriving a Courant-Fisher-like variational formulation of $|\lambda_2|$. This is no longer the case here and we had to perform different algebraic manipulations and to approximate the stationary distribution; this will be done in the proof of Proposition \ref{local_eigvals} (Section \ref{section:prop_loc_eigvals}).

\subsection{Definitions: tangles and variants of $P$.}\label{sec:defintangle}

This subsection introduces the main tools for our proof of Theorem \ref{mainthm}.

\subsubsection{Paths.}

Even though the graph $G$ is a multigraph, its construction with half-edges described in Section \ref{section_modele} is extremely useful and will be of paramount importance in the paper. This is why we do not define paths as a usual path in a graph (or multigraph), but as a sequence of half-edges that could be paired through $\sigma$. Through all the sequel, $t>0$ is an integer.

\begin{defin}\label{def:path} A path of length $t$ between vertices $i$ and $j$ is a sequence of half-edges $(\mathbf{e}_1, \mathbf{f}_1, \dotsc, \mathbf{e}_t, \mathbf{f}_t)$ such that 

\begin{enumerate}\item for every $s \leqslant t$, $\mathbf{e}_s$ is a head and $\mathbf{f}_s$ is a tail,
\item for every $s<t$, $\mathbf{f}_s$ and $\mathbf{e}_{s+1}$ are attached to the same vertex, 
\item $\mathbf{e}_1$ is attached to $i$ and $\mathbf{f}_t$ is attached to $j$.
\end{enumerate}

We note $\mathscr{P}^t(i,j)$ the set of paths of length $t$ connecting $i$ to $j$. Usually, we will denote paths by the bold letter $\mathbf{p}$, meaning $\mathbf{p} = (\mathbf{e}_1, \mathbf{f}_1, \dotsc, \mathbf{e}_t, \mathbf{f}_t)$.

\end{defin}

Keep in mind that our definition of a \emph{path} does not depend on $\sigma$ or $G$: it is a \emph{potential} path in $G$. The path itself is a purely combinatorial object and is not random; it will become a true path in the random graph $G$ if in addition, $\sigma(\mathbf{e}_s) = \mathbf{f}_s$ for every $s \in \{1, \dotsc, t\}$. In this setting we have the following useful expression for powers of the matrix $P$: 
\begin{equation}P^t(i,j) = \sum_{\mathbf{p} \in \mathscr{P}^t(i,j)} \prod_{s=1}^t \frac{ \mathbf{1}_{\sigma(\mathbf{e}_s) = \mathbf{f}_s}}{d_{\mathbf{e}_s}^+} \end{equation}

where $d_{\mathbf{e}}^+$ is in fact $d^+_{u}$ if the half-edge $\mathbf{e}$ is attached to the vertex $u$ (see notation \ref{nota:deg}). When $t=1$, this expression reduces to 
$$P(i,j) = \sum_{\mathbf{e}\in E^+(i)}\sum_{\mathbf{f} \in E^-(j)} \frac{\mathbf{1}_{\sigma(\mathbf{e}) = \mathbf{f}}}{d_i^+}.$$

Taking expectations on both sides yelds the following identity: 
\begin{equation}
\label{exp_P}
\mathbf{E}[P(i,j)] = \frac{d_j^-}{M} := \pi^- (j).
\end{equation}

The probability distribution $\pi^-$ is also called the \emph{out-degree distribution}.

\subsubsection{Tangles and cycles.}In an oriented multigraph, we say that two vertices $u$ and $v$ are adjacent if there is an edge between them, regardless of its orientation. A \emph{cycle} is a sequence of vertices $(x_1, \dotsc, x_n)$ such that for every $i\neq n$, $x_i$ and $x_{i+1}$ are adjacent and $x_n$ is adjacent with $x_1$. Loops and multi-edges count as cycles. 

If $G$ is an oriented multigraph and $x,y$ are two vertices, a digraph-path from $x$ to $y$ is a sequence $(x_1, \dotsc, x_n)$ such that $x_1=x, x_n=y$, and for every $i$ the vertex $x_i$ leads to the vertex $x_{i+1}$. Its length is $n-1$. We denote by $d(x,y)$ the length of the shortest digraph-path from $x$ to $y$.  Let $x$ be a vertex and $r$ a positive integer. The forward ball of center $x$ and radius $r$, noted $B^+(x,r)$, is the oriented multigraph induced by $G$ on the vertices $y$ such that $d(x,y) \leqslant r$. 

\bigskip 

We now give our first definition of \emph{tangles}, in the context of digraphs: 
\begin{itemize}
\item Let $G$ be an oriented multigraph. We say that it is \textbf{tangled} if it has at least two cycles. If $G$ is not tangled, it is tangle-free. 
\item Let $d$ be a positive integer. If, for every vertex $x$, the oriented multigraph $B^+(x,d)$ is tangle-free, we say that $G$ is \textbf{$\mathbf{d}$-tangle free}. Otherwise, it is $d$-tangled. 
\end{itemize}

\begin{figure}[H]\centering

\begin{tabular}{cc}
\includegraphics[scale=1.3]{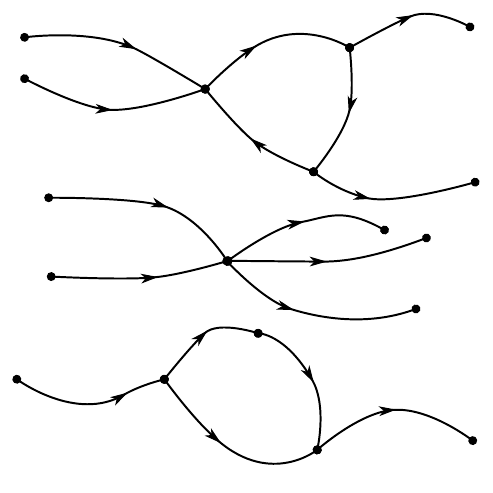}  &\includegraphics[scale=1.3]{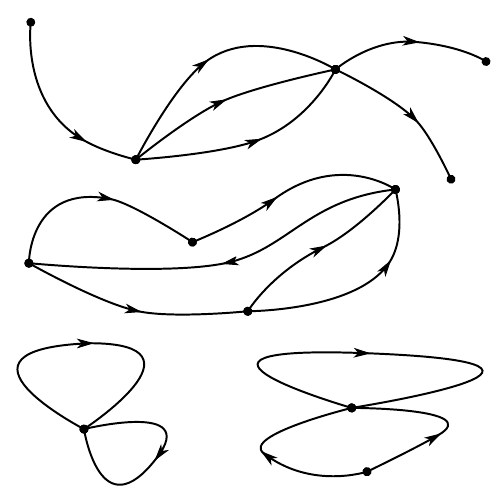}  \\
Some examples of tangle-free digraphs. & Some tangled digraphs.
\end{tabular}
\caption{Examples.}
\end{figure}

We now extend this to \emph{paths}, as defined in definition \ref{def:path}. Fix a path $\mathbf{p}$. It induces an oriented multigraph $G(\mathbf{p})$ with the following construction: 
\begin{itemize}
\item the vertices of $G(\mathbf{p})$ are the vertices having an half-edge appearing in $\mathbf{p}$, 
\item the number of edges going from vertex $x$ to vertex $y$ is the number of \emph{distinct} couples $(\mathbf{e}, \mathbf{f})$ appearing in $\mathbf{p}$, such that $\mathbf{e}$ is a head attached to $x$ and $\mathbf{f}$ is a tail attached to $y$.
\end{itemize}

If $(\mathbf{e}, \mathbf{f})$ appears more than once in the path $\mathbf{p}$, then it will only account for one edge in $G(\mathbf{p})$. The definition of tangles naturally extends to paths $\mathbf{p}$: 

\begin{defin}[tangle-free paths]Let $\mathbf{p}$ be a path. It is tangle-free if $G(\mathbf{p})$ is tangle free. The set of all paths of length $t$ going from $i$ to $j$ that are tangle-free will be noted $\mathscr{T}^t(i,j)$. 
\end{defin}

Note that a path $\mathbf{p}$ can be tangle-free and have a cycle crossed many times. For example, fix a head $\mathbf{e}$ and a tail $\mathbf{f}$ attached to the same vertex $x$. Define the path 
$$\mathbf{p} = (\mathbf{e}, \mathbf{f}, \mathbf{e}, \mathbf{f}, \mathbf{e}, \mathbf{f}).$$
The corresponding graph $G(\mathbf{p})$ is the simple loop based at $x$, which has only one cycle, thus $\mathbf{p}$ is tangle-free. However, the loop is explored three times by the path $\mathbf{p}$. 

Now take another tail attached to $x$, say $\mathbf{f}'$. Consider the path 
$$\mathbf{q} = (\mathbf{e}, \mathbf{f}, \mathbf{e}, \mathbf{f}').$$
Then $G(\mathbf{q})$ is simply the multigraph with one vertex and two distinct loops based at $x$, thus $\mathbf{q}$ is tangled.

\subsubsection{Variants of $P$.}

We now define: 

\begin{itemize}
\item the \textbf{centered} analogue of $P^t$, which is $\underline{P}^t$ defined by 
\begin{equation}\label{def:CP}\underline{P}^t(i,j) = \sum_{\mathbf{p} \in \mathscr{P}^t(i,j)} \prod_{s=1}^t \frac{ \mathbf{1}_{\sigma(\mathbf{e}_s) = \mathbf{f}_s} - 1/M}{d_{\mathbf{e}_s}^+} .\end{equation}

Using \eqref{exp_P}, we see that the matrix $\underline{P}^1$ is centered. This is not true for $\underline{P}^t$, but an important step in this work will be to prove that $\underline{P}^t$ is \emph{nearly} centered. 

\item the  \textbf{tangle-free}   analogue of $P$, defined by
\begin{equation}\label{def:TFP}P^{(t)}(i,j) = \sum_{\mathbf{p} \in \mathscr{T}^t(i,j)} \prod_{s=1}^t \frac{ \mathbf{1}_{\sigma(\mathbf{e}_s) = \mathbf{f}_s} }{d_{\mathbf{e}_s}^+}.\end{equation}

Here, we just got rid of all the tangled paths. When the underlying graph is $t$-tangle free, we obviously have $P^t = P^{(t)}$.

\item and finally the  \textbf{centered tangle-free}   analogue of $P$, defined by
\begin{equation}\label{def:CTFP}\underline{P}^{(t)}(i,j) = \sum_{\mathbf{p} \in \mathscr{T}^t(i,j)} \prod_{s=1}^t \frac{ \mathbf{1}_{\sigma(\mathbf{e}_s) = \mathbf{f}_s} - 1/M}{d_{\mathbf{e}_s}^+}. \end{equation}
\end{itemize}

The matrix $\underline{P}^{(t)}$ is the main tool of the forthcoming analysis, because it is ``nearly centered" and the sum runs over tangle-free paths. A key step in this paper will be to check if the perturbation $P^t - \underline{P}^{(t)}$ is small: to this end, first remark that the sparsity of the graph $G$ implies that tangles are not frequent if we choose the right scale for the path length $t$:

\begin{prop}\label{lemme_pas_de_noeuds}
Let $G$ be the random graph associated with the degree sequence $(d_i^+, d_i^-)$ satisfying  hypothesis \eqref{H}. Define $t = \lceil \alpha \log_\Delta (n) \rceil$. Then, as $n$ goes to infinity, we have 

\begin{equation}
\lim_{n \to \infty} \mathbf{P}(G \text{ is }t\text{-tangled }) = 0.
\end{equation}
\end{prop}

The proof relies on a classical breadth-first-search exploration argument and can be found in section 3.2 of \cite{caputo_salez_bordenave}. In particular, under assumption \eqref{H}, $\mathscr{T}^s = \mathscr{P}^s$ with high probability for every $s \leqslant t$, so $P^s = P^{(s)}$. Some related work on cycles in those random digraphs can be found in \cite{cooper-frieze}.

For the rest of the paper, we fix $t$ as in the preceding proposition with $\alpha<1/4$, that is 
\begin{equation}\label{def:t}
t = \lceil \alpha \log_\Delta(n) \rceil .
\end{equation} 

The parameter $\alpha$ can be chosen arbitrarily small, as long as it is strictly smaller than $1/4$. This freedom will be used in Section \ref{use_alg}.

\subsubsection{Tangled remainders}\label{sec:tangled_remainders}

We finally define our last ingredient: tangles. We first need a notation for the concatenation of two paths.

\begin{notation}[concatenation]\label{nota:conca}If $\mathbf{p} = (\mathbf{e}_s, \mathbf{f}_s)_{1\leqslant s \leqslant k}$ is a path of length $k$ and if $\mathbf{p}' = (\mathbf{e}'_s, \mathbf{f}'_s)_{1\leqslant s \leqslant k'}$ is a path of length $k'$, with $\mathbf{f}_k$ attached to the same vertex as $\mathbf{e}'_1$, then the concatenation $(\mathbf{p}, \mathbf{p}')$ will be the path of length $k+k'$ defined by
$$(\mathbf{e}_1, \mathbf{f}_1, \dotsc, \mathbf{e}_k, \mathbf{f}_k, \mathbf{e}'_1, \mathbf{f}'_1, \dotsc, \mathbf{e}'_{k'}, \mathbf{f}'_{k'} ). $$
This definition obviously extends to the concatenation of three or more paths, provided that the final tail of each path is attached to the same vertex as the beginning head of the next path.
\end{notation}

\begin{defin}\label{defin:RTL}$\mathscr{R}^{t,\ell} (i,j)$ is the set of all \textit{tangled} paths $\mathbf{p}$ going from $i$ to $j$, but which can be written in the form $\mathbf{p} = (\mathbf{p}_1,  \mathbf{p}_2,  \mathbf{p}_3 )$ where
\begin{itemize}
\item the path $\mathbf{p}_1$ belongs to $\mathscr{T}^{\ell-1}(i,g)$ where $g$ is a vertex of the graph,
\item $\mathbf{p}_2 = (\mathbf{e}, \mathbf{f} )$ is a path which goes from $g$ to $h$ in only one step, with $h$ a vertex of the graph,
\item the path $\mathbf{p}_3$ belongs to $\mathscr{T}^{t-\ell}(h, j)$.
\end{itemize}

We also define the \emph{tangled rest} by 

\begin{equation}\label{def:tangled_rest}
R^{t,\ell}(i,j) = \sum_{\mathbf{p} \in \mathscr{R}^{t,\ell}(i,j)} \prod_{s=1}^{\ell-1} A(\mathbf{e}_s, \mathbf{f}_s ) \frac{1}{d^+_{\mathbf{e}_{\ell}}} \prod_{s=\ell+1}^t \underline{A}(\mathbf{e}_{s}, \mathbf{f}_s ).
\end{equation}

\end{defin}

In other words, the set $\mathscr{R}^{t,\ell}$ is the set of all paths that can be obtained by gluing two tangle-free paths with a bridge, but which in the end are tangled. 
\begin{figure}[H]
\centering
\includegraphics[scale=0.7]{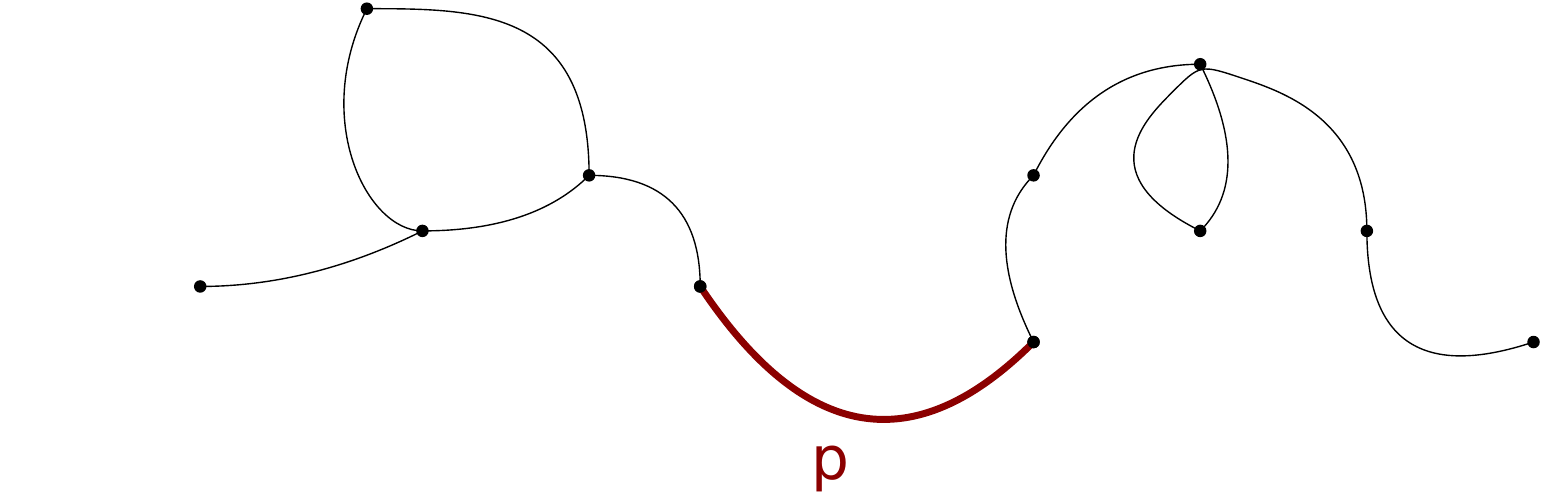}
\caption{An element in $\mathscr{R}^{t,\ell}$. The two black paths are tangle-free, but when we glue them together with the ``bridge" $ \mathbf{p} $, we create a tangle.}
\end{figure}

\subsection{Proof of the main theorem.}\label{use_alg}

The main algebraic idea of the proof relies on the fact that one can bound $|\lambda_2|$ using the operator norm of matrices $\underline{P}^{(t)}$ and $R^{t,\ell}$ for $\ell \leqslant t$. The core of the paper will consist in bounds for $\Vert \underline{P}^{(t)}\Vert $ and $\Vert R^{t,\ell}\Vert $. Recall that $\tilde{\rho} = \rho \vee \delta^{-1}$.

\begin{prop}\label{control1}
Let $t$ be as in \eqref{def:t}. For any $c>1$, with high probability, we have 
\begin{equation}
\Vert \underline{P}^{(t)}\Vert  \leqslant \ln(n)^D (c\tilde{\rho})^{t}, 
\end{equation}
where $D$ is a positive constant. 
\end{prop}

\begin{prop}\label{control2}
Let $t$ be as in \eqref{def:t} and let $\ell$ be in $\{1, \dotsc, t \}$. With high probability, we have 
\begin{equation}
\Vert R^{t,\ell}\Vert  \leqslant n \ln(n)^D (c\tilde{\rho})^{t+\ell}
\end{equation}
where $D$ is a positive constant.
\end{prop} 

The proof of those two propositions is an application of the classical trace method and is quite technical. It will be postponed at Sections \ref{tech_section} - \ref{section:tangled_rest}. We now state the central proposition for bounding the second eigenvalue of $P$. Its proof is exposed in Section \ref{section:prop_loc_eigvals}. 

\begin{prop}\label{local_eigvals}
With high probability, the second eigenvalue $\lambda_2$ of the matrix $P$ satisfies the following inequality: 
\begin{equation}\label{eqe:local_eigvals}
|\lambda_2|^t\leqslant 2 \ln(n)^3 \left( \Vert \underline{P}^{(t)}\Vert  + \frac{1}{M} \sum_{\ell=1}^t \Vert R^{t,\ell}\Vert  \right).
\end{equation} 
\end{prop}

We now conclude the proof of Theorem \ref{mainthm} from Propositions \ref{local_eigvals}, \ref{control1} and \ref{control2}.  For simplicity, note

\begin{equation}\label{def:K}
K_t = \Vert \underline{P}^{(t)}\Vert  + \frac{1}{M} \sum_{\ell=1}^t \Vert R^{t,\ell}\Vert .
\end{equation}

 As a direct consequence of the two preceding theorems and the fact $M \geqslant \delta n \geqslant n$, it is clear that with high probability, $K_t \leqslant \ln(n)^D (c\tilde{\rho})^{t} + \delta^{-1}(c\tilde{\rho})^t \ln(n)^D \sum_{\ell=1}^t  (c\tilde{\rho})^\ell$ which is equal to $ \ln(n)^D (c\tilde{\rho})^{t} \big(1+\delta^{-1}c\tilde{\rho}\frac{(c\tilde{\rho})^t - 1}{c\tilde{\rho} -1} \big)$.  If $c$ is close enough to $1$ to ensure that  $c \tilde{\rho}<1$, then as $n$ goes to infinity the term $1+\delta^{-1}c\tilde{\rho}\frac{(c\tilde{\rho})^t - 1}{c\tilde{\rho} -1}$ is bounded by some absolute constant $C$. We have proven that, with high probability, 
\begin{equation}\label{borne:kt}K_t \leqslant  \ln(n)^D (c\tilde{\rho})^t C.
\end{equation}

We now use Proposition \ref{local_eigvals} which states that $|\lambda_2|^t \leqslant 2\ln(n)^3K_t$, hence 
\begin{equation}
|\lambda_2|^t \leqslant 2C\ln(n)^{D+3}(c\tilde{\rho})^t.
\end{equation}
Take powers $1/t$ on both sides and use $t = \Theta(\ln(n))$: 
\begin{equation}|\lambda_2| \leqslant \big( 2C\ln(n)^{D+3} \big)^\frac{1}{t} c\tilde{\rho} = \big(1+o(1) \big) c\tilde{\rho} \end{equation}
which finally ends the proof of \eqref{claim} and Theorem \ref{mainthm}.

\subsection{Organisation of the rest of the paper.}

The rest of this paper is mainly devoted to the proof of Propositions \ref{control1}-\ref{control2}. Both are inspired from \cite{bordenave2015}. 

\begin{enumerate}
\item Section \ref{section:prop_loc_eigvals} gives the proof of Proposition \ref{local_eigvals}. 
\item In Section \ref{tech_section}, we state a lemma on correlation functions in the multigraph $G$ that will be used in the proof of Propositions \ref{control1} and \ref{control2}. This section is essentially technical and the proof of \eqref{tech:equation1} is postponed to Appendix \ref{app:tech}.
\item In Section \ref{strat}, we develop the general strategy used to prove Proposition \ref{control1} which is an adaptation of the trace method. This leads to two subproblems, one purely combinatorial and one purely probabilistic. The combinatorial part (counting paths) is treated in Section \ref{sec:combi} and the probabilistic one (bounding expectations) in Section \ref{sec:analysis_f}.
\item Finally, the asymptotic analysis is done in Section \ref{sec:asymptotic_analysis}, thus concluding the proof of Proposition \ref{control1}. 
\item The exact same steps are adapted to the proof of Proposition \ref{control2} in the last section. 
\end{enumerate}

\section{Proof of Proposition \ref{local_eigvals}.}\label{section:prop_loc_eigvals}

The method for the bound \eqref{eqe:local_eigvals} is inspired from \cite{massoulie_rama} and was developped in \cite{bordenave_lelarge_massoulie} and \cite{bordenave2015}. The main steps are as follows: 
\begin{enumerate}
\item express $P^t$ as a weighted sum of matrix products involving the tangle-free centered matrices $\underline{P}^{(t)}$ and the tangled rest $R^{t,\ell}$, 
\item use this expression to make $P^t$ appear as a perturbation of a rank $1$ matrix, 
\item and finally use classical results from linear algebra to link the eigenvalues of $P^t$ with those of this perturbed matrix.
\end{enumerate}

\textbf{Notation}. If $\mathbf{e}$ is a head and $\mathbf{f}$ is a tail, then we will adopt the following notations: 
\begin{equation}\label{nota:A}A(\mathbf{e}, \mathbf{f})  =  \frac{ \mathbf{1}_{\sigma(\mathbf{e}) = \mathbf{f}}}{d_{\mathbf{e}}^+}  \qquad \text{and} \qquad \underline{A}(\mathbf{e}, \mathbf{f})  =  \frac{ \mathbf{1}_{\sigma(\mathbf{e}) = \mathbf{f}} - 1/M}{d_{\mathbf{e}}^+}.\end{equation}
With these notations, the matrix $P^t$ has the following expression:
$$P^t (i,j) = \sum_{\mathbf{p} \in \mathscr{P}_{i,j}^t} \prod_{s=1}^t A(\mathbf{e}_{s}, \mathbf{f}_{s})$$

\subsection{Telescoping products of real numbers.}

If $x_1, \dotsc,  x_t, y_1, \dotsc,  y_t$ are arbitrary complex numbers, we have the following ``telescopic product-sum" :
\begin{equation}\label{telescopic}
 \prod_{s=1}^t y_s=\prod_{s=1}^t x_s - \sum_{\ell=1}^t \prod_{s=1}^{\ell-1} y_s (x_\ell - y_\ell) \prod_{\ell+1}^t x_s   . 
\end{equation}

Recall Definitions \ref{def:CP} of $\underline{P}^t$ and Definition \ref{def:CTFP} of $\underline{P}^{(t)}$ on page \pageref{def:CTFP}. We apply \eqref{telescopic} to the matrix $P^{(t)}$, with $y_s=A(\mathbf{e}_s, \mathbf{f}_s)$ and $x_s = y_s - (Md_{\mathbf{e}_s}^+)^{-1}$. Note that the choice \eqref{def:t} for $t$ implies that $P^t = P^{(t)}$ with high probability due to Proposition \ref{lemme_pas_de_noeuds}. Hence, with high probability,

\begin{align}
P^t = P^{(t)}&= \sum_{\mathbf{p} \in \mathscr{T}^t(i,j)} \prod_{s=1}^t A(\mathbf{e}_{s}, \mathbf{f}_{s}) \\
&= \sum_{\mathbf{p} \in \mathscr{T}^t(i,j)} \prod_{s=1}^t \underline{A} (\mathbf{e}_{s}, \mathbf{f}_{s}) - \sum_{\mathbf{p} \in \mathscr{T}^t(i,j)}  \sum_{\ell=1}^t  \prod_{s=1}^{\ell-1} A(\mathbf{e}_{s}, \mathbf{f}_s ) \Big( \underline{A}(\mathbf{e}_\ell, \mathbf{f}_\ell) - A(\mathbf{e}_\ell, \mathbf{f}_\ell ) \Big)\prod_{\ell+1}^t \underline{A}(\mathbf{e}_{s}, \mathbf{f}_s).
\end{align}
By definition (see \eqref{nota:A}), we have $\underline{A}(\mathbf{e}_\ell, \mathbf{f}_\ell) - A(\mathbf{e}_\ell, \mathbf{f}_\ell ) = - (Md^+_{\mathbf{e}_\ell} )^{-1}$, so finally 
\begin{equation}\label{eq:Pt_aux}
P^t= \underline{P}^{(t)} -  \sum_{\ell=1}^t \frac{1}{M} \sum_{\mathbf{p} \in \mathscr{T}^t(i,j)}  \prod_{s=1}^{\ell-1} A(\mathbf{e}_{s}, \mathbf{f}_s )\frac{1}{d_{\mathbf{e}_\ell}^+} \prod_{s=\ell+1}^t \underline{A}(\mathbf{e}_{s}, \mathbf{f}_s ).
\end{equation}

\subsection{Gluing paths and gathering the remainders.}

We now decompose the set $\mathscr{T}^t(i,j)$ appearing in the sum in the right hand side of \eqref{eq:Pt_aux}. Recall that the out-degree distribution $\pi^-$ was defined in \eqref{exp_P} on page \pageref{exp_P}.

\begin{lem}  With high probability, 
\begin{equation}\label{eq:lemme_decomp}
P^t = \underline{P}^{(t)} - \sum_{\ell=1}^t P^{\ell-1} \mathbf{1}(\pi^-)^\top \underline{P}^{t-\ell} + \frac{1}{M} \sum_{\ell=1}^t  R^{t,\ell}.
\end{equation}
\end{lem}

\begin{proof}We start from \eqref{eq:Pt_aux}: our main task will be to reorganize the sum
\begin{equation}\label{eq:pt_aux2}\frac{1}{M}\sum_{\mathbf{p} \in \mathscr{T}_{i,j}^t}  \prod_{s=1}^{\ell-1} A(\mathbf{e}_{s}, \mathbf{f}_s )\frac{1}{d_{\mathbf{e}_\ell}^+} \prod_{s=\ell+1}^t \underline{A}(\mathbf{e}_{s}, \mathbf{f}_s ).\end{equation}

We have the following decomposition when $\ell<t$ (remind that the union over $g,h$ is taken over all pairs of vertices): 
\begin{equation}\label{eq:decomposition_Tt}
\mathscr{T}^t(i,j) = \bigcup_{g,h} \{(\mathbf{p}_1,  \mathbf{p}_2, \mathbf{p}_3 ): \mathbf{p}_1 \in \mathscr{T}^{\ell-1}(i,g), \mathbf{p}_2 \in \mathscr{T}^{1}(g,h)\} ,  \mathbf{p}_3 \in \mathscr{T}^{t-\ell}(h,j)\} \setminus \mathscr{R}^{t,\ell} (i,j). 
\end{equation}

Therefore, we have the following symbolic identity between sums: 
\begin{equation} \sum_{\mathscr{T}^t (i,j)} = \sum_{g} \sum_h \sum_{\mathscr{T}^{\ell-1}(i,g)} \sum_{\mathscr{T}^{1}(g,h)} \sum_{\mathscr{T}^{t-\ell}(h,j)} - \sum_{\mathscr{R}^{t,\ell} (i,j)}.\end{equation}

In the RHS, the sum over $\mathscr{R}^{t,\ell}$ will be exactly the $(i,j)$ entry of the matrix $R^{t,\ell}$ (see \eqref{def:tangled_rest}). Note that, if the path $\mathbf{p} = (\mathbf{e}_s, \mathbf{f}_s)_{s \leqslant t}$ can be written in the form $(\mathbf{p}_1, \mathbf{p}_2, \mathbf{p}_3)$ with $\mathbf{p}_1$ in $\mathscr{T}^{\ell - 1}$ and so on as in \eqref{eq:decomposition_Tt}, then 
\begin{equation}\prod_{s=1}^{\ell-1} A(\mathbf{e}_{s}, \mathbf{f}_s )\frac{1}{d_{\mathbf{e}_\ell}^+} \prod_{s=\ell+1}^t \underline{A}(\mathbf{e}_{s}, \mathbf{f}_s ) = \left( \prod_{s=1}^{\ell-1} A(\mathbf{e}^1_{s}, \mathbf{f}^1_s ) \right) \left( \frac{1}{d^+_{\mathbf{e}^2_1}} \right)\left( \prod_{s=1}^{t-\ell} A(\mathbf{e}^3_{s}, \mathbf{f}^3_s ) \right)\end{equation}
where we noted $\mathbf{p}_1 = (\mathbf{e}^1_1, \mathbf{f}^1_1, ..., \mathbf{e}^1_{\ell-1}, \mathbf{f}^1_{\ell - 1} )$ and so on. With the same notations, we plug this into the five sums found above: 
\begin{multline}\label{eq:compo2}
\frac{1}{M}R^{t,\ell}(i,j) + \frac{1}{M}\sum_{\mathbf{p} \in \mathscr{T}^t(i,j)}  \prod_{s=1}^{\ell-1} A(\mathbf{e}_{s}, \mathbf{f}_s )\frac{1}{d_{\mathbf{e}_\ell}^+} \prod_{s=\ell+1}^t \underline{A}(\mathbf{e}_{s}, \mathbf{f}_s ) =\\ \sum_{g,h} \left( \sum_{\mathbf{p}_1 \in \mathscr{T}^{\ell-1}(i,g)} \prod_{s=1}^{\ell-1} A(\mathbf{e}^1_{s}, \mathbf{f}^1_s )\right)\left( \sum_{\mathbf{p}_2 \in \mathscr{T}^{1}(g,h)}\frac{1}{Md_{\mathbf{e}^2_\ell}^+} \right) \left(\sum_{\mathbf{p}_3 \in \mathscr{T}^{t-\ell}(h,j)} \prod_{s=1}^{t-\ell} \underline{A}(\mathbf{e}^3_{s}, \mathbf{f}^3_s ) \right)
\end{multline}

This is a matrix product : the first and third parentheses are $P^{(\ell-1)}(i,h)$ and $\underline{P}^{(t-\ell)}(h, j)$. The term in the middle is equal to $\sum_{\mathbf{e} \in E^+(g), \mathbf{f} \in E^-(h)} \frac{1}{M d_g^+}$ which simplifies to $d_h^-/M = \pi^- (h)$. We define $X(g,h) =  \pi^-(h) $ --- note the useful identity $X = \mathbf{1}(\pi^- )^\top$. The RHS of \eqref{eq:compo2} then becomes 
\begin{equation} \sum_{g,h} P^{(\ell-1)}(i,h) X(g,h) \underline{P}^{(t-\ell)}(h, j) = \big( P^{(\ell-1)} X \underline{P}^{(t-\ell)} \big)(i,j)\end{equation}
and the whole expression \eqref{eq:pt_aux2} becomes equal to $\big( P^{(\ell-1)} X \underline{P}^{(t-\ell)} \big)(i,j) - M^{-1} R^{t,\ell}(i,j)$. Putting it back in \eqref{eq:Pt_aux}, we get 
\[
P^t = \underline{P}^{(t)} - \sum_{\ell=1}^t  P^{(\ell-1)} X \underline{P}^{(t-\ell)}  +  \frac{1}{M}\sum_{\ell=1}^t  R^{t,\ell}
\]
which is exactly the claim in the lemma because because (due to Proposition \ref{lemme_pas_de_noeuds}), with high probability we have $\underline{P}^{(t-\ell)} = \underline{P}^{t-\ell}$ and $P^\ell = P^{(\ell)}$.
\end{proof}

\subsection{Expressing $P$ as a perturbation of a rank 1 matrix.}

We first define two real vectors $x,y \in \mathbb{R}^n$ by  
\begin{equation}\label{def:x,y}
x= \mathbf{1},\qquad y= \frac{1}{n}(P^t)^\top x
\end{equation}
and we recall the definition of $K_t$ given in \eqref{def:K}: 
\[
K_t = \Vert \underline{P}^{(t)}\Vert  + \frac{1}{M} \sum_{\ell=1}^t \Vert R^{t,\ell}\Vert .
\]
Note the presence of the important $M^{-1}$ factor in the right. The following lemma is crucial: it quantifies the distance between the matrix $P^t$ and a rank-1 matrix, namely $xy^\top$. 
\begin{lem}

With high probability, 
\begin{equation}
\Vert P^t - x y^\top \Vert  \leqslant K_t.
\end{equation}
\end{lem}

\begin{proof}
Let $f$ be a vector such that $\langle f, \mathbf{1} \rangle = 0$; multiply \eqref{eq:lemme_decomp} to the left by $f^\top$ to get
\begin{equation}f^\top P^t  = f^\top \underline{P}^{(t)} - \sum_{\ell=1}^t f^\top P^{\ell-1} \mathbf{1}(\pi^-)^\top \underline{P}^{t-\ell} +  \frac{1}{M}\sum_{\ell=1}^t  f^\top R^{t,\ell}.\end{equation}
The matrix $P^{\ell-1}$ is a Markov matrix, therefore $P^{\ell-1}\mathbf{1}=\mathbf{1}$ and the product $f^\top P^{\ell-1}\mathbf{1}(\pi^-)^\top \underline{P}^{t-\ell}$ vanishes. We get the fundamental inequality
\begin{equation}\label{eq:fund_lemme_perturb}
\Vert (P^t)^\top f\Vert =\Vert f^\top P^t\Vert  \leqslant \left(\Vert \underline{P}^{(t)}\Vert  + \frac{1}{M} \sum_{\ell=1}^t \Vert R^{t,\ell}\Vert \right)\cdot\Vert f\Vert  = K_t \Vert f\Vert .
\end{equation}
Let us momentarily note $ Q = P^t - xy^\top$ so that 
\begin{equation}\label{perturbation}
P^t = xy^\top +Q.
\end{equation}
These choices imply the crucial following observation: $x^\top P^t =x^\top x y^\top + x^\top Q = ny^\top +x^\top Q$. But as $x^\top P^t = ny^\top$ we get $x^\top Q = 0$. Hence, $Q$ vanishes when multiplied on the left by $\mathbf{1}$. Let $v$ be any unit vector: there is a real number $\alpha$ and a vector $f$ with $\langle f, x \rangle = 0$ such that $v=f + \alpha x$. The triangle inequality implies $\Vert v^\top Q\Vert  \leqslant \Vert f^\top Q\Vert  + \alpha \Vert x^\top Q\Vert  = \Vert f^\top Q\Vert $, hence $\Vert Q\Vert \leqslant \sup \Vert f^\top Q\Vert /\Vert Q\Vert $, where the supremum is taken over all nonzero vectors $f$ such that $\langle f, x \rangle=0$. Moreover, as $\langle f, x \rangle = 0$ we have $f^\top P^t = f^\top x y^\top + f^\top Q = f^\top Q$. Putting all these observations together with \eqref{eq:fund_lemme_perturb} yelds the following: 
\begin{align*}
\Vert Q\Vert   &\leqslant \sup_{\langle f, \mathbf{1} \rangle = 0} \frac{\Vert f^\top Q \Vert }{\Vert f\Vert } \\
&\leqslant \sup_{\langle f, \mathbf{1} \rangle = 0} \frac{\Vert f^\top P^t \Vert }{\Vert f\Vert } \\
&\leqslant K_t
\end{align*}
which is exactly the claim in the lemma.
\end{proof}

\subsection{Classical algebra to link the eigenvalues of $P^t$ with those of $xy^\top$.}

The main ingredient for the proof of Proposition \ref{local_eigvals} will be the following basic algebraic lemma (see Appendix \ref{app:algebra} for a complete proof of this result).

\begin{lem}[eigenvalue perturbation for rank 1 matrices]\label{quanti_bauer_fike}
Let $H,M$ be two real $n \times n$ matrices, with $M$ diagonalizable with rank $1$. Let $x,y$ be two vectors such that $M = xy^\top$. Define $\mu = \langle x, y\rangle$. 
\begin{enumerate}
\item The eigenvalues of $M+H$ lie in the union of the two balls $B(0,\varepsilon)$ and $B(\mu, \varepsilon)$, with $\varepsilon = 2\Vert x\Vert ^2\Vert y\Vert ^2 \mu^{-2} \Vert H\Vert $.
\begin{center}
\includegraphics[scale=1]{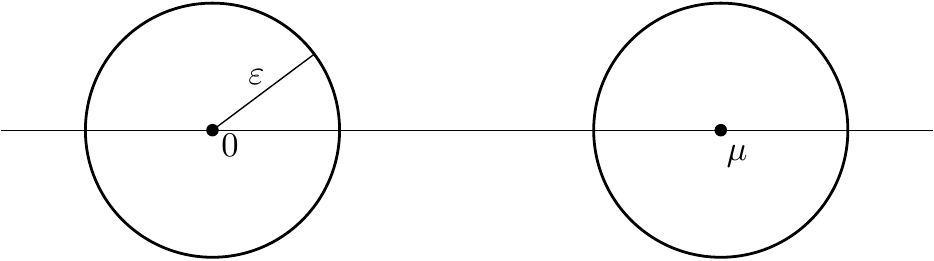}
\end{center}
\item If $B(0,\varepsilon) \cap B(\mu, \varepsilon) = \emptyset$, then there is exactly one eigenvalue of $M+H$ inside $B(\mu, \varepsilon)$ and all the other eigenvalues of $M+H$ are contained in $B(0, \varepsilon)$. 
\end{enumerate}
\end{lem}

\begin{proof}[Proof of Proposition \ref{local_eigvals}]
Let $x,y$ be as in \eqref{def:x,y}. We apply Lemma \ref{quanti_bauer_fike} to the matrix $P^t = xy^\top + Q$. First of all, note that $\mu=\langle x, y \rangle = \langle P^t x, x/n\rangle = \langle x, x/n \rangle = 1$.  All the eigenvalues of $P^t$ lie in the union of the two balls $B(0, \varepsilon)$ and $B(\langle x, y \rangle, \varepsilon)$ where $\varepsilon$ is smaller than 
$$\frac{2\Vert x\Vert ^2\Vert y\Vert ^2}{\langle x, y\rangle^2} K_t = 2\Vert x\Vert ^2\Vert y\Vert ^2 K_t. $$

We clearly have $\Vert x\Vert  = \sqrt{n}$. We should now have a control over the norm of $y$. Note that  $\Vert y\Vert ^2 = \sum_{i=1}^n (\pi_0^\top P^t )_i^2$ where $\pi_0$ is the uniform measure over the vertices of the graph (i.e.  $\pi_0(v) = 1/n$); hence, $\pi_0^\top P^t$ can be interpreted as the distribution of the Markov chain after $t$ steps on the directed graph $G$ when started from a uniform vertex. In particular, for every $i$ the term $({\pi_0}^\top P^t)^2_i$ is equal to $\mathbf{P}(X_t = Y_t = i)$ when $X,Y$ are two independant Markov chains, each one being independently started from a uniform vertex. We will note $\mathbf{P}', \mathbf{E}'$ the probability and expectation of the Markov chain conditionnally on $G$. The overall term $\Vert y\Vert ^2$ is thus equal to $\mathbf{P}'(X_t = Y_t)$. An elegant argument from \cite{caputo_salez_bordenave} (see section 4) shows that 
\begin{equation}\label{annealed}
\mathbb{P}(X_t = Y_t) = O\left( \frac{\ln(n)^2}{n}\right) = o \left(\frac{\ln(n)^3}{n} \right)
\end{equation}
where $\mathbb{P}$ denotes the so-called \emph{annealed probability}, that is the probability according to both the environment and the walk: $\mathbb{P}(X_t = Y_t) = \mathbf{E}[\mathbf{P}'(X_t = Y_t)]$. Using the Markov inequality with $\mathbf{P}$, \eqref{annealed} yelds that with high probability, 
$$\Vert y\Vert  = \sqrt{\mathbf{P}(X_t = Y_t)} \leqslant \sqrt{\frac{\ln(n)^3}{n}}.$$

Finally, with high probability we have $\Vert x\Vert ^2 \Vert y\Vert ^2 \leqslant \ln(n)^3$, hence $\varepsilon \leqslant 2\ln(n)^3 K_t$. 

We now use the second part of Lemma \ref{quanti_bauer_fike}. To this end, we have to check that the two balls $B(0,\varepsilon)$ and $B(1, \varepsilon)$ are disjoint, at least when $n$ is big. It is easy to see that 
\begin{equation}
\varepsilon = O\left(\ln(n)^{D+3} (c\tilde{\rho})^t \right),
\end{equation}
see for instance the short computations on page \pageref{borne:kt} leading to \eqref{borne:kt}. As a consequence of \eqref{H}, we also get $\tilde{\rho}<1$ so if $c$ is close enough to $1$, then $c\tilde{\rho}<1$ and $\varepsilon$ goes to $0$ as $n$ goes to infinity. The two balls $B(0,\varepsilon)$ and $B(1, \varepsilon)$ are thus disjoint. Using the second point of Lemma \ref{quanti_bauer_fike}, exactly one eigenvalue of $P^t$ is inside the ball $B(1, \varepsilon)$ and this eigenvalue is obviously $1$ because $P^t$ is a transition matrix. All the other eigenvalues, and in particular $\lambda_2$, are in $B(0,\varepsilon)$. 
\end{proof}

\section{Expectation of a product of centered random variables.}\label{tech_section}
 
In this technical section, we present a method for obtaining upper bounds on the expectations of a product having the form $\prod_{s \in I} (\mathbf{1}_{E_s} - \mathbf{P}(E_s)) $ when the events $E_s$ are nearly independant for most of them, and strongly dependent for a few ones. The general setting is the same as before. Such expectations will appear in the proofs of Propositions \ref{control1} and \ref{control2}. 

For the sake of clarity in the following sections, we need a definition of ``potential paths", i.e. collections of half-edges that are not paths, but who could give rise to real paths in the graph. Those are called \emph{proto-paths}:

\begin{defin}\label{def:protopath}
A \emph{proto-path} is a sequence $\mathfrak{p} = (\mathbf{e}_1, \mathbf{f}_1, \dotsc,  \mathbf{e}_N, \mathbf{f}_N)$ with $N$ an integer, such that for every $s$ in $\{1, \dotsc,  N\}$, $\mathbf{e}_s$ is a head and $\mathbf{f}_s$ is a tail. 
\end{defin}

There is no restriction whatsoever on the half-edges of a proto-path. Indeed, a proto-path is meant to be a path in the graph $G$, but it is not necessarily a path: some half-edge could appear twice of more in $\mathfrak{p}$, there is no vertex-consistency statement.

We are interested in computing different probabilistic quantities depending on $\mathfrak{p}$, the simplest of them being the probability of the event ``for all $s$, the head $\mathbf{e}_s$ is matched with the tail $\mathbf{f}_s$".

\bigskip

Fix some integer $p$ smaller than $N$. Recall that $\underline{A}$ and $A$ had been defined in \eqref{nota:A}. We define a function $F_p$ by
$$F_p(\mathfrak{p}) = \mathbf{E}\left[ \prod_{s=1}^{p} \underline{A}(\mathbf{e}_{s}, \mathbf{f}_{s})   \prod_{s=p+1}^{N} A(\mathbf{e}_{s}, \mathbf{f}_{s})\right]$$
Most of the times, the index $p$ will be dropped and we will just note $F$. We introduce several useful definitions and notations.
\begin{itemize}
\item We will note $B(\mathbf{e}, \mathbf{f}) = \mathbf{1}_{\sigma(\mathbf{e}) = \mathbf{f}} - 1/M$ and $B'(\mathbf{e}, \mathbf{f}) = \mathbf{1}_{\sigma(\mathbf{e})=\mathbf{f}}$. This implies $\underline{A}(\mathbf{e}, \mathbf{f}) = B(\mathbf{e}, \mathbf{f})/d_\mathbf{e}^+$. 
\item An edge of $\mathfrak{p}$ is a couple $(\mathbf{e}_s, \mathbf{f}_s)$ appearing in $\mathfrak{p}$.
\item $a$ is the number of distinct edges appearing in the proto-path $\mathfrak{p}$: 
$$a = \#\{(\mathbf{e}_s, \mathbf{f}_s): 1\leqslant s \leqslant N\}.$$ We will denote those edges by $y_1, \dotsc,  y_a$. 
\item For each $i \in \{1, \dotsc,  a\}$,  the weight $w_i$ of edge $y_i$ is the number of times edge $y_i$ is visited by the proto-path before $p$ and $w'_i$ is the number of times edge $y_i$ is visited after $p$: 
$$w_i = \#\{s \leqslant p: (\mathbf{e}_s, \mathbf{f}_s )= y_i\} \qquad \qquad  w'_i = \#\{s > p: (\mathbf{e}_s, \mathbf{f}_s )= y_i\}.$$
\item If $y_i= (\mathbf{e}, \mathbf{f})$, we will note $B(y_i)$ or $\underline{A}(y_i)$ instead of $B(\mathbf{e}, \mathbf{f})$ or $\underline{A}(\mathbf{e}, \mathbf{f})$. 
\item The \emph{weight} of the proto-path $\mathfrak{p}$ is 
$$\omega(\mathfrak{p}) = \prod_{s=1}^N \frac{1}{d_{\mathbf{e}_s}^+}.$$
\item 
Call an edge $y_i$ \emph{consistent} if both of its end-half-edges appear only once in the proto-path $\mathfrak{p}$. Call an edge \emph{simple} if its weight is $1$. If an edge is not consistent, it is \emph{inconsistent}. 
If the edge $(\mathbf{e}, \mathbf{f})$ is inconsistent, there is another edge $(\mathbf{e}', \mathbf{f}')$ in the proto-path such that $\{\mathbf{e}, \mathbf{f}\} \cap \{\mathbf{e}', \mathbf{f}' \} \neq \emptyset$. 
\end{itemize}

The main result of this section is the following theorem.

\begin{theorem}\label{tech}
Let $\mathfrak{p}$ be any proto-path of length $N \leqslant \sqrt{M}$, $p$ an integer smaller than $N$, and let $a_1$ be the number of simple, consistent edges of $\mathfrak{p}$, before $p$. Also, let $b$ be the number of inconsistent edges of $\mathfrak{p}$. Then, for every $c>1$, there is an integer $n_0$ such that if $n$ is larger than $n_0$, we have  
\begin{equation}\label{tech:equation1}
|F(\mathfrak{p})| \leqslant 24 \cdot \omega(\mathfrak{p})3^b \left( \frac{c}{M} \right)^{a} \left(\frac{N}{\sqrt{M}} \right)^{a_1}.
\end{equation}
\end{theorem}

The proof of Theorem \ref{tech} is essentially technical and is a mere adaptation of \cite{bordenave2015}. The complete proof can be found in Appendix \ref{app:tech}.

\section{General strategy and definitions for the proof of Proposition \ref{control1}.}\label{strat}

In this section, we study the quantity $\Vert \underline{P}^{(t)}\Vert $ for the choice of $t = \lfloor \alpha \log_\Delta (n) \rfloor$ as in \eqref{def:t}. For the rest of the paper, we set
\begin{equation}\label{def:m}
m = \left\lfloor \frac{\ln(n)}{\A \ln \ln(n)} \right\rfloor.
\end{equation}

\subsection{A simplified version of Proposition \ref{control1}.}

In order to prove Proposition \ref{control1}, we are going to prove the following lemma. 

\begin{lem}Fix $t$ as in \eqref{def:t} and $m$ as in \eqref{def:m}. Fix $c$ close to $1$ and $\tilde{\rho} = \rho \vee \delta^{-1}$. When $n$ is large enough, we have
\begin{equation}\label{objectif_aux}
\mathbf{E} \big[ \Vert \underline{P}^{(t)}\Vert ^{2m}\big] = o(1)n^3 (c\tilde{\rho})^{2tm}.
\end{equation}
\end{lem}

\begin{proof}[Proof of Proposition \ref{control1} using \eqref{objectif_aux}]
For any constant $D$, 
\begin{align}
\mathbf{P}(\Vert \underline{P}^{(t)}\Vert  > \ln(n)^{D}(c\tilde{\rho})^t) &\leqslant \frac{\mathbf{E} \big[ \Vert \underline{P}^{(t)}\Vert ^{2m}\big]}{(\ln(n)^D)^{2m}(c\tilde{\rho})^{2tm}}\\
&\leqslant \frac{o(1) n^3 }{(\ln(n)^D)^{2m}}.\\
\end{align}
Now, the choice of $D = \A \times 3/2$ yields $\ln(n)^{2Dm} \sim n^3$, and $\mathbf{P}(\Vert \underline{P}^{(t)}\Vert  > \ln(n)^{D}(c\tilde{\rho})^t) = o(1)$.
\end{proof}

Before going further in the application of the trace method, we gather here some basic consequences of the choice $m = \Theta(\ln (n)/ \ln \ln(n))$ as in \eqref{def:m}. They will be used several times in the forthcoming analysis without necessary reference. 

\begin{lem}\label{asymptotic_lemma}
For any $m = \Theta\left( \frac{\ln(n)}{\ln \ln (n)} \right)$ and any $c_n>0$ such that $\ln(c_n) = o(\ln \ln (n))$ we have $(c_n)^m = n^{o(1)}$. In particular, for any constant $c>0$ we have $c^m = n^{o(1)}$.

For any $A>0$ and $m = \frac{A\ln(n)}{\ln \ln (n)}$ and any $t_n = O(\ln(n))$ we have $(t_n)^m \leqslant n^{A + o(1)}$.

For any $A>0$ and $m = \frac{A\ln(n)}{\ln \ln (n)}$ and any $t_n = O(\ln(n)^B)$ we have $(t_n)^m \leqslant n^{AB + o(1)}.$
\end{lem}

\subsection{Use of the classical trace method.}\label{subsec:trace}

The proof of \eqref{objectif_aux} relies on the trace method. To somewhat lighten the notations, we will note $X = \underline{P}^{(t)}$ in this paragraph. From now on we will choose an even integer $r=2m$, so that $\Vert X \Vert^{2m} = \Vert X^* X \Vert^m $. As $X^* X$ is symmetric, we have 
\begin{align}
\Vert X\Vert^{2m} \leqslant \mathrm{tr}\big( (X^* X)^m \big) &= \sum_{i_1, \dotsc,  i_m} \prod_{s=1}^m (X^* X)_{i_s, i_{s+1}}\\
&= \sum_{i_1, i_2, \dots,  i_{2m}} \prod_{s=1}^m X_{i_{2s-1}, i_{2s}}X_{i_{2s+1}, i_{2s}}
\end{align}
where we adopted the cyclic notation $i_{m+1} = i_1$ in the first line and $i_{2m+1} = i_1$ in the second line. With $\underline{P}^{(t)}$ this becomes
\begin{equation}\Vert\underline{P}^{(t)}\Vert^{2m} \leqslant \sum_{i_1, \dots,  i_{2m}} \prod_{s=1}^m \underline{P}^{(t)} (i_{2s-1}, i_{2s}) \underline{P}^{(t)}(i_{2s+1}, i_{2s}). \end{equation}
Developping according to the definition of $\underline{P}^{(t)}$, we get
\begin{align}\prod_{s=1}^m X_{i_{2s-1}, i_{2s}}X^*_{i_{2s+1}, i_{2s}}&= \sum_{\mathbf{p}_1 \in \mathscr{T}^t (i_1, i_2)} \sum_{\mathbf{p}_2 \in \mathscr{T}^t (i_3, i_2)} \dotso\sum_{\mathbf{p}_{2m} \in \mathscr{T}^t (i_1, i_{2m} )} \prod_{i=1}^{2m} \prod_{s=1}^t \underline{A}(\mathbf{e}_{i,s}, \mathbf{f}_{i,s})    
\end{align}
where we noted $\mathbf{p}_i = (\mathbf{e}_{i,s}, \mathbf{f}_{i,s})_{s \leqslant t}$ the $i$-th path in the ``path of paths" $\mathbf{p} = (\mathbf{p}_1, \dotsc,  \mathbf{p}_{2m})$ (remember the concatenation notation \ref{nota:conca}). We define $\mathscr{C}_m$ as the set of  ``paths of paths" corresponding to the sum, that is $2m$-tuples $(\mathbf{p}_1, \dotsc,  \mathbf{p}_{2m})$ such that $\mathbf{p}_1$ and $\mathbf{p}_2$ have the same endpoint, $\mathbf{p}_2$ and $\mathbf{p}_3$ have the same beginning point, and so on. For the following analysis, it will be easier to ``reverse" all odd paths in $\mathbf{p}$, leading to the following central definition:

\begin{figure}[H]\centering
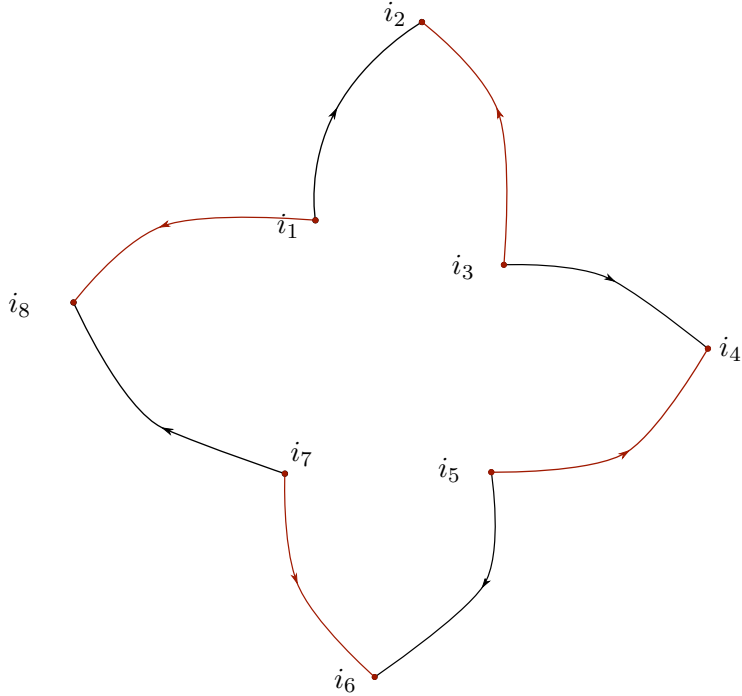
\caption{A path in $\mathscr{C}_{4}$. The red paths are the "odd" paths, corresponding to "reversed tangle-free paths". The black ones are "even" paths. }
\end{figure}

\begin{defin}$\mathscr{C}_m$ is the set of $2m$-tuples $\mathbf{p} = (\mathbf{p}_1, \dotsc,  \mathbf{p}_{2m})$ such that  
\begin{itemize}
\item for every $i$, the path $\mathbf{p}_{2i-1}$ is in $\mathscr{T}^t$ and the ``reversed path" 
$$\bar{\mathbf{p}}_{2i} = (\mathbf{f}_{2i, t}, \mathbf{e}_{2i, t}, \dotsc,  \mathbf{f}_{2i, 1}, \mathbf{e}_{2i, 1} )$$
is in $\mathscr{T}^t$. 
\item For every $i$, the last half-edge of $\mathbf{p}_i$ and the first half-edge of $\mathbf{p}_{i+1}$ are  attached to the same vertex (boundary condition).
\end{itemize}
\end{defin}
Note that there is a little lack of consistency with our convention that $\mathbf{e}$ denotes heads and $\mathbf{f}$ denotes tails, for in this case $\mathbf{e}_{2i, s}$ denotes a tail and $\mathbf{f}_{2i, s}$ denotes a head. For every element $\mathbf{p} \in \mathscr{C}_m$, we note
\begin{equation}\label{def:f}f(\mathbf{p}) = \mathbf{E}\left[ \prod_{i=1}^{m} \prod_{s=1}^t \underline{A}(\mathbf{e}_{2i-1,s}, \mathbf{f}_{2i-1,s})\prod_{s=1}^t \underline{A}(\mathbf{f}_{2i,s}, \mathbf{e}_{2i,s})\right].\end{equation}

We have obtained the following fundamental inequality: 
\begin{equation}\label{FUNDINEQ}\mathbf{E}\left[ \Vert\underline{A}^{(t)}\Vert^{2m}  \right] \leqslant \sum_{\mathbf{p} \in \mathscr{C}_m} |f(\mathbf{p})|.
\end{equation}

In the last expression, the probabilistic part, which is contained in the function $f$, is entirely decoupled from the combinatoric part, which is contained in the set $\mathscr{C}_m$. Both parts will be separately treated in the forthcoming analysis.

\subsection{Geometry of paths in $\mathscr{C}_m$.}\label{geometry}

We now introduce some definitions that will be commonly used in the sequel. Let $\mathbf{p}$ be any element in $\mathscr{C}_m$. It induces a walk on the vertices of the graph $G$. We will note $V(\mathbf{p})$ (or generally $V$ if there is no ambiguity) the set of all visited vertices, and $v = v(\mathbf{p}) = \# V(\mathbf{p})$. Any $\mathbf{p} \in \mathscr{C}_m$ is composed of $2m$ path of length $t$, hence we have $v \leqslant 2tm$. 

\begin{defin}\label{def:edges}
We had already defined an \emph{edge of $\mathbf{p}$} as any pair of a head followed by a tail appearing in one of the $\mathbf{p}_i$'s (for example $(\mathbf{e}_{1,s}, \mathbf{f}_{1,s} )$ or $(\mathbf{f}_{2, s}, \mathbf{e}_{2,s})$) A \emph{graph edge} is the corresponding (oriented) edge between vertices. 
\end{defin}

\begin{exemple}
Let $(\mathbf{e}, \mathbf{f} )$ be an edge of $\mathbf{p}$, with $\mathbf{e}$ a head and $\mathbf{f}$ a tail. If $\mathbf{e}$ is attached to vertex $u$ and $\mathbf{f}$ to vertex $u'$, then the corresponding graph edge will be $(u,u')$. Thus, each graph-edge $(u,v)$ corresponds to at most $d_u^+ d_{u'}^-$ distinct edges. 
\end{exemple}

We will note $E(\mathbf{p})$ the set of edges. The total number of distinct edges will be noted $a = a(\mathbf{p}) = \#E(\mathbf{p})$.  Any $\mathbf{p} \in \mathscr{C}_m$ induces an oriented multigraph on the set of vertices $V(\mathbf{p})$: its edges are just the graph edges of $\mathbf{p}$, counted with multiplicities. Let us call $\vec{G}(\mathbf{p})$ this oriented multigraph; the corresponding \emph{unoriented multigraph} $G(\mathbf{p})$ is connected. We will note $\chi = \chi (\mathbf{p}) = a-v+1 $ the \textbf{tree excess} of $G(\mathbf{p})$. This quantity will be used many times in the sequel.

\section{Combinatorics of \texorpdfstring{$\mathscr{C}_m$}{}.}\label{sec:combi}

We split $\mathscr{C}_m$ in various disjoints subsets, taking into account the number of visited vertices and also the number of edges. The counting argument is inspired from \cite{bordenave2015} which itself stems from the seminal paper \cite{komlos}. 

\begin{defin}Let $a,v$ be integers and let $\mathbf{i}=(i_1, \dotsc,  i_v)$ be a $v$-tuple of vertices. We define
$$ X^{v,a}_m(\mathbf{i}) = X_m^{v,a}(i_1, \dotsc,  i_v)$$
as the set of all the elements in $\mathscr{C}_m $ whose vertex set is precisely $(i_1, ..., i_v)$ (visited in this order) and who have $a$ edges. 
\end{defin}

The aim of this section is to prove the following result on the number of elements in $X^{v,a}_m(\mathbf{i})$. 

\begin{prop}\label{prop:combidefinitive}
Fix $v,\mathbf{i}$ and $a$. Recall that $\chi = a-v+1$. Then, there is a constant $C>0$ and an integer $n_1$ such that for every $n \geqslant n_1$, we have
\begin{equation}\label{combidefinitive}
\# X^{v,a}_m(\mathbf{i}) \leqslant  \left( \prod_{i \in \mathbf{i}} d_{i}^+ d_i^- \right) C^\chi n^{\frac{25}{\A} + \frac{17}{\A}  \chi} .
\end{equation}
\end{prop}

The core tool for the proof of \eqref{combidefinitive} will be a simple partition of the elements of $\#X^{v,a}_m(\mathbf{i})$ with the following notion of equivalence:

\begin{defin}
Let $\mathbf{p}$ and $\mathbf{p}'$ be two elements in $\mathscr{C}_m$; we note $\mathbf{e}_{i,s}, \mathbf{f}_{i,s}$ the half-edges of $\mathbf{p}$ and $\mathbf{e}'_{i,s}, \mathbf{f}'_{i,s}$ those of $\mathbf{p}'$. They are said \textbf{equivalent} if 
\begin{itemize}
\item they both belong to $X^{a,v}_m(\mathbf{i})$ and they visit the same vertices at the same time, 
\item for every vertex $u \in \mathbf{i}$, there are two permutations $\sigma_u  \in \mathfrak{S}_{d^+_u}$ and $\tau_u \in \mathfrak{S}_{d_u^-}$ such that for every $i$ and $s$, if $\mathbf{e}_{i,s}$ is a head attached to $u$ and $\mathbf{f}_{i,s}$ a tail attached to $u$, then
$$\mathbf{e}_{i,s} = \sigma_u (\mathbf{e}'_{i,s}) \quad \text{and} \quad \mathbf{f}_{i,s} = \tau_u (\mathbf{f}'_{i,s}). $$
\end{itemize}

In other words, two elements of $\mathscr{C}_m$ are equivalent if they only differ by a permutation of their half-edges. 
\end{defin}

The proof is organized as follows:

\begin{itemize}
\item In \ref{Subsection:cardinal}, we prove an upper bound for the number of elements within each equivalence class. 
\item In \ref{Subsection:number}, we prove an upper bound for the number of equivalence classes. 
\item In \ref{Subsection:total} we prove Proposition \ref{prop:combidefinitive}. 
\end{itemize}

\subsection{Cardinal of equivalence classes}\label{Subsection:cardinal}

Let $\mathbf{p}$ be an element of $X^{v,a}_m (\mathbf{i})$. How many elements of $\mathscr{C}_m$ are equivalent to $\mathbf{p}$ ? The vertices are fixed so there is no choice from this part. We have to chose the half-edges. If there is exactly one tail and one head attached to each of these vertices, we would have $d^+_{i_1}$ choices for the first head, then $d^-_{i_2}$ for the first tail, and so on until the last head with $d^+_{i_v}$ choices and the last tail with $d_1^-$ choices. Thus, we have at most $ \prod_{i \in \mathbf{i}} d^+_{i}d^-_{i}$ paths equivalent with $\mathbf{p}$ in this case. In the general case, there are some vertices with \emph{more} than one half-edge visited by $\mathbf{p}$ attached to these vertices.

\begin{lem}\label{lem:cardinal_equiv_classes}
Let $\mathbf{p}$ be in $X^{v,a}_m (\mathbf{i})$. Note $\alpha_s$ the number of heads visited by $\mathbf{p}$ attached to the vertex $i_s$, and let $\beta_s$ be the same with tails. Then, we have at most
\begin{equation}\label{bound_equiv_card}
C^\chi \prod_{i \in \mathbf{i}} d^+_{i}d^-_{i}
\end{equation}
elements in $\mathscr{C}_m$ equivalents to $\mathbf{p}$, where $C>0$ is a constant.
\end{lem}

\begin{figure}[H]\centering
\includegraphics[scale=1]{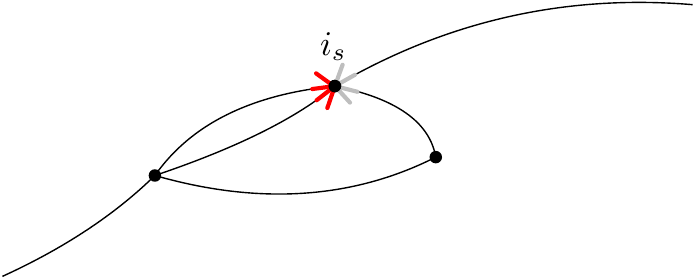}
\caption{Here, we have $d_{i} = 4$, but $\alpha_i = 2$ and $\beta_i = 2$.}
\end{figure}

In the proof we will make use of the Pocchammer symbol: if $a$ is a real number and $k$ and integer, then $(a)_k = a(a-1)...(a-k+1)$. 

\begin{proof}Fix $\mathbf{p}$. When choosing equivalent elements to $\mathbf{p}$, we have \emph{at most}
$$ \prod_{i \in \mathbf{i}} (d_i^+)_{\alpha_i} (d_{i}^-)_{\beta_i} = \prod_{i \in \mathbf{i}}d_i^+d_i^- \prod_{i \in \mathbf{i}} (d_i^+-1)_{\alpha_i-1} (d^-_{i}-1)_{\beta_i-1} $$
choices, with the convention that a product over an empty set is equal to $1$. We also have $(d^+_{i}-1)_{\alpha_i-1} \leqslant(\Delta-1)^{\alpha_i - 1}$ and $(d_{i}^--1)_{\beta_i-1} \leqslant(\Delta-1)^{\beta_i - 1}$, so if we set $K^+_t = \# \{i \in \mathbf{i}: \alpha_i = t\}$ and $K^-_t = \# \{i \in \mathbf{i}: \beta_i = t\}$ we have 
\begin{align*}
\prod_{i \in \mathbf{i}} (d^+_{i}-1)_{\alpha_i-1} (d^-_{i}-1)_{\beta_i-1}&= \prod_{t \geqslant 1}\prod_{i \in K^+_t} (\Delta-1)^{t-1}\prod_{i \in K^-_t} (\Delta-1)^{t-1}\\
&\leqslant \prod_{t \geqslant 1} (\Delta - 1)^{(t-1)(K^+_t + K^-_t)}\\
&\leqslant (\Delta-1)^{\sum_{t\geqslant 1} (t-1)K^+_t + (t-1)K^-_t}.
\end{align*}

Counting edges going out of every vertex yields $\sum_t t K^+_t = a$ and counting vertices according to the number of edges going out this vertex gives $\sum_t K^+_t = v$ (the same holds for $K^-_t$), so we get $\sum_{t\geqslant 1}(t-1)K^+t = \sum_{t\geqslant 1}(t-1)K^-_t = a-v$, and
$$\prod_{i \in \mathbf{i}} (d_{i}^+-1)_{\alpha_i-1} (d_{i}^--1)_{\beta_i-1}  \leqslant (\Delta-1)^{2(a-v)} \leqslant C^\chi$$
where $C = (\Delta-1)^2$, thus closing the proof of \eqref{bound_equiv_card}.
\end{proof}

\subsection{Number of equivalence classes.}\label{Subsection:number}

Now, we count the number of equivalence classes in $X^{a,v}_m(\mathbf{i})$. The result of this paragraph is:

\begin{lem}\label{prop:counting}
There is an integer $n_1$ such that for every $n \geqslant n_1$, the total number of equivalence classes of paths in $\mathscr{C}_m$ visiting vertices $(i_1, \dotsc, i_v)$ and having $a$ edges is bounded by
\begin{equation}\label{counting}
n^{\frac{25}{\A}+\frac{17}{\A}  \chi}.
\end{equation}
\end{lem}

We now prove this lemma. The explored vertices are $\mathbf{i} = (i_1, ..., i_v)$, in this order. Recall Notation \ref{nota:deg}: half-edges are noted $(u,i,\varepsilon)$ with $\varepsilon \in \{-,+\}$ and $i \leqslant d^\varepsilon_u$. We first describe a coding pattern for the equivalence classes (in Paragraphs \ref{subsub:choice}-\ref{subsec:superfluous}) and then prove \eqref{counting} in Paragraph \ref{proof:counting}.

\subsubsection{Choice of the path.}\label{subsub:choice}

In any equivalence class, we choose a $\mathbf{p}$ visiting heads and tails in the ``alternating lexicographic order", that is 

\begin{itemize}
\item  vertex $u$ before vertex $v>u$, 
\item head $(u,s,+)$ before head $(u,s',+)$ with $s'>s$ and the same for tails, 
\item and such that 
\begin{itemize}
\item if $i$ is even, $\mathbf{e}_{i,s}$ is a head and $\mathbf{f}_{i,s}$ is a tail, 
\item if $i$ is odd, $\mathbf{e}_{i,s}$ is a tail and $\mathbf{f}_{i,s}$ is a head. 
\end{itemize}
\end{itemize}

The chosen $\mathbf{p}$ will be called the \emph{representative path} of the class $X^{a,v}_m(\mathbf{i})$. We will note $\mathbf{p} = (\mathbf{e}_{i,s}, \mathbf{f}_{i,s})_{i,s}$. The edge $(\mathbf{e}_{i,s}, \mathbf{f}_{i,s})$ will be noted $y_{i,s}$. We see $\mathbf{p}$ as a walk on the vertices $\mathbf{i}$. The index $(i,s)$ in $\mathbf{p}$ is seen as a time parameter. At time $(i,s)$, the walk is located on the vertex $u$ attached to $\mathbf{e}_{i,s}$, and then moves along the edge $y_{i,s}$ to go to the vertex $v$ to which is attached $\mathbf{f}_{i,s}$. 

\subsubsection{Creating the spanning tree.}

We build a marked graph $T$ on the vertex-set $\mathbf{i}$ by adding the graph-edge\footnote{Recall notations from section \ref{strat}. \emph{Edges} are pairs of half-edges seen in $\mathbf{p}$ while \emph{graph-edges} are pairs of vertices corresponding to some edge.} $(u,v)$ with mark $y_{i,s}$ when vertex $v$ is explored for the first time at time $(i,s)$. The edge $y_{i,s}$ is called a tree edge. The (unmarked) graph $T$ is clearly a tree on the vertex set $\mathbf{i}$. The mark over every edge of $T$ keeps track of the half-edges used to discover for the first time the endvertex of this edge.

Suppose that we are at time $(i,t)$ and the edge we are currently exploring is $y_{i,t} = (\mathbf{e}_{i,t}, \mathbf{f}_{i,t} )$ and leads to vertex $u$. If the vertex $u$ is already part of the tree $T$ then the edge $y_{i,t}$ is called an excess edge and time $(i,t)$ is called a \emph{cycling time} for obvious reasons. 

Due to the very specific structure of $\mathbf{p}$ (a sequence of tangle-free paths with boundary conditions), such times can easily be understood: either they count as cycling times inside a tangle-free path $\mathbf{p}_i$ (which can happen only once for every $i \leqslant 2m$), or they are cycling times between different $\mathbf{p}_i$. 

\bigskip

We are now going to give an encoding of $\mathbf{p}$: the idea is roughly that if there were no cycling times, $\mathbf{p}$ would perfectly be uncoded without needing anything, due to the choice of lexicographic ordering of half-edges. Therefore, by noting the different cycling times and giving them a minimal amount of information on how to decode them, we will be able to explore the non-cycling times as usual and create the tree $T$ in the process, and when stepping on a cycling time we will use all the previous information (mainly, $T$) and the mark to determine where to go.

\subsubsection{Short cycling times.}
 Each sub-path $\mathbf{p}_i$ is tangle-free. Let $r_i$ denotes the first time when $\mathbf{f}_{i,r_i}$ is attached to a vertex \emph{already visited by $\mathbf{p}_i$}: this time is called a \emph{short cycling time}. If this cycling time does not exist, we artificially set it to be the symbol $\otimes$; thus, $r_i  =\otimes$ means that $\mathbf{p}_i$ has no cycles. Also, let $\sigma_i$ be the first time when the path left this vertex after its first visit in $\mathbf{p}_i$. If $r_i = \otimes$, we set $\sigma_i=0$. If $r_i \neq \otimes$, the cycle $\mathfrak{C}_i$ in $\mathbf{p}_i$ is precisely given by the edges $\mathfrak{C}_i = \{y_{i, \sigma_i}, y_{i, \sigma_i +1}, ..., y_{i, r_i}\}$ and it might be visited more than once. Note $\ell_i$ the ``total time spent in the loop", that is the number of times $(i,t)$ such that $y_{i,t}$ is in $\mathfrak{C}_i$. Then, the knowledge of 
\begin{enumerate}
\item the cycling time $(i,r_i)$
\item the half-edges $\mathbf{e}_{i,r_i}$ and $\mathbf{f}_{i,r_i}$
\item the total time spent ``in the loop" $\ell_i$ and the half-edge $\mathbf{e}_{i,\tau_i}$ where we're leaving the cycle,
\item the next vertex $u_i$ where we will leave the edges of the tree $T$,
\end{enumerate}
are sufficient to reconstruct the path $\mathbf{p}_i$ up to the visit of vertex $u_i$. Note that in the second step, if $\mathbf{e}_{i,r_i} = (v_{r_i}, j_{r_i}, \pm)$, the vertex $v_{r_i}$ is already known, and whether $\mathbf{e}_{i,r_i}$ is a head or a tail is also known according to the parity of $i$, so we only need to know $j_{r_i}$. Thus, if $r_i \neq \otimes$, the mark for the $i$-th short cycling time $(i,r_i)$ will be 
\begin{equation}\label{mark_SCT}
(j_{i,r_i}, \mathbf{f}_{i,r_i}, \ell_i, \mathbf{e}_{i, \tau_i}, u_i)
\end{equation}
and if $r_i = \otimes$ this mark is set to be $\emptyset$.

We have at most one short cycling time per $\mathbf{p}_i$ which is a path of length $t$. Fix $i$: if there is no cycling time, $r_i=\emptyset$ (one possibility). If there is a cycling time, there are $t$ choices for its location. Once this time has been chosen, there are at most $\Delta (\Delta v) t (\Delta v) v = \Delta^3 v^3 t$ possible marks as \eqref{mark_SCT} for the short cycling time. This bound is extremely crude but will be sufficient for our purpose. Thus, the total number of possible marks for the short cycling time of $\mathbf{p}_i$ is $1+\Delta^3 v^3 t$.

\begin{remarque}\label{remark_uncoding}
Suppose we are decoding a short cycling time. The last part of the mark is $u_i$; as $T$ is a tree, this means that the path to follow is perfectly known up to $u_i$. Arriving at $u_i$ at a certain time, say $(i', t')$, we know that we are going to leave the tree $T$ constructed so far, and this can lead to two situations. 
\begin{itemize}
\item  The time $(i', t')$ can be another cycling time. In this case, the procedure defined on this paragraph (if the cycling time is short) or the next paragraph (if it is long) will tell us where to go next. 
\item The time $(i', t')$ is not a cycling time. If we note $v$ the next vertex after $u_i$, this means that the edge $(u_i, v)$ is not in the tree $T$ constructed so far, and that $v$ is not already discovered.  Therefore, the path is just going to explore this new vertex $v$ and we are going to add the edge $(u_i, v)$ to $T$. Note that the use of the lexicographic order clearly tells us which half-edges to use. 
\end{itemize}
\end{remarque}

\subsubsection{Long cycling times.}
There are also cycling times that are not ``short cycling times": basically, it is when a path $\mathbf{p}_i$ collides with another path $\mathbf{p}_j$ with $j<i$. More precisely, let $(i,t)$ be a cycling time  leading to the (already known) vertex $u$. If $u$ is not one of the vertices discovered by $\mathbf{p}_i$, then $(i,t)$ is called a \emph{long cycling time}: in this case, $u$ had already been visited by some $\mathbf{p}_j$ with $j<i$. Here again, we are going to mark long cycling times with different items, so they could be easily deduced from the marks. When arriving at a long cycling time, we need to know: 
\begin{enumerate}
\item the head $\mathbf{e}_{i,t}$ and the tail $\mathbf{f}_{i,t}$,
\item the next vertex $u_i$ where we will leave the edges of the tree $T$ (no extra information is needed: see Remark \ref{remark_uncoding}). 
\end{enumerate}

The mark obtained has the form
\begin{equation}\label{LCM}(j_{i,t}, \mathbf{f}_{i,t}, u_i ).
\end{equation}

For every long cycling time, there are at most $\Delta^2 v^2$ marks like \eqref{LCM}. 

\begin{figure}[H]
\centering
\includegraphics[scale=0.6]{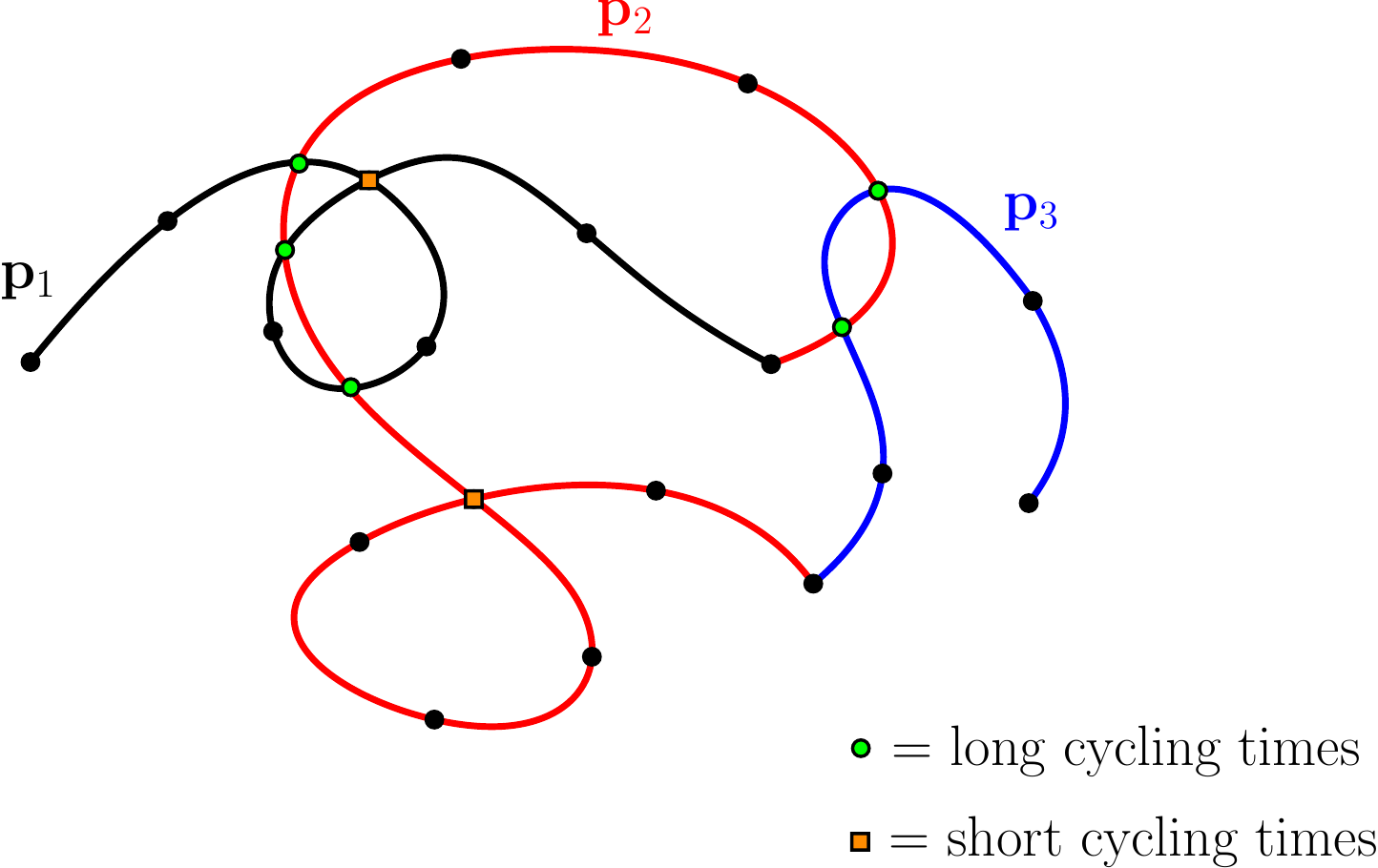}
\caption{Some examples of vertices generating long and short cycling times. }
\end{figure}

\subsubsection{Superfluous times.}\label{subsec:superfluous}There is another kind of cycling times we have not yet coded: those times are the cycling times ``embedded in the loop" of a short cycling time, that is all the times \emph{except for the first one} when $(i,t)$ when $\mathbf{f}_{i,t}$ is attached to a vertex already visited by $\mathbf{p}_i$. Those times need no special treatment as they are decoded with the mark of the short cycling time associated with $i$. For this reason, they will be called \emph{superfluous cycling times} and play no role in the coding procedure. 

\begin{figure}[H]
\centering
\includegraphics[scale=0.6]{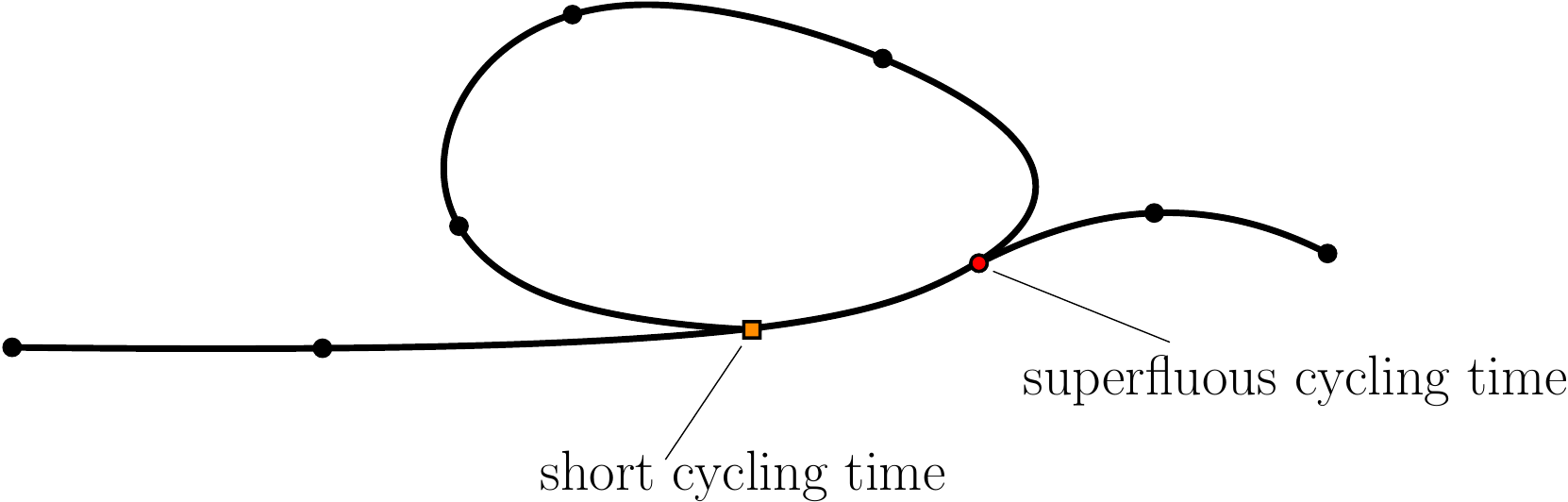}
\end{figure}

\subsubsection{Proof of Lemma \ref{prop:counting}.}\label{proof:counting}
We now gather the number of different types of marks to get a bound on the number of equivalence classes in $\mathscr{C}_m$. Recall the definitions given in Subsection \ref{geometry} (page \pageref{geometry}) and the difference between \emph{edges} of $\mathbf{p}$ and \emph{graph-edges} of $\mathbf{p}$. Consider the undirected multi-graph spanned by the \emph{unoriented} graph-edges of $\mathbf{p}$ on vertices $\mathbf{i} = (i_1, ..., i_v)$. This graph is connected. Its total number of edges is at most $a$ (if no edge is visited two times in opposite directions\footnote{Observe that it is also at least $a/2$ if all edges are visited twice, in opposite directions. This will not be used in the proof.}. Therefore, there are at most $\chi := a-v+1$ excess edges. For each $i \leqslant 2m$, there are at most $\chi$ cycling times, \emph{a fortiori} there are at most $\chi$ long cycling times. Therefore, we have at most $t^{2m\chi}$ choices for the positions for the long cycling times and we have already seen that we have $t^{2m}$ choices for the positions of the short cycling times. Now the total count amounts to $t^{2m(\chi +1)}  ((\Delta v)^2)^{2m\chi} ((\Delta v)^3 t )^{2m}$ possible codings. Organizing termes leads to $t^{2m\chi + 4m}(\Delta v)^{4m\chi + 6m}$ which (using $v \leqslant 2tm$) is bounded by 
$$(2\Delta tm)^{8m\chi + 12m}.$$

Using the asymptotic properties exposed in Lemma \ref{asymptotic_lemma}, this expression can be simplified. Note for example that there is an integer $n_1$ only depending on $\Delta$ such that for every $n \geqslant n_1$, we have $(2\Delta tm)^{8m}  \leqslant  n^\frac{17}{\A}$, and the same argument gives $(2\Delta tm)^{12m} \leqslant n^\frac{24}{\A}$. Hence, when $n$ is larger than $n_1$, we have 
$$(2\Delta tm)^{8m\chi +12m} \leqslant  n^{\frac{25}{\A}+\frac{17}{\A}  \chi}$$
which ends the proof of \eqref{counting}

\subsection{Proof of Proposition \ref{prop:combidefinitive}} \label{Subsection:total} Let us note $\mathscr{N}(a,v,\mathbf{i})$ the set of equivalence classes $\mathscr{E}$ inside $X^{v,a}_m(\mathbf{i})$. We have 
$$\# X^{v,a}_m(\mathbf{i}) = \sum_{\mathscr{E} \in \mathscr{N}(a,v,\mathbf{i})} \#\mathscr{E}.$$
Using Lemmas \ref{lem:cardinal_equiv_classes} and \ref{prop:counting}, when $n$ is larger than $n_1$ we get 
\begin{align*}
\# X^{v,a}_m(\mathbf{i}) &\leqslant \sum_{\mathscr{E} \in \mathscr{N}(a,v,\mathbf{i})} C^\chi \prod_{i \in \mathbf{i}} d^+_{i}d^-_{i} \leqslant n^{\frac{25}{\A}+\frac{17}{\A}  \chi} C^\chi \prod_{i \in \mathbf{i}} d^+_{i}d^-_{i} 
\end{align*}
which is the conclusion of Proposition \ref{prop:combidefinitive}.

\section{Upper bound for \textit{f}.}\label{sec:analysis_f}

Our aim in the next paragraphs will be to bound $f(\mathbf{p})$ (which was defined in \eqref{def:f}) with an expression that depends on the variables $a,v,m,t,\mathbf{i}$. We recall a definition from Section \ref{tech_section}: if $\mathfrak{p}$ is a proto-path of length $N$, then 
$$\omega(\mathfrak{p}) = \prod_{s=1}^N \frac{1}{d^+_{\mathbf{e}_s}}.$$ 
Every path is itself a proto-path, so we can extend the definition of the weight $\omega$ in a natural way to $\mathbf{p}\in\mathscr{C}_m$:
$$\omega(\mathbf{p}) = \prod_{i=1}^{m}\omega(\mathbf{p}_i)\omega(\bar{\mathbf{p}}_i).$$

The result of this section is the following proposition which gives upper bounds for $|f(\mathbf{p})|$ depending on $a,\chi$. 

\begin{prop}\label{regions}
Let $\mathbf{p}$ be any path with $v$ vertices and $a$ edges. Note $\chi = a-v+1$. Then, there is a constant $C>0$ and an integer $n_2$ such that for every $n \geqslant n_2$, we have the following inequalities:
\begin{itemize}
\item If $\chi \geqslant v- tm - 1$, then 
\[
|f(\mathbf{p})|\leqslant  \frac{n^{o(1)}}{\delta^{2(tm-v)}} \prod_{i \in \mathbf{i}} \left(\frac{1}{d_i^+}\right)^2  \left( \frac{C}{M} \right)^{\chi} \left( \frac{c}{M} \right)^{v-1} .
\]
\item Else $\chi \leqslant v-tm-1$ and we have
\[
|f(\mathbf{p})|\leqslant  \frac{n^{o(1)}}{\delta^{2(tm-v)}} \prod_{i \in \mathbf{i}} \left(\frac{1}{d_i^+}\right)^2  \left( \frac{C}{M} \right)^{\chi} \left( \frac{c}{M} \right)^{v-1} \left( \frac{6tm}{\sqrt{M}}\right)^{2(v-tm-1-\chi)}.
\]

\end{itemize}
\end{prop}

The rest of the section is devoted to the proof of this proposition.

\subsection{Expressing the weight $\omega(\mathbf{p})$ with graph-dependant variables.}

Fix $\mathbf{p}$ in $X^{a,v}_m(\mathbf{i})$. For every $s>0$, let $V_s$ be the set of vertices that are visited by $\mathbf{p}$ exactly $s$ times and note $v_s = \#V_s$, so that $\sum_{s >0} v_s = v$ and $\sum_{s>0} sv_s = 2tm$.  A vertex is called a \emph{boundary vertex} if it is the endpoint or beginning point of a sub-path of $\mathbf{p}$: if $\mathbf{p} = (\mathbf{p}_i)_{i \leqslant 2m}$ (with each of the $\mathbf{p}_i$'s being tangle-free paths of length $t$) then boundary vertices are those attached to half-edges $\mathbf{e}_{i,0}$ or $\mathbf{f}_{i,t}$. We also recall that $a_1$ is the number of consistent edges of $\mathbf{p}$ visited exactly once: this quantity was introduced in Section \ref{tech_section} and appears in the statement of Theorem \ref{tech}. Also, recall that $b$ is the number of inconsistent edges.

\begin{lem}\label{lem:weightp}
There is a constant $C>0$ such that for every $\mathbf{p} \in X^{a,v}_m(\mathbf{i})$ we have 
\begin{equation}\label{weightp}
\omega(\mathbf{p}) \leqslant n^{o(1)} \prod_{i \in \mathbf{i}} \left(\frac{1}{d_i^+}\right)^2  \frac{C^{\chi + a_1}}{\delta^{2(tm-v)}}.
\end{equation}
\end{lem}

\begin{proof}

As a consequence of the definition of the sets $V_s$, we have 
$$\omega(\mathbf{p}) = \prod_{s>0} \prod_{i \in V_s} \left(\frac{1}{d_i^+}\right)^s, $$
with the usual convention that a product over an empty set is equal to $1$. All the products are in fact finite. Isolating $(d_i^+)^2$ for each $i$, we get the following: 
\begin{equation}\omega(\mathbf{p}) = \prod_{i \in \mathbf{i}} \left( \frac{1}{d_i^+}\right)^2\prod_{i \in V_1}d_i^+ \prod_{s>2}\prod_{i \in V_s} \left( \frac{1}{d_i^+}\right)^{s-2}.\end{equation}
Using hypothesis \eqref{H}, this can be bounded by 
\begin{equation}\prod_{i \in \mathbf{i}} \left( \frac{1}{d_i^+}\right)^2 \Delta^{v_1} \left(\frac{1}{\delta}\right)^{\sum_{s>2} \sum_{i \in V_s} (s-2)}.\end{equation}
We also have 
\begin{align*}
\sum_{s>2} \sum_{i \in V_s} (s-2 ) &= \sum_{s>2} s v_s - 2 \sum_{s>2}v_s \\
&= 2tm - v_1 - 2v_2 - 2v + 2v_1 + 2v_2 \\
&= 2(tm-v) +v_1.
\end{align*}

Thus, we have $\omega(\mathbf{p} )\leqslant \prod_{i \in \mathbf{i}} (d_i^+)^{-2} \Delta^{v_1} \delta^{-2(tm-v)}\delta^{-v_1}$. 

We are now going to give a bound on $v_1$,  the number of vertices visited once. At most $2m$ of them belong to the boundary vertices of $\mathbf{p}$. If $i$ is in $V_1$ but is not a boundary vertex, there are exactly two simple edges adjacent with $i$, one entering in $i$ and one going out of $i$. One simple edge is adjacent to at most two vertices, so two distinct vertices in $V_1$ can be adjacent to at most one common simple edge, and we have an injection from the set of non-boundary vertices in $V_1$ into the set of simple edges, whose cardinal will be denoted by $a_1'$: as there are no more than $2m$ boundary vertices, we have $v_1 \leqslant 2m + a'_1$. Those $a'_1$ edges might however be inconsistent: if $a_1' = a_1+z'$ with $z'$ the simple and inconsistent edges, we have $z' \leqslant b$.

\begin{lem}\label{b}With the preceding notations, $b \leqslant 4\chi$.\end{lem}

This yelds $v_1 \leqslant 2m+4\chi + a_1$. As $\Delta/\delta\geqslant 1$, we have $(\Delta/\delta)^{v_1} \leqslant (\Delta/\delta)^{2m+4\chi+a_1}$ and finally
$$\omega(\mathbf{p}) \leqslant \prod_{i \in \mathbf{i}} (d_i^+)^{-2} (\Delta/\delta)^{2m+4\chi + a_1} \frac{1}{\delta^{2(tm-v)}}.$$
Asymptotics \ref{asymptotic_lemma} give $\Delta^{2m} = n^{o(1)}$. Taking $C = (\Delta/\delta)^4$ ends the proof of \eqref{weightp}.
\end{proof}

\begin{proof}[Proof of lemma \ref{b}]
Fix some inconsistent edge $y=(\mathbf{e}, \mathbf{f})$. Without loss of generality we can suppose that there is another edge with $\mathbf{e}$ as its beginning half-edge (say, $(\mathbf{e}, \mathbf{f}')$ with $\mathbf{f}'\neq \mathbf{f}$)  in $\mathbf{p}$. If $\mathbf{e}$ is attached to vertex $v$, then there are at most $4$ excess edges caused by the fact that $y$ is not consistent. Therefore, the total number of inconsistent edges is at most $4\chi$. 
\end{proof}

\subsection{Expressing $f$ with graph-dependant variables.}
Let $\mathbf{p}$ be in $X^{a,v}_m(\mathbf{i})$. In order to apply Theorem \ref{tech} to $\mathbf{p}$, we need a finer knowledge on the number of consistent or simple edges depending on $a$ and $v$. The general idea is the following: the more excess edges, the lesser simple and consistent edges. To apply Theorem \ref{tech}, we define $\mathfrak{p}$ to be the proto-path naturally given by $\mathbf{p}$. All the quantities $a,a_1$ and $b$ appearing in \eqref{tech:equation1} depend on $\mathfrak{p}$. A plain application of Theorem \ref{tech} and \eqref{weightp} with any $n$ greater than $n_0$, $N = 2tm$ and $p=2tm$ yields the following inequality:
\begin{equation}|f(\mathbf{p})|\leqslant 24 n^{o(1)} \prod_{i \in \mathbf{i}} \left(\frac{1}{d_i^+}\right)^2  \frac{\mathtt{C}_1^{\chi} 3^b}{\delta^{2(tm-v)}} \left( \frac{c}{M} \right)^{a} \left(\frac{6tm}{\sqrt{M}} \right)^{a_1}.\end{equation}
We now simplify this expression. The term $24n^{o(1)}$ is still of order $n^{o(1)}$. Let $a'_1$ be the number of simple edges (not necessarily consistent) and $a'_2$ be the number of other edges. It is clear that 
$$\begin{cases}a'_1 + a'_2 = a \\ a'_1 + 2a'_2 \leqslant 2mt\end{cases}$$ so $a'_1 \geqslant 2 (a - mt)$.  If $b$ is the number of inconsistent edges we have $a_1 \geqslant a'_1 - b$ so $a_1 \geqslant (2(a-tm) - b)_+$. Using Lemma \ref{b}, we get $a_1 \geqslant \big( 2(a-tm) - 4\chi \big)_+$. We use again Lemma \ref{b}:
\begin{equation}\label{ineq:f(p)}|f(\mathbf{p})|\leqslant n^{o(1)} \prod_{i \in \mathbf{i}} \left(\frac{1}{d_i^+}\right)^2  \frac{(3^4 C)^{ \chi}}{\delta^{2(tm-v)}} \left( \frac{c}{M} \right)^{a} \left(\frac{6tm}{\sqrt{M}} \right)^{\big( 2(a-tm) - 4\chi \big)_+}.\end{equation}
Proposition \ref{regions} now follows from \eqref{ineq:f(p)} by noting that $(2(a-tm) - 4\chi)_+=0$ if and only if $\chi \geqslant v-tm-1$. 

\section{Asymptotic analysis. }\label{sec:asymptotic_analysis}

We finally gather all the results from Sections \ref{sec:combi}-\ref{sec:analysis_f} and study their limit as $n$ grows to infinity. More precisely, we will pick only integers $n$ greater than $\max\{n_0, n_1\}$. We first decompose the sum \eqref{FUNDINEQ} according to $v,\chi$ and $\mathbf{i}$:
\begin{equation}\label{DECO-avi}
\mathbf{E}\left[ \Vert \underline{P}^{(t)}\Vert ^{2m}  \right] \leqslant \sum_{v=2}^{2mt} ~~ \sum_{\mathbf{i}=(i_1, \dotsc,  i_v)} ~~\sum_{\chi=0}^{2tm-v+1} \left( \sum_{\mathbf{p} \in X_m^{a,v}(\mathbf{i})} |f(\mathbf{p})| \right) = \mathcal{H}_1+\mathcal{H}_2+ \mathcal{L}
\end{equation}
where

\begin{align}
\mathcal{H}_1 &=  \sum_{v=2}^{mt+1} ~~ \sum_{i_1, ..., i_v} ~~\sum_{\chi=0}^{2tm-v+1}\left( \sum_{\mathbf{p} \in X_m^{a,v}(\mathbf{i})} |f(\mathbf{p})| \right)\label{def:H1zone} \\ \nonumber \\
\mathcal{H}_2 &=  \sum_{v=mt+2}^{2mt} ~~ \sum_{i_1, ..., i_v} ~~\sum_{\chi=v-tm-1}^{2tm-v+1}\left( \sum_{\mathbf{p} \in X_m^{a,v}(\mathbf{i})} |f(\mathbf{p})| \right)\label{def:H2zone} \\ \nonumber \\
\label{def:Lzone}\mathcal{L} &=  \sum_{v=mt+2}^{2mt} ~~ \sum_{i_1, ..., i_v} ~~\sum_{\chi=0}^{v-tm-2}\left( \sum_{\mathbf{p} \in X_m^{a,v}(\mathbf{i})} |f(\mathbf{p})| \right).
\end{align}

Each term will be separately bounded by $o(1)n^3(c\tilde{\rho})^{2tm}$ as claimed in \eqref{objectif_aux}. 

\subsection{Bound for $\mathcal{H}_1$.}

In this sum we sum over $v \leqslant tm + 1$. We use Proposition \ref{regions} and \eqref{combidefinitive} with $n$ greater than $n_0$. 
\begin{align}
 \sum_{\mathbf{p} \in X_m^{a,v}(\mathbf{i})} |f(\mathbf{p})| &\leqslant   \sum_{\mathbf{p} \in X_m^{a,v}(\mathbf{i})} \frac{n^{o(1)}}{\delta^{2(tm-v)}} \prod_{i \in \mathbf{i}} \left(\frac{1}{d_i^+}\right)^2  \left( \frac{C}{M} \right)^{\chi} \left( \frac{c}{M} \right)^{v-1} \\
&\leqslant  \left( \prod_{i \in \mathbf{i}} d_{i}^+ d_i^- \right) C^\chi n^{\frac{25}{\A} + \frac{17}{\A}  \chi} \frac{n^{o(1)}}{\delta^{2(tm-v)}} \prod_{i \in \mathbf{i}} \left(\frac{1}{d_i^+}\right)^2  \left( \frac{C}{M} \right)^{\chi} \left( \frac{c}{M} \right)^{v-1} \\
&\leqslant \left( \prod_{i \in \mathbf{i}} \frac{d_{i}^-}{ d_i^+} \right) \frac{ n^{\frac{25}{\A} +o(1)} }{\delta^{2(tm-v)}}  \left( \frac{C n^{\frac{17}{\A}}}{\delta n} \right)^{\chi} \left( \frac{c}{M} \right)^{v-1} \\
&\leqslant  \left( \prod_{i \in \mathbf{i}} \frac{d_{i}^-}{ d_i^+} \right) \frac{ n^{3} }{\delta^{2(tm-v)}} (Cn^{-\gamma})^{\chi} (cM^{-1})^v \label{LPM:H1}
\end{align}
where we noted $\gamma=1-17/\A \in ]0,1[$ and we chose $n$ large enough to ensure that the term $o(1)$ is smaller than $1/\A$. 
Putting \eqref{LPM:H1} into \eqref{def:H1zone} yelds
\begin{equation*}
\mathcal{H}_1 \leqslant \sum_{v=2}^{mt+1}\frac{  n^{3 } }{\delta^{2(tm-v)}}  \sum_{\mathbf{i}}(cM^{-1})^{v}  \left( \prod_{i \in \mathbf{i}} \frac{d_{i}^-}{ d_i^+} \right) \left\lbrace\sum_{\chi=0}^{2tm-v+1} (Cn^{-\gamma} )^{\chi}\right\rbrace .
\end{equation*}
The sum in $\chi$ (between braces) is a geometric sum started at $2$ and the ratio goes to $0$ as $n$ goes to infinity, so the whole term in braces is of order $o(1)$. Recall the definition of $\rho$: we have 
\begin{equation}
\label{rho_bound}\sum_{\mathbf{i}}(cM^{-1})^{v}  \left( \prod_{i \in \mathbf{i}} \frac{d_{i}^-}{ d_i^+} \right)  \leqslant \left(cM^{-1}\sum_{i=1}^n \frac{d_i^-}{d_i^+}\right)^v \leqslant (c\rho)^{2v}.
\end{equation}
Now
\begin{align}
\mathcal{H}_1&\leqslant \frac{o(1) n^{3}}{\delta^{2tm}} \sum_{v=2}^{mt+1} (c\delta \rho)^{2v} .
\end{align}

Here again, the sum is indeed geometric with ratio $c\delta\rho \leqslant c \delta \tilde{\rho}$ where we recall that $\tilde{\rho} = \rho \vee \delta^{-1}$. As $\delta \tilde{\rho} \geqslant \delta^{-1}$, we have $c\delta \tilde{\rho} \geqslant 1$, and 
\begin{align}
\sum_{v=2}^{mt+1} (c\delta \tilde{\rho})^{2v} &\leqslant (c\delta \tilde{\rho})^{2mt + 2}.
\end{align}
After simplifications, we get $\mathcal{H}_1 \leqslant o(1) n^{3}(c\tilde{\rho})^{2mt}$ which is the desired bound.

\subsection{Bound for $\mathcal{H}_2$.}

In this sum, $v>tm+1$ and $\chi \geqslant v-tm-1$. The computations are extremely similar to what was done in the preceding section, so we omit the details. As in the preceding section we have
\begin{equation} \sum_{\mathbf{p} \in X_m^{a,v}(\mathbf{i})} |f(\mathbf{p})|  \leqslant \left( \prod_{i \in \mathbf{i}} \frac{d_{i}^-}{ d_i^+} \right) \frac{ n^{\frac{25}{\A}+1 } }{\delta^{2(tm-v)}} (Cn^{-\gamma})^{\chi}(cM^{-1})^{v}. \end{equation}
The sum in $\chi$ is now started at $v-tm-1$. We have 
\begin{align}
\mathcal{H}_2 &\leqslant \sum_{v=mt+2}^{2mt} \frac{ M n^{\frac{25}{\A} } }{c \delta^{2(tm-v)}} \sum_{\mathbf{i}}  \left( \frac{c}{M} \right)^{v}\left( \prod_{i \in \mathbf{i}} \frac{d_{i}^-}{ d_i^+} \right)  \left\lbrace \sum_{\chi=v-tm-1}^{2tm-v+1} \left( \frac{C}{n^{\gamma}} \right)^{\chi} \right\rbrace
\end{align}
The sum between braces is geometric and the ratio is $o(1)$, hence it is bounded by the first term times some constant close to $1$. The first term is $(Cn^{-\gamma})^{v-tm-1}$. We also have \eqref{rho_bound} and the fact $Mn^{25/\A}/c \leqslant n^2$ when $n$ is large enough. Putting it all together, we get 
\begin{align}
\mathfrak{H}_2 &\leqslant n^2 \sum_{v=mt+2}^{2mt} \frac{(c\rho)^{2v}  (C n^{-\gamma})^{v-tm-1}  }{ \delta^{2(tm-v)}}.
\end{align}
This is indeed a geometric sum and the ratio is of order $ O(n^{-\gamma})$. After quick simplifications left to the reader, we get $\mathfrak{H}_2 \leqslant n^2 (c\tilde{\rho})^{2tm}C n^{-\gamma}$ which is also generously bounded by $o(1)n^3(c\tilde{\rho})^{2tm}$ when $n$ is large.

\subsection{Bound for $\mathcal{L}$}

In this sum, $v>tm+1$ and $\chi \leqslant v-tm-1$. The main difference with the two other regions is the extra term in the bound for $f(\mathbf{p})$. We use Proposition \ref{regions} and \eqref{combidefinitive}. 
\begin{multline*}
 \sum_{\mathbf{p} \in X_m^{a,v}(\mathbf{i})} |f(\mathbf{p})|\leqslant \\  \left( \prod_{i \in \mathbf{i}} d_{i}^+ d_i^- \right) C^\chi n^{\frac{25}{\A} + \frac{17}{\A}  \chi} \frac{n^{o(1)}}{\delta^{2(tm-v)}} \prod_{i \in \mathbf{i}} \left(\frac{1}{d_i^+}\right)^2  \left( \frac{C}{M} \right)^{\chi} \left( \frac{c}{M} \right)^{v-1} \left( \frac{6tm}{\sqrt{M}}\right)^{2(v-tm-1-\chi)}
\end{multline*}
This can be simplified when $n$ is large enough to
\begin{equation}\label{borne:LPM_light}
  \left(\frac{c}{M}\right)^v \prod_{i \in \mathbf{i}} \frac{d_i^-}{d_i^+} \frac{n^{\frac{27}{50}}}{\delta^{2(tm-v)}}   (Cn^{1-\gamma} )^{\chi}  \left( \frac{6tm}{\sqrt{M}}\right)^{2(v-tm-1)}
\end{equation}

We plug \eqref{borne:LPM_light} into the definition of $\mathcal{L}$ and we use \eqref{rho_bound}:
\begin{align}
\mathcal{L} &\leqslant  \sum_{v=mt+2}^{2mt} ~~ \sum_{\mathbf{i}} ~~\sum_{\chi=0}^{v-tm-2} \left(\frac{c}{M}\right)^v \prod_{i \in \mathbf{i}} \frac{d_i^-}{d_i^+} \frac{n^{\frac{27}{50}}}{\delta^{2(tm-v)}}   \left( Cn^{1-\gamma} \right)^{\chi}  \left( \frac{6tm}{\sqrt{M}}\right)^{2(v-tm-1)}\\
&\leqslant  \sum_{v=mt+2}^{2mt}  \frac{(c\tilde{\rho})^{2v} n^{\frac{27}{50}}}{\delta^{2(tm-v)}}   \left( \frac{6tm}{\sqrt{M}}\right)^{2(v-tm-1)} \left\lbrace \sum_{\chi=0}^{v-tm-2}\left(Cn^{1-\gamma} \right)^{\chi}   \right\rbrace.
\end{align}
As for other regions, the term between braces is a geometric with ratio greater than $1$ so it is bounded by $(Cn^{1-\gamma})^{v-tm-1}$. We are now left with a sum in $v$ 
\begin{align}
\mathcal{L}&\leqslant \sum_{v=mt+2}^{2mt}  \frac{(c\tilde{\rho})^{2v} n^{\frac{27}{50}}}{\delta^{2(tm-v)}}   \left( \frac{6tm}{\sqrt{M}}\right)^{2(v-tm-1)}(Cn^{1-\gamma})^{v-tm-1} 
\end{align}
and this is generously bounded by $o(1)n^2 (c\tilde{\rho})^{2tm}$; note that $\mathcal{L}$ is negligible in front of $\mathcal{H}_1$, $\mathcal{H}_2$.

\section{Proof of Proposition \ref{control2}.}\label{section:tangled_rest}

We now prove Proposition \ref{control2}. The strategy is exactly the same as for Proposition \ref{control1} and runs along the lines of its proof. We omit the details. First, we recall \eqref{def:tangled_rest}: 
$$R^{t,\ell}(i,j) = \sum_{\mathbf{p} \in \mathscr{R}^{t,\ell}(i,j)} \prod_{s=1}^{\ell-1} A(\mathbf{e}_s, \mathbf{f}_s ) \frac{1}{d^+_{\mathbf{e}_{\ell}}} \prod_{s=\ell+1}^t \underline{A}(\mathbf{e}_{s}, \mathbf{f}_s )$$
where $\mathscr{R}^{t,\ell} (i,j)$ had been defined in Definition \ref{defin:RTL} on page \pageref{defin:RTL}.

\subsection{Trace method.}

We note $Y(\mathbf{p}) =  \prod_{s=1}^{\ell-1} A(\mathbf{e}_s, \mathbf{f}_s ) \big(d^+_{\mathbf{e}_{\ell}} \big)^{-1} \prod_{s=\ell+1}^t \underline{A}(\mathbf{e}_{s}, \mathbf{f}_s )$ when $\mathbf{p}$ is in $\mathscr{R}^{t,\ell}$.  Using the classical trace method as in Subsection \ref{subsec:trace}, we find 
\begin{align}
\Vert R^{t, \ell}\Vert ^{2m} &\leqslant \sum_{i_1, ..., i_{2m}} \prod_{s=1}^m R^{t, \ell} (i_{2s-1}, i_{2s}) R^{t, \ell}(i_{2s+1}, i_{2s}).  \\
&\leqslant \sum_{i_1, ..., i_{2m}} \prod_{s=1}^m \left( \sum_{\mathbf{p} \in \mathscr{R}^{t, \ell}(i_{2s-1}, i_{2s})}Y(\mathbf{p}) \right) \left( \sum_{\mathbf{p} \in \mathscr{R}^{t, \ell}(i_{2s+1}, i_{2s})}Y(\mathbf{p}) \right) \\
&\leqslant \sum_{i_1, ..., i_{2m}}\sum_{(\mathbf{p}_1, ..., \mathbf{p}_{2m})} \prod_{s=1}^{2m} Y(\mathbf{p}_i)
\end{align} 
where the sum is over all $2m$-tuples $(\mathbf{p}_1, ..., \mathbf{p}_{2m})$ such that $\mathbf{p}_{2s}$ is in $\mathscr{R}^{t,\ell}(i_{2s-1}, i_{2s})$ and $\mathbf{p}_{2s+1}$ is in $\mathscr{R}^{t,\ell}(i_{2s+1}, i_{2s})$; note that we used the cyclic convention $i_{2m+1} = i_1$. Now, going back to the definition of $R^{t, \ell}$, we have 
\begin{equation}\label{ERTL}
\mathbf{E}[\Vert R^{t,\ell}\Vert ^{2m}] \leqslant \sum_{\mathbf{p} \in \mathscr{C}'_{m,\ell}}|g(\mathbf{p}) |
\end{equation}
where $\mathscr{C}'_{m,\ell}$ and $g$ are now defined in the same fashion as $\mathscr{C}_m$ and $f$ in Section \ref{strat}. 

\begin{defin}
$\mathscr{C}'_{m,\ell}$ is the set of $2m$-tuples $(\mathbf{p}_1, ..., \mathbf{p}_{2m})$ such that 
\begin{itemize}
\item $\mathbf{p}_s$ is in $R^{t,\ell}$ for every $s$ odd, 
\item $\bar{\mathbf{p}}_s$ is in $R^{t,\ell}$ for every $s$ even,  where $\bar{\mathbf{p}}_s$ denotes path $\mathbf{p}_s$ ``reversed", 
\item the beginning vertex of $\bar{\mathbf{p}}_{2s}$ is the beginning vertex of $\mathbf{p}_{2s+1}$
\item the endvertex vertex of $\mathbf{p}_{2s-1}$ is the endvertex of $\bar{\mathbf{p}}_{2s}$. 
\end{itemize}
\end{defin}

Finally, for every $\mathbf{p}=(\mathbf{p}_1, \dotsc, \mathbf{p}_{2m})$ in $\mathscr{C}'_{m,\ell}$, we set
 \begin{equation}
 \label{def:g}
 g(\mathbf{p}) = \prod_{s=1}^{m} Y(\mathbf{p}_{2i-1})Y(\bar{\mathbf{p}}_{2i}).
 \end{equation}

\begin{remarque}\label{rq:tangle}
If $\mathbf{p}_i$ is in $\mathscr{R}^{t,\ell}$, then it has at least two cycles. This fact has two consequences:  the number of vertices visited by $\mathbf{p}_i$ is smaller than $t-2$, and the tree excess $\chi(\mathbf{p}_i)$ is greater than $2$. When this is applied to $\mathbf{p}$, we get the following facts: 
\begin{itemize}
\item $\mathbf{p}$ visits no more than $2tm - 2m=2m(t-1) $ vertices.
\item $\chi(\mathbf{p})$ is greater than $4m$.
\end{itemize}

\end{remarque}

Our task is to prove Proposition \ref{control2}. To this end, we are going to prove the following lemma: 

\begin{lem}If $n$ is large enough, then 
\begin{equation}\label{lem:RTL}\mathbf{E}[\Vert  R^{t,\ell} \Vert ^{2m}] = o (1)n^{2m+3} (c\tilde{\rho})^{2m(t+\ell)} 
\end{equation}
with $D>0$ a constant (we can take $D=100$).
\end{lem}

\begin{proof}[Proof of Proposition \ref{control2} using \eqref{lem:RTL}]By the Markov inequality, we have
\begin{align*}
\mathbf{P}\big(\Vert R^{t,\ell}\Vert > n\ln(n)^D (c\tilde{\rho})^{t+\ell}) &\leqslant \frac{\mathbf{E}[\Vert R^{t,\ell}\Vert ^{2m}]}{n^{2m} \ln(n)^{2Dm}(c\tilde{\rho})^{2m(t+\ell)}}\\
&\leqslant \frac{o(1)n^3 n^{2m} (c\tilde{\rho})^{2m(t+\ell)}}{n^{2m} \ln(n)^{2Dm}(c\tilde{\rho})^{2m(t+\ell)}} \leqslant o(1)n^{3 - \frac{2D}{\A}} .
\end{align*}
If $D$ is chosen great enough ($D=100$ is sufficient), then the last term goes to zero and we get $\mathbf{P}\big(\Vert R^{t,\ell}\Vert > n\ln(n)^D (c\tilde{\rho})^{t+\ell}) = o(1)$, which is the desired result.
\end{proof}

We are now going to prove \eqref{lem:RTL}, first studying the combinatorics of $\mathscr{C}'_{m,\ell}$, then bounding $g(\mathbf{p})$ and finally doing the asymptotic analysis.

\subsection{Combinatorics of $\mathscr{C}'_{m,\ell}$.}\label{subsec:combi2}

We split $\mathscr{C}'_{m,\ell}$ into disjoints subsets. 

\begin{defin}Let $a,v$ be integers and let $\mathbf{i}=(i_1, ..., i_v)$ a $v$-tuple of vertices. We define
$$ X^{v,a}_{m,\ell}(\mathbf{i}) = X_{m,\ell}^{v,a}(i_1, ..., i_v)$$
as the set of all the paths in $\mathscr{C}'_{m,\ell} $ whose vertex set is precisely $(i_1, ..., i_v)$ (in this order) and who have $a$ edges. 

Let $\mathbf{p}$ and $\mathbf{p}'$ be two paths in $\mathscr{C}'_{m,\ell}$; we note $\mathbf{e}_{i,s}, \mathbf{f}_{i,s}$ the half-edges of $\mathbf{p}$ and $\mathbf{e}'_{i,s}, \mathbf{f}'_{i,s}$ those of $\mathbf{p}'$. Those paths are said \emph{equivalent} if 

\begin{itemize}
\item they both belong to $X^{a,v}_{m,\ell}(\mathbf{i})$ and they visit the same vertices at the same time, 
\item for every vertex $u \in \mathbf{i}$, there are two permutations $\sigma_u  \in \mathfrak{S}_{d^+_u}$ and $\tau_u \in \mathfrak{S}_{d_u^-}$ such that for every $i$ and $s$, if $\mathbf{e}_{i,s}$ is a head attached to $u$ and $\mathbf{f}_{i,s}$ a tail attached to $u$, then
$$\mathbf{e}_{i,s} = \sigma_u (\mathbf{e}'_{i,s}) \quad \text{and} \quad \mathbf{f}_{i,s} = \tau_u (\mathbf{f}'_{i,s}). $$
\end{itemize}

\end{defin}

Indeed, two paths are equivalent if they only differ by a permutation of their half-edges. 
We state again Lemma \ref{lem:cardinal_equiv_classes}. Its proof remains unchanged, and Lemma \ref{b} is also true in this case. 

\begin{lem}\label{lem:cardinal_equiv_classes2}
Let $\mathbf{p}$ be a path in $X^{v,a}_{m,\ell} (\mathbf{i})$. Then, we have at most
\begin{equation}
C^\chi \prod_{i \in \mathbf{i}} d^+_{i}d^-_{i}
\end{equation}
paths equivalents to $\mathbf{p}$, where $C$ is a constant.
\end{lem}

\subsection{Number of equivalence classes.}

Now, we count the number of equivalence classes in $X^{a,v}_{m,\ell}(\mathbf{i})$. The explored vertices are $\mathbf{i} = (i_1, ..., i_v)$, in this order.

In any equivalence class, we choose a path $\mathbf{p}$ visiting heads and tails in the ``alternating lexicographic order" in the same fashion as in \ref{subsub:choice}. The chosen path $\mathbf{p}$ will be called the \emph{representative path} of the class $X^{a,v}_{m,\ell}(\mathbf{i})$. We build the tree $T$ in the exact same way.

Cycling times are defined in the same way, but now another phenomenon can occur: there can be more than one cycle inside one subpath $\mathbf{p}_i$. However, the path $\mathbf{p}_i$ is composed of two subpaths, say $\mathbf{p}_i'$ and $\mathbf{p}_i''$, linked by a single edge\footnote{Which can also be considered as a tangle-free path of length $1$.}, and inside one of the two paths $\mathbf{p}_i', \mathbf{p}_i''$, there can be no more than one cycle. Thus, a small variation of the code for $\mathscr{C}_m$ will be sufficient for our purpose. To this end, define the \emph{bridging time}

$$\ell_i := \begin{cases}
\ell-1 \text{ if } i \text{ is even} \\
t-\ell \text{ else.}
\end{cases}$$

\subsubsection{Short cycling times.}
 Each sub-path $\mathbf{p}'_i, \mathbf{p}_i''$ is tangle-free. Let $r'_i$ denotes the first time when $\mathbf{f}_{i,r'_i}$ is attached to a vertex \emph{already visited by $\mathbf{p}'_i$}, and similarly $r''_i$ for $\mathbf{p}''_i$.  Those are \emph{short cycling times}.

 If these cycling times does not exist, we artificially set them to be the symbol $\otimes$. Let $\sigma'_i, \sigma''_i$ be the first time when the path $\mathbf{p}'_i, \mathbf{p}''_i$ left this vertex after its first visit. Finally, note $h'_i, h''_i$ the ``total time spent in the loop".

 We mark the cycling times $r'_i, r''_i$ as follows: 
\begin{equation}
(j'_{i,r'_i}, \mathbf{f}_{i,r'_i}, h'_i, \mathbf{e}_{i, \tau'_i}, u'_i) \quad \text{and}\quad (j''_{i,r''_i}, \mathbf{f}_{i,r''_i}, h''_i, \mathbf{e}_{i, \tau''_i}, u''_i)
\end{equation}
and if $r'_i, r''_i = \otimes$ this mark is set to be $\emptyset$.

We also have to deal with what happens at the bridge between $\mathbf{p}'_i$ and $\mathbf{p}''_i$. To this end, we simply mark the bridging time $\ell_i$ with the whole bridge, that is we set 
$$\beta_i = (\mathbf{e}_{i,\ell_i}, \mathbf{f}_{\ell_i+1}).$$

All those informations are enough to reconstruct the short cycling times and the bridge. Note that we did not fully exploit the $R^{t,\ell}$-structure of the paths $\mathbf{p}_i$: in particular, we did not use the fact that in the end, $\mathbf{p}_i$ is tangled. This will be used further.

Let us count those codes. We have at most two short cycling time per $\mathbf{p}'_i$ or $\mathbf{p}''_i$. There are $\ell_i$ choices for the first short cycling time and at most $\Delta (v\Delta)t(v\Delta)v = t \Delta^2 v^3$ choices for its mark, then there are at most $\Delta v$ choices for the bridge, and finally there are at most $t - \ell_i$ choices for the second short cycling time and $t\Delta^2 v^3$ choices for its mark.

\subsubsection{Long cycling times.}
Let $(i,t)$ be a cycling time  leading to the (already known) vertex $u$. If $(i,t)$ is not a short cycling time, then
\begin{enumerate}
\item either $u$ belongs to the verties discovered by some $\mathbf{p}_j$ with $j<i$, 
\item either $t>\ell_i$ and $u$ belongs to the vertices discovered by $\mathbf{p}'_i$. 
\end{enumerate}

In either cases, we say $(i,t)$ is a long cycling time. We mark long cycling times with a triple
\begin{equation}\label{LCM2}(j_{i,t}, \mathbf{f}_{i,t}, u_i ).
\end{equation}

where $j_{i,t}$ is the index of the head \footnote{Or the head, depending on the parity of $i$.} by which we're leaving the current vertex, $\mathbf{f}_{i,t}$ is the tail we are going to, and $u_i$ is the next vertex when we will be leaving the tree $T$. For every long cycling time, there are at most $\Delta^2 v^2$ marks like \eqref{LCM2}.

\subsubsection{Superfluous times.}Superfluous cycling times are defined as in \ref{subsec:superfluous} and play no role in the sequel.

\subsubsection{Total number.}
We now gather the number of different types of marks to get a bound on the number of equivalence classes in $\mathscr{C}'_{m,\ell}$.

\begin{prop}
The total number of equivalence classes of paths in $\mathscr{C}'_{m,\ell}$ visiting vertices $\mathbf{i}=(i_1, ..., i_v)$ and having $a$ edges is at most
\begin{equation}\label{counting2}
4^{-m}(2\Delta tm)^{4m\chi + 22m}.
\end{equation}
\end{prop}

\begin{proof}

Recall the definitions of section \ref{strat} and the difference between \emph{edges} of $\mathbf{p}$ and \emph{graph-edges} of $\mathbf{p}$. Consider the undirected multi-graph spanned by the \emph{unoriented} graph-edges of $\mathbf{p}$ on vertices $\mathbf{i} = (i_1, ..., i_v)$. This graph is connected. Its total number of edges is at most $a$ (if no edge is visited two times in opposite directions\footnote{Observe that it is also at least $a/2$ if all edges are visited twice, in opposite directions. This will not be used in the proof.}. Therefore, there are at most $\chi := a-v+1$ excess edges. For each $i \leqslant 2m$, there are at most $\chi$ cycling times, \emph{a fortiori} there are at most $\chi$ long cycling times. Therefore, we have at most $t^{2m\chi}$ choices for the positions of the long cycling times. For each $i$, there are at most two cycling times, one before $\ell_i$ and one after. The total number of choices for these short cycling times is thus $\prod_{i=1}^{2m} \ell_i (k-\ell_i) = \ell^{2m}(t-\ell)^{2m} \leqslant 4^{-m}t^{2m}$.

For each one of these choices, we have the following number of possibilities for the marks: $(t\Delta^2 v^3)^{2 \times 2m} $ for short cyclings, $(\Delta v)^{4m}$ for bridges, $(\Delta^2 v^2)^{2m\chi}$ for long cyclings. The total number of codings is at most $4^{-m} t^{6m} \Delta^{12m + 4m\chi} v^{16m + 4m\chi}$ which (using $v \leqslant 2tm$) is largely bounded by \eqref{counting2}. 

\end{proof}

Using the asymptotic properties exposed in lemma \ref{asymptotic_lemma}, the reader can check that \eqref{counting2} is  bounded by $n^{\frac{45}{\A} + \frac{17}{\A}\chi}$ when $n$ is large enough. Using Lemma \ref{lem:cardinal_equiv_classes2}, we get the following variant of Proposition \eqref{prop:combidefinitive}:

\begin{prop}\label{thm:combidefinitive2}
Fix $\ell, v,\mathbf{i}$ and $a$. Then, when $n$ is big enough we have
\begin{equation}\label{combidefinitive2}
\# X^{v,a}_{m,\ell}(\mathbf{i}) \leqslant  \left( \prod_{i \in \mathbf{i}} d_{i}^+ d_i^- \right)  C^\chi n^{\frac{45}{\A} + \frac{17}{\A}  \chi} .
\end{equation}

\end{prop}

\subsection{Analysis of $g$.}\label{sec:analysis_g}

We now bound $g(\mathbf{p})$ when $\mathbf{p}$ is in $\mathscr{C}'_{m,\ell}$, following the ideas in Section \ref{sec:analysis_f}. Recall the definition of $g$ as in \eqref{def:g}. When developping the terms in $Y$, if we note $\mathbf{p}_i = (\mathbf{e}_{i,s}, \mathbf{f}_{i,s})_{s\leqslant t}$ for $i$ odd and $\bar{\mathbf{p}}_i = (\mathbf{e}_{i,s}, \mathbf{f}_{i,s})_{s\leqslant t}$ for $i$ even, then we have 
\begin{equation}\label{eq:gprod}g(\mathbf{p}) = \prod_{i=1}^{2m}\frac{1}{d^+_{\mathbf{e}_{i,\ell}}} \times  \prod_{i=1}^{2m} \prod_{s < \ell} A(\mathbf{e}_{i,s}, \mathbf{f}_{i,s})\prod_{s>\ell}\underline{A}(\mathbf{e}_{i,s}, \mathbf{f}_{i,s}).\end{equation}

Fix a path $\mathbf{p}$ in $X^{a,v}_{m,\ell}(\mathbf{i})$. Lemma \ref{lem:weightp} remains exactly the same.

\begin{lem}\label{lem:weightp2}
There is a constant $C$ such that for every $\mathbf{p} \in X^{a,v}_{m,\ell}(\mathbf{i})$ we have 
\begin{equation}\label{weightp2}
\omega(\mathbf{p}) \leqslant n^{o(1)} \prod_{i \in \mathbf{i}} \left(\frac{1}{d_i^+}\right)^2  \frac{C^{\chi + a_1}}{\delta^{2(tm-v)}}.
\end{equation}
\end{lem}

Now comes the application of Theorem \ref{tech} to the second factor in the right of \eqref{eq:gprod}. Let $\mathbf{p}$ be a path in $X^{a,v}_{m,\ell}(\mathbf{i})$. In order to apply Theorem \ref{tech}, we need to define an auxiliary path, say $\hat{\mathbf{p}}$, by deleting each $\ell$-th edge in a subpath $\mathbf{p}_i$. We plug \eqref{weightp2} into the bound given by Theorem \ref{tech} to get 
$$|g(\mathbf{p})|\leqslant 24 n^{o(1)} \prod_{i \in \mathbf{i}} \left(\frac{1}{d_i^+}\right)^2  \frac{C^{\chi} 3^b}{\delta^{2(tm-v)}} \left( \frac{c}{M} \right)^{\hat{a}} \left(\frac{Cm(t-1)}{\sqrt{M}} \right)^{a_1}.$$
where $\hat{a}$ is the total number of edges of $\hat{\mathbf{p}}$, so $\hat{a} \geqslant a - 2m$ with $a$ the total number of edges of $\mathbf{p}$. Also, $a_1$ is now the number of simple, consistent edges that appear in the path $\hat{\mathbf{p}}$: 
\begin{itemize}
\item in $\mathbf{p}_i$, after $\ell$ if $i$ is odd, 
\item in $\mathbf{p}_i$, before $t-\ell$ if $i$ is even.
\end{itemize}

Such edges will be called \emph{good edges} just for this paragraph. Note $\bar{a}_1$ the total number of simple, consistent edges in $\mathbf{p}$; as there are no more than $2m\ell$ edges that are not good, we have $a_1 \geqslant (\bar{a}_1 - 2m\ell)_+$.

Let $\bar{a}'_1$ be the number of simple edges (not necessarily consistent) of $\mathbf{p}$ and $\bar{a}'_2$ be the number of other edges. It is clear that $\bar{a}'_1 + \bar{a}'_2 = a $ and $ \bar{a}'_1 + 2\bar{a}'_2 \leqslant 2mt$ so $\bar{a}'_1 \geqslant 2 (a - mt)$.  If $b$ is the number of inconsistent edges we have $\bar{a}_1 \geqslant \bar{a}'_1 - b$ so $\bar{a}_1 \geqslant 2(a-tm) - b$, and using Lemma \ref{b}, we get $\bar{a}_1 \geqslant 2(a-tm) - 4\chi $ and finally $a_1 \geqslant (2(a-tm) - 4\chi - 2\ell m)_+$.  We also have $24n^{o(1)} = n^{o(1)}$. Note that $ 2(a-tm) - 4\chi - 2\ell m =2\big((v-1) - (t + \ell)m - \chi \big)$.  Using once again Lemma \ref{b}, we get
\[
|g(\mathbf{p})|\leqslant n^{o(1)} \prod_{i \in \mathbf{i}} \left(\frac{1}{d_i^+}\right)^2  \frac{C^{ \chi}}{\delta^{2(tm-v)}} \left( \frac{c}{M} \right)^{a-2m} \left(\frac{C tm}{\sqrt{M}} \right)^{2\big((v-1) - (t + \ell)m - \chi \big)_+}.
\]

The $2\big((v-1) - (t + \ell)m - \chi \big)_+$ term is zero if and only if $\chi \geqslant v-tm-t\ell-1$, hence the following result.

\begin{prop}\label{regions2}
Let $\mathbf{p}$ be any path in $\mathscr{C}'_{m,\ell}$ with $v$ vertices and $a$ edges. Note $\chi = a-v+1$. There is a constant $C$ such that when $n$ is large enough, we have
\begin{itemize}
\item  If  $\chi \geqslant v-(t+\ell)m - 1$, then 
$$|g(\mathbf{p})|\leqslant  \frac{n^{o(1)}}{\delta^{2(tm-v)}} \prod_{i \in \mathbf{i}} \left(\frac{1}{d_i^+}\right)^2  \left( \frac{C}{M} \right)^{\chi} \left( \frac{c}{M} \right)^{v-1  -2m} .$$
\item  Else, $\chi \leqslant v-(t+\ell)m-1$ and in this case, 
$$|g(\mathbf{p})|\leqslant  \frac{n^{o(1)}}{\delta^{2(tm-v)}} \prod_{i \in \mathbf{i}} \left(\frac{1}{d_i^+}\right)^2  \left( \frac{C}{M} \right)^{\chi} \left( \frac{c}{M} \right)^{v-1} \left( \frac{C2tm}{\sqrt{M}}\right)^{2(v-tm - \ell m-1-\chi)}.$$

\end{itemize}
\end{prop}

\subsection{Asymptotic analysis.}

All the computations in this section have already been done in Section \ref{sec:asymptotic_analysis}, se we do not write the details. Go back to \eqref{ERTL} and decompose the sum according to $a,v,\mathbf{i}$: 
$$\mathbf{E}\left[\Vert R^{t,\ell}\Vert^{2m} \right] \leqslant \sum_{v=2}^{2mt} ~~ \sum_{\mathbf{i}=(i_1, ..., i_v)} ~~\sum_{\chi=4m}^{2tm-v+1}\left( \sum_{\mathbf{p} \in X_{m,\ell}^{a,v}(\mathbf{i})} |g(\mathbf{p})| \right)= \mathcal{H}_1'+\mathcal{H}_2'+\mathcal{L}'$$
where
\begin{align}
\mathcal{H}'_1 &=  \sum_{v=2}^{m(t+\ell)+1} ~~ \sum_{i_1, ..., i_v} ~~\sum_{\chi=4m}^{2tm-v+1}\left( \sum_{\mathbf{p} \in X_{m,\ell}^{a,v}(\mathbf{i})} |g(\mathbf{p})| \right)\\ \nonumber \\
\mathcal{H}'_2 &=  \sum_{v=m(t+\ell)+2}^{2m(t-2)} ~~ \sum_{i_1, ..., i_v} ~~\sum_{\chi=v-(t+\ell)m-1}^{2tm-v+1}\left( \sum_{\mathbf{p} \in X_{m,\ell}^{a,v}(\mathbf{i})} |g(\mathbf{p})| \right)\\ \nonumber \\
\mathcal{L}' &=  \sum_{v=m(t+\ell)+2}^{2mt} ~~ \sum_{i_1, ..., i_v} ~~\sum_{\chi=4m}^{v-tm-\ell m-2}\left( \sum_{\mathbf{p} \in X_{m,\ell}^{a,v}(\mathbf{i})} |g(\mathbf{p})| \right).
\end{align}

Each one of those terms can be bounded by the appropriate quantity as requested in Proposition \ref{control2}, that is $o(1) n^{2m+3}(c\tilde{\rho})^{2m(t+\ell)}$. 

For example, in $\mathcal{H}_1'$, we sum over indices such that $v \leqslant (t+\ell)m + 1$. We then have 
\begin{align*}
\sum_{\mathbf{p} \in X_{m,\ell}^{a,v}(\mathbf{i})} |g(\mathbf{p})| &\leqslant   \sum_{\mathbf{p} \in X_{m,\ell}^{a,v}(\mathbf{i})} \frac{n^{o(1)}}{\delta^{2(tm-v)}} \prod_{i \in \mathbf{i}} \left(\frac{1}{d_i^+}\right)^2  \left( \frac{C}{M} \right)^{\chi} \left( \frac{c}{M} \right)^{v-1-2m} \\
&\leqslant  \left( \prod_{i \in \mathbf{i}} d_{i}^+ d_i^- \right)C^\chi n^{\frac{45}{\A} + \frac{17}{\A}  \chi} \frac{n^{o(1)}}{\delta^{2(tm-v)}} \prod_{i \in \mathbf{i}} \left(\frac{1}{d_i^+}\right)^2  \left( \frac{C}{M} \right)^{\chi} \left( \frac{c}{M} \right)^{v-1-2m} \\
&\leqslant \left( \prod_{i \in \mathbf{i}} \frac{d_{i}^-}{ d_i^+} \right) \frac{ n^{\frac{45}{\A} } }{\delta^{2(tm-v)}}  \left(Cn^{-\gamma} \right)^{\chi} \left(cM^{-1} \right)^{v-1-2m}
\end{align*}
with $\gamma = 1-17/\A \in ]0,1[$. As noted in Remark \ref{rq:tangle}, if $\mathbf{p}$ is in $\mathscr{C}'_{m,\ell}$, then $\chi$ cannot be less than $4m$. We thus have \begin{align*}
\mathcal{H}_1 &\leqslant \sum_{v=2}^{mt+m\ell+1}\left(\frac{M}{c}\right)^{2m+1}\frac{  n^{\frac{45}{\A} } }{ \delta^{2(tm-v)}}  \sum_{\mathbf{i}}\left( cM^{-1} \right)^{v}  \left( \prod_{i \in \mathbf{i}} \frac{d_{i}^-}{ d_i^+} \right) \left\lbrace\sum_{\chi=4m}^{2tm-v+1}   \left( Cn^{-\gamma} \right)^{\chi}\right\rbrace \\
&\leqslant \sum_{v=2}^{mt+m\ell+1}n^{2m+1+o(1)}\frac{  n^{\frac{45}{\A} } }{ \delta^{2(tm-v)}}  \sum_{\mathbf{i}}\left( cM^{-1}\right)^{v}  \left( \prod_{i \in \mathbf{i}} \frac{d_{i}^-}{ d_i^+} \right) \left( Cn^{-\gamma} \right)^{4m} \left\lbrace\sum_{\chi=0}^{2m(t-1)-v+1}   \left(Cn^{-\gamma} \right)^{\chi}\right\rbrace \\
\end{align*}
The sum in $\chi$ (between braces) is a bounded by $2$ if $n$ is large enough and the sum in $\mathbf{i}$ is bounded by $(c \rho)^{2v}$, hence
\begin{align*}\mathcal{H}_1 \leqslant 2n^{2+2m+o(1)-4m\gamma } \sum_{v=2}^{mt+m\ell+1}\frac{ (c\rho)^{2v} }{ \delta^{2(tm-v)}} &\leqslant n^{2m+2-4m\gamma} \delta^{-2tm} \frac{(c\delta \rho)^{2mt +2m\ell -2} -1}{(c\delta \rho)-1} (c\delta \rho)^2 \\
&\leqslant n^{2m+3-4m\gamma}\delta^{2\ell m} (c\tilde{\rho})^{2mt+2m\ell}.
\end{align*}
To conclude, note that $\delta^{2tm} = n^{2m\alpha/\ln(\Delta)}$; when $\alpha$ is chosen to be strictly smaller than $2\ln(\Delta)\gamma$, the term $n^{-4m\gamma} \delta^{2m\ell}$ becomes $o(1)$.

\bigskip

We bound $\mathcal{H}'_2$ and $\mathcal{L}'$ in the same way, adapting the computations already done in the preceding sections. 

\vspace*{1cm}
$$\star\star\star$$
\vspace*{1cm}

\appendix 

\section{Algebraic tools.}\label{app:algebra}

In this section, we prove Lemma \ref{quanti_bauer_fike}. We begin with a classical theorem (\cite{bauer}) connecting the eigenvalues of any diagonalizable matrix $A$ with the eigenvalues of any perturbation of $A$. If $M$ is a matrix, we note $\sigma(M)$ the set of its eigenvalues.

\begin{theorem}[Bauer-Fike]\label{thm:Bauer-Fike}
Let $A$ be a diagonalizable matrix, $A=PDP^{-1}$ with $P$ invertible and $D$ diagonal, and let $H$ be any matrix. 
\begin{enumerate}
\item Define $\varepsilon = \Vert P\Vert  \cdot \Vert P^{-1}\Vert  \cdot \Vert H \Vert $. Then, 
\begin{equation}\label{bauer-fike}
\sigma(A+H) \subset \bigcup_{\lambda \in \sigma(A)} B(\lambda, \varepsilon).
\end{equation}
\item If $I$ is a subset of $\{1, ..., n \}$ such that 
\[
\bigcup_{i \in I} B(\lambda_i, \varepsilon) \cap \bigcup_{i \notin I} B(\lambda_i, \varepsilon) = \emptyset 
\]
then the number of eigenvalues of $A+H$ lying in $\bigcup_{i \in I} B(\lambda_i, \varepsilon)$ is exactly $\#I$.
\end{enumerate}

\end{theorem}

Hence, the spectrum of the perturbed matrix $A+H$ is entirely contained in the $\varepsilon$-blowup around the spectrum of $A$ (see also Figure \ref{contour}). Note that whenever $A$ is hermitian, the matrix $P$ is unitary and $\Vert P\Vert =\Vert P^{-1}\Vert =1$. Therefore, the ``eigenvalue maximal perturbation", namely $\varepsilon$, depends on the amplitude of the perturbation matrix (i.e. the term $\Vert H\Vert $) and on the ``lack of hermitian-ness" of the matrix $A$ (since we always have $\Vert P\Vert \cdot \Vert P^{-1}\Vert  \geqslant 1$ ).  

Here is the entertaining proof of the Bauer-Fike theorem.

\begin{proof}[Proof of the first point]
Let $\mu$ be an eigenvalue of the perturbed matrix $A+H$; then $A+H - \mu \mathrm{Id}$ is singular. Suppose that $\mu \notin \sigma(A)$; in this case, $D-\mu \mathrm{I}$ is nonsingular, and we have 
$$A+H-\mu \mathrm{I} = P(D- \mu\mathrm{I} ) (\mathrm{I} + (D- \mu \mathrm{I})^{-1}P^{-1}HP )P^{-1}.$$

This shows that $\mathrm{I} + (D- \mu \mathrm{I})^{-1}P^{-1}HP$ is singular, so $-1$ is an eigenvalue of $M:=(D- \mu \mathrm{I})^{-1}P^{-1}HP$; in particular, $1\leqslant \Vert M\Vert  \leqslant \Vert (D-\mu\mathrm{I})^{-1}\Vert \cdot \Vert P^{-1}\Vert \cdot\Vert H\Vert \cdot\Vert P\Vert $. It is easy to see that the norm of the diagonal matrix $(D-\mu \mathrm{I})^{-1}$ is $|\lambda_k - \mu|^{-1}$, where $k$ is such that $|\lambda_k - \mu| = \min |\lambda_i - \mu|$. This proves the inequality $|\lambda_k - \mu|\leqslant  \Vert P^{-1}\Vert \cdot\Vert H\Vert \cdot\Vert P\Vert $ which is the claim \eqref{bauer-fike}.
\end{proof}

\begin{figure}[H]\centering
\includegraphics[scale=0.8]{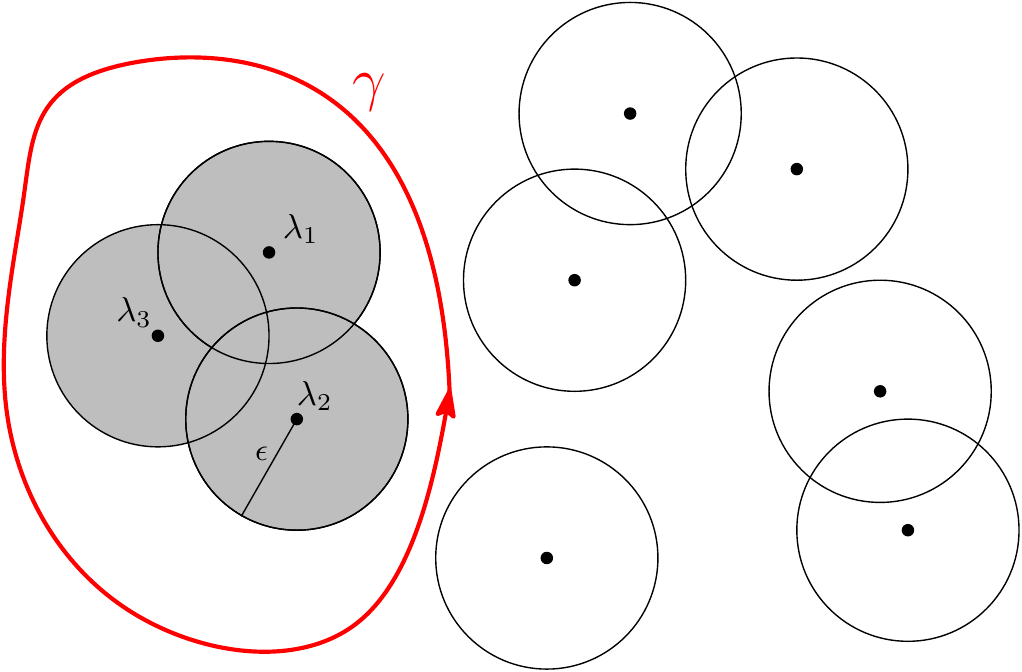}  
\caption{Black dots denote the spectrum of $A=A_0$. All the eigenvalues of $A+H$ are inside the circles and the number of eigenvalues of $A+H$ inside the grey zone is equal to exactly $3$. }\label{contour}
\end{figure}

\begin{proof}[Proof of the second point]
Let $s$ be in $[0,1]$. Note $A_s = A + s H$ and $p_s (z) = \mathrm{det}(A_s - z \mathrm{Id})$.  The eigenvalues of $A$ are the roots of $p_0$ and those of $A+H$ are the roots of $p_1$. Let $\gamma$ be a simple Jordan curve in the complex plane and let $U$ be the bounded connected component of $\mathbb{C} \setminus \gamma$ and $V$ the other component; suppose that $\cup_{i \in I} B(\lambda_i, \epsilon) \subset U$ and $\cup_{i \notin I} B(\lambda_i, \epsilon)\subset V$ (see figure \ref{contour}). Then, the argument principle yelds that the number $n(s)$ of roots of $p_s$ in $U$ is equal to 
$$\frac{1}{2i\pi} \oint_\gamma \frac{p'_s(\zeta)}{p_s(\zeta)}\mathrm{d}\zeta. $$

The polynomial $p_s$ depends continuously on the coefficients of $A_s$, so the application $s \mapsto n(s)$ is continuous from $[0,1]$ into $\mathbb{N}$, so by connectedness it is constant. We thus have $n(0) = n(1)$ and it is clear that $n(0) = \#I$.

\end{proof}

In order to use the Bauer-Fike theorem, we need a control on the \emph{condition number} of $P$, that is $c(P) = \Vert P\Vert \cdot\Vert P^{-1}\Vert $. When $A$ has rank $1$ this can be easily done; note that every rank $1$ matrix can be written $xy^\top$ with $x,y$ two nonzero vectors.

\begin{prop}\label{app_prop:rank1}
Let $A$ be a diagonalizable matrix with rank $1$, $A=PDP^{-1}$ with $P$ invertible and $D$ diagonal, say $D = \mathrm{diag}(\mu, 0, ..., 0)$ with $\mu$ the unique non-zero eigenvalue of $A$. Let $x,y$ be two vectors such that $A = xy^\top$.  Then, $\mu$ is equal to $\langle x, y\rangle$ and 
\begin{equation}\label{condition_number}
 c(P) \leqslant \frac{2\Vert x\Vert ^2\Vert y\Vert ^2}{\mu^2}.
\end{equation}
\end{prop}

\begin{proof}First, note that if $A = xy^\top$, then by Sylvester's determinant formula, for every $z$ we have $\mathrm{det}(z\mathrm{I}-xy^\top) = z^n(1 - z^{-1} y^\top x) = z^{n-1}(z-\langle x, y \rangle)$, so the eigenvalues of $A$ are $0$ and $\langle x,y\rangle$; indeed, if $A$ is diagonalizable and has rank $1$, then necessarily $\langle x,y \rangle \neq 0$ and $\mu=\langle x, y \rangle$. 

We first suppose that $\Vert x\Vert =\Vert y\Vert =1$. The right-eigenvector associated with $\mu$ is $x$, the left-eigenvector is $y^\top$. Every basis of $\mathrm{vect}(y)^\perp$ provides a family of right-eigenvectors for the eigenvalue $0$ and every basis of $\mathrm{vect}(x)^\perp$ provides a family of left-eigenvectors for the eigenvalue $0$. For every orthonormal basis of $\mathrm{vect}(y)^\perp$, say $(e_2, ..., e_n)$,  define a matrix by  $P = (x, e_2, ..., e_n)$. Then $P$ is a diagonalization matrix for $A$. Now, define $X=(y, e_2, ..., e_n)$: this matrix is unitary and we can check that 
$$X^* P = \begin{pmatrix}
\mu & 0 &0  & ... &&  0 \\
\langle x, e_2 \rangle & 1 & 0 & ... && 0 \\
\langle x, e_3 \rangle & 0 & 1 & ... && 0 \\
\vdots &&&\ddots && \vdots \\
&&&&1&0\\
\langle x, e_n \rangle & 0 & ...&& 0 & 1
\end{pmatrix}$$
We can also choose the basis $(e_i)$ so that $x$ belongs to $\mathrm{vect}(y, e_2)$. Let $b$ be a real number such that $x=\mu y + be_2$. As $\Vert x\Vert =1$ and $\mu \neq 0$, we have $b \in ]-1, 1[$ we must have $b^2 = 1-\mu^2$ and $b \in ]-1, 1[$. Then, 
$$X^* P = \begin{pmatrix}
\mu & 0 \\ 
b & 1 \\
&& \mathrm{I}_{n-2}
\end{pmatrix}.$$

We thus have proven that if 
$$R = \begin{pmatrix}
\mu & 0 \\ 
b & 1 
\end{pmatrix}$$

then $c(P) = c(X^* P) = c(R)$, and the condition number $c(R)$ can be computed; indeed, we find $c(R) = \sqrt{(1+|b|)/(1-|b|)}$. Remember that $|b|=\sqrt{1-\mu^2} \leqslant 1 - \mu^2/2$. Let $f$ be the increasing function defined on $[0, 1[$ by $f: t \mapsto \sqrt{(1+t)/(1-t)}$. Then $c(R) = f(|b|) \leqslant f(1-\mu^2/2)$ and it can be quickly checked, using $\sqrt{s-1} \leqslant s/2$, that $f(1-\mu^2/2) \leqslant 2/\mu^2$. We thus have proven that $c(P) \leqslant 2/\mu^2$. 
 
 Now, suppose that $\Vert x\Vert $ or $\Vert y\Vert $ are not equal to $1$ and define $\tilde{x} = x/\Vert x\Vert $ and $\tilde{y}=y/\Vert y\Vert $. Set $\tilde{A} = \tilde{x}\tilde{y}^\top$ so that $\Vert x\Vert \Vert y\Vert \tilde{A} = A$. Note $r=\Vert x\Vert \Vert y\Vert $. We have $\tilde{A} = P\tilde{D}P^{-1}$ with $\tilde{D} = \mathrm{diag}(\mu/r, 0, ..., 0)$ and $c(P) \leqslant 2r^2/\mu^2$ by the preceding arguments. As we also have $A = PDP^{-1}$, this yelds the final conclusion
 $$c(P) \leqslant \frac{2\Vert x\Vert ^2\Vert y\Vert ^2}{\mu^2}.$$
\end{proof}

We now conclude the proof of Lemma \ref{quanti_bauer_fike} on Theorem \ref{thm:Bauer-Fike} and Proposition \ref{app_prop:rank1}.

\begin{proof}[Proof of Lemma \ref{quanti_bauer_fike}] Apply the first point of the Bauer-Fike theorem to the matrix $M+H$: all the eigenvalues of $M+H$ lie in the union of the balls $B(\lambda, \varepsilon)$ with $\varepsilon = c(P) \Vert H\Vert $. As $M$ has rank $1$, apply Proposition \ref{app_prop:rank1}: $c(P) \leqslant 2\Vert x\Vert ^2\Vert y\Vert ^2\mu^{-2}$. Now, apply the second part of the Bauer-Fike theorem and suppose that $B(\mu, \varepsilon)$ and $B(0,\varepsilon)$ are disjoint. There is exactly one eigenvalue of $M$ in $B(\mu, \varepsilon)$ which is $\mu$, so there is exactly one eigenvalue of $M+H$ in $B(\mu, \varepsilon)$ and all other eigenvalues are in $B(0, \varepsilon)$. 
\end{proof}

\section{Proof of Theorem \ref{tech}.}\label{app:tech}

In this appendix, we prove Theorem \ref{tech} by adapting the arguments of the proof of Proposition 8 of \cite{bordenave2015} to our setting. All the required definitions and notations have already been introduced in Section \ref{tech_section} (page \pageref{tech_section}). The proof begins with the simple case where all edges of  $\mathfrak{p}$ are consistent and then goes on to the general case. We start with a preliminary remark.

\begin{remarque}
Remember that $c$ is a constant arbitrarily close to $1$. As $a \leqslant N \leqslant \sqrt{M}$, we have $(M-a)^{-1} \leqslant (M-N)^{-1} \leqslant (M-\sqrt{M})^{-1}$ and when $n$ is large, this is smaller than $c M^{-1}$. This inequality will be used multiple times in the proof of Theorem \ref{tech}. 
\end{remarque}

\subsection{Proof, part I: all edges are consistent.}

\subsubsection{Definitions of some useful sets.}

This section deals with the general case, where some edges might not be simple in the proto-path. However, we suppose for the moment that no edge is inconsistent. First, define sets $T,T_q$ as follows: 

\begin{itemize}
\item $T$ is the set of all edges such that $w'_i >0$. Those edges appear after $p$; they can appear both before and after $p$. We note $d = \#T$. 
\item $T_q$ is the set of all edges such that $w_i=q$ (with $q>0$).
\end{itemize}

The sets $T_q$ are distinct, but $T$ and $T_q$ might have a nonempty intersection. However, we still have 
$$ F(\mathfrak{p}) = \omega(\mathfrak{p}) \mathbf{E}\left[\prod_{q>0} \prod_{i \in T_q} B(y_i)^q \prod_{i \in  T}B'(y_i) \right].$$

We note $\mathbf{1}_\Omega = \prod_{i \in T}B'(y_i)$. Some of the edges $y_i$ with $i\in T$ might also appear in the proto-path before $p$, and in this case $B(y_i)B'(y_i) = (1-1/M)B'(y_i)$; we must keep track of these edges. We define: 
\begin{itemize}
\item $T'_q = \{i: w_i = q, w'_i>0 \}$ and $d'_q = \#T'_q$, 
\item $T^*_q = \{i: w_i = q, w'_i=0 \}$ so that $T'_q \cup T^*_q = T_q$ and $T'_q, T^*_q$ are disjoint.
\end{itemize}

Through the definition of $T^*_q$, we see that $|T^*_1| = a_1$, the number of simple (and consistent) edges of the proto-path, before $p$. Noting $\zeta = \sum_{q>0} qd'_q$, this yelds

\begin{equation}\label{TP_exp}
 F(\mathfrak{p}) = \omega(\mathfrak{p})(1-1/M)^{\zeta} \mathbf{E}\left[\mathbf{1}_\Omega \prod_{q>0} \prod_{i \in T^*_q} B(y_i)^q \right].
\end{equation}

The greatest contribution to the expectation \eqref{TP_exp} is due to the $q=1$ factor, so we are going to split the edges into two parts, those matched with another edge in some $T_q^*$ and those who are not.

\begin{itemize}
\item $\hat{T}_1$ is the set of all $i \in T^*_1$ such that there is a $j$ in $T^*_q$ for some $q>0$, such that if $y_i = (\mathbf{e}, \mathbf{f})$ and $y_j = (\mathbf{e}', \mathbf{f}')$, then either $\sigma(\mathbf{e}) = \mathbf{f}'$ or $\sigma(\mathbf{e}') = \mathbf{f}$ (or maybe both). 
\item For every $q>1$, $\hat{T}_q$ is the set of all $i \in T^*_q$ such that there is a $j$ in $T^*_1$, such that if $y_i = (\mathbf{e}, \mathbf{f})$ and $y_j = (\mathbf{e}', \mathbf{f}')$, then either $\sigma(\mathbf{e}) = \mathbf{f}'$ or $\sigma(\mathbf{e}') = \mathbf{f}$ (or maybe both). 
\item Finally, note $S_q = T^*_q \setminus \hat{T}_q$. If $i$ is in $S_1$ and  $y_i = (\mathbf{e}, \mathbf{f})$, then either $\sigma(\mathbf{e}) = \mathbf{f}$, or $\sigma(\mathbf{e})$ is some tail $\mathbf{f}$ which does not belong to any other edge of the proto-path $\mathfrak{p}$.
\end{itemize}

Those sets are random as they depend on the environment $\sigma$. Finally, note $X_q = \prod_{i \in S_q}B(y_i)^q$ and $\zeta'= \sum_{q \geqslant 1} q |\hat{T}_q|$. Then, we have

\begin{equation}
\mathbf{1}_\Omega \prod_{q>0} \prod_{i \in T^*_q} B(y_i)^q= \left(\frac{-1}{M} \right)^{\zeta'}\mathbf{1}_\Omega  \prod_{q>0}X_q .
\end{equation}

\subsubsection{First conditionning.}
Let $\mathcal{F}$ be the sigma-algebra generated by 
\begin{itemize}
\item the event $\Omega$, 
\item the matchings $\sigma(\mathbf{e})$ and $\sigma^{-1}(\mathbf{f})$ for every $y_i = (\mathbf{e}, \mathbf{f})$ with $i$ not in $S_1$. 
\end{itemize}

 \begin{lem}\label{lem:esperance1}With the notations given above, if $n$ is large enough we have
\begin{equation}\label{eq:esperance1}|\mathbf{E}[X_1|\mathcal{F}]|\leqslant 8 \left( \frac{3cN}{M\sqrt{M}}\right)^{|S_1|}.\end{equation}
\end{lem}

The proof of this lemma relies on the following remark: $|S_1|$ is measurable with respect to $\mathcal{F}$, so if $H$ is the number of $i \in S_1$ such that $\sigma(\mathbf{e}) \neq \mathbf{f}$, then
\begin{equation}\mathbf{E}[X_1 |\mathcal{F}] = \mathbf{E} \left[ \left( 1 - \frac{1}{M} \right)^{|S_1| - H} \left(\frac{1}{M} \right)^H \right].\end{equation}
We first give the law of $H$ conditionnally on $\mathcal{F}$. For simplicity we note $r = |S_1|$. 

\begin{lem} Given $\mathcal{F}$, for every $k$, we have
\begin{equation}\mathbf{P}(H=k | \mathcal{F}) = \frac{\binom{r}{k}(M-a)_{k}}{\sum_{k=0}^{r} \binom{r}{k} (M-a)_k}.\end{equation}
\end{lem}

\begin{proof}
Let us count the favorable cases for the event $\{H=k\}$ (again, reasonning conditionnally on $\mathcal{F}$). We have to choose those $k$ edges among the $r$ that haven't been matched yet. Once they have been chosen, all the $r-k$ remaining ones have to be matched with one tail not belonging to any edge in the proto-path $\mathfrak{p}$, and those edges are exactly $M-a$. Thus there are $\binom{r}{k}(M-a)_{k}$ favorable cases. The sum in the denominator is the sum of all cases.
\end{proof}

The reader can check that if $a \leqslant \sqrt{M}$, then if $n$ is large enough, for every $k \leqslant a$, we have $(M-a)_k \geqslant (M-a)^k /2 $, so if we note $Z=Z(a, r,M) = \sum_{k=0}^{r} \binom{r}{k} (M-a)_k$ then we have
\begin{equation}\label{ineq:Z}Z \geqslant \frac{1}{2} \sum_{k=0}^r \binom{r}{k} (M-a)^k = \frac{1}{2}(M-a+1)^r \geqslant \frac{1}{2}(M-a)^r.\end{equation}

On the other hand, 
\begin{align}
\mathbf{E}[X|\mathcal{F}] &=  \frac{1}{Z}\sum_{k=0}^r \binom{r}{k}(M-a)_k \left(1 - \frac{1}{M} \right)^k \left( \frac{-1}{M} \right)^{r-k}.\\
&=  \frac{(-1)^r}{Z}\mathbf{E}\left[ (M-a)_Q (-1)^Q \right]
\end{align}

where $Q$ is a random variable with law $\mathcal{B}(r, 1/M)$. Note that 
\begin{align*}
(M-a)_Q (-1)^Q  &= \prod_{n=0}^{Q-1} (M-a-n)\times (-1)  \\
&= \prod_{n=0}^{Q-1}(n-(M-a)).
\end{align*}

We now use (\cite{bordenave2015}, Lemma 9): 

\begin{lem}
Let $z \geqslant 1, r\in \mathbb{N}^*$ and $0<p\leqslant q <1$. Let $B$ a binomial random variable with parameters $r,p$. If $8(1-p/q))^2\leqslant 2zqr^2 \leqslant 1$, then
\begin{equation}
\left|\mathbf{E} \left[ \prod_{n=1}^B \left( zn - \frac{1}{q} \right) \right] \right| \leqslant 4 (r\sqrt{8zq})^r. 
\end{equation}

\end{lem}

We apply the lemma with 
\begin{itemize}
\item $q=1/(M-a)$ and $p=1/M$ (they satisfy $p \leqslant q$),
\item the random variable $Q$ as $B$,
\item $z=1$ (we can check that the condition of the lemma is verified).
\end{itemize}

Then, the lemma yelds 
\begin{equation}\left| \mathbf{E}\left[ (-1)^Q (M-a)_Q \right] \right| \leqslant 4 \left(r\sqrt{\frac{8}{M-a}} \right)^r .\end{equation}

We now plug this into $|\mathbf{E}[X_1|\mathcal{F}]|$. Using this and the preliminary remark on $n$ large and using inequality \eqref{ineq:Z}, we get 
\begin{align}
|\mathbf{E}[X_1|\mathcal{F}]|&\leqslant \frac{4}{Z}  \left(\frac{3r}{\sqrt{M-a}} \right)^r \\
&\leqslant \frac{8}{(M-a)^r}\left(\frac{3r}{\sqrt{M-a}} \right)^r \\
&\leqslant 8\left(\frac{3cN}{M\sqrt{M}} \right)^r
\end{align}
This ends the proof of Lemma \ref{lem:esperance1}. As a consequence, we get
\begin{align}
\mathbf{E}\left[ \mathbf{1}_\Omega \prod_{q>0} \prod_{i \in T^*_q} B(y_i)^q\right] &\leqslant 8 \mathbf{E}\left[\left( \frac{3cN}{M\sqrt{M}}\right)^{|S_1|} \left(\frac{1}{M} \right)^{\zeta'}\mathbf{1}_\Omega  \prod_{q>1}|X_q|\right].
\end{align}

\subsubsection{Second conditionning.}Let $\mathcal{G}_i$ be the $\sigma$-algebra generated by 
\begin{itemize}
\item the event $\Omega$, 
\item the matchings $\sigma(\mathbf{e})$ and $\sigma^{-1}(\mathbf{f})$ for every $y_j = (\mathbf{e}, \mathbf{f})$ with $i\neq j$. 
\end{itemize}
The random variables $\zeta',|S_q|$ are $\mathcal{G}_i$-measurable. Fix $i$ in some $S_q$. Then, as $q>1$ we have
\begin{align*}
\mathbf{E}[|B(y_i)|^q |\mathcal{G}_i]&\leqslant \mathbf{E}[|B(y_i)|^2 |\mathcal{G}_i] \\
&= \left(1 - \frac{1}{M} \right)^2 \mathbf{P}(\sigma(\mathbf{e}) = \mathbf{f}|\mathcal{G}_i) + \frac{1}{M^2}\mathbf{P}(\sigma(\mathbf{e}) \neq \mathbf{f}|\mathcal{G}_i)
\end{align*}

Conditionnally on $\mathcal{G}_i$, the head $\mathbf{e}$ cannot be matched with a tail belonging to $y_j$ for $j \neq i$ (recall the definition of $S_p$). Hence, if $M_i$ is the total number of unmatched tails after the matching of all the heads belonging to some $y_j$, we have $\mathbf{P}(\sigma(\mathbf{e}) = \mathbf{f}|\mathcal{G}_i) = 1/M_i$. Remember that if $n$ is large enough, we have $1/(M-a) \leqslant c/M$ (see the preliminary remark). Hence, we have
\begin{align*}
\mathbf{E}[|B(y_i)|^q |\mathcal{G}_i]&\leqslant \mathbf{E}[|B(y_i)|^2 |\mathcal{G}_i] \\
&\leqslant c \left(1 - \frac{1}{M} \right)^2 \frac{1}{M} + c\frac{1}{M^2}\left(1-\frac{1}{M} \right) \\
&\leqslant c\frac{1}{M}\left( 1-\frac{1}{M} \right) \leqslant \frac{c}{M}.
\end{align*}
By conditionning repeatedly on all the $\mathcal{G}_i$ for every $i$ in some $S_q$, for $q>1$, we get 
$$\mathbf{E}\left[ \mathbf{1}_\Omega \prod_{p>0} \prod_{i \in T^*_p} B(y_i)^p\right] 
\leqslant 8 \mathbf{E}\left[\left( \frac{3cN}{M\sqrt{M}}\right)^{|S_1|} \left(\frac{c}{M} \right)^{\zeta' + \sum_{q>1}q|S_q|}\mathbf{1}_\Omega \right].$$
As $\zeta' = \sum_{q>0}q|\hat{T}_q|$, we have $\zeta' + \sum_{q>1}q|S_q| = |\hat{T}_1|+ \sum_{q>1}q (|\hat{T}_q|+|S_q|) = |\hat{T}_1|+ \sum_{q>1}q |T^*_q|$.

\subsubsection{Third conditionning. }We now condition on the sigma-algebra $\mathcal{G}$ generated by the matchings $\sigma(\mathbf{e})$ and $\sigma^{-1}(\mathbf{f})$ for every $y_j = (\mathbf{e}, \mathbf{f})$ with $i\notin T$. Note $\tilde{\Omega}$ the event ``no half-edge belonging to $y_i$ for some $i \notin T$ has been matched with a half-edge $y_j$ with $j$ in $T$". This event is $\mathcal{G}$-measurable, and when $n$ is large enough, 
$$\mathbf{E}[\mathbf{1}_\Omega | \mathcal{G}] \leqslant \mathbf{1}_{\tilde{\Omega}}\left( \frac{c}{M} \right)^{|T|} \leqslant \left( \frac{c}{M} \right)^{|T|}.$$ 

Hence,  
\begin{equation}\mathbf{E}\left[\mathbf{1}_\Omega  \prod_{q>0} \prod_{i \in T^*_q} B(y_i)^q\right] 
\leqslant 8\mathbf{E}\left[\left( \frac{3cN}{M\sqrt{M}}\right)^{|S_1|} \left(\frac{c}{M} \right)^{|\hat{T}_1|+ \sum_{q>1}q|T^*_q|+|T|} \right].\end{equation}

\subsubsection{Endstep.}Recall \eqref{TP_exp}; we have 
\begin{align}
 F(\mathfrak{p}) &= \omega(\mathfrak{p})(1-1/M)^{\zeta} \mathbf{E}\left[\mathbf{1}_\Omega \prod_{q>0} \prod_{i \in T^*_q} B(y_i)^q \right] \\
 &\leqslant 8\omega(\mathfrak{p})(1-1/M)^\zeta  \mathbf{E}\left[\left( \frac{3cN}{M\sqrt{M}}\right)^{|S_1|} \left(\frac{c}{M} \right)^{|\hat{T}_1|+ \sum_{q>1}q |T^*_q|+|T|} \right] \\
 &\leqslant 8\omega(\mathfrak{p})\left( \frac{3cN}{M\sqrt{M}}\right)^{a_1}\mathbf{E}\left[\left( \frac{3cN}{M\sqrt{M}}\right)^{-|\hat{T}1|} \left(\frac{c}{M} \right)^{|\hat{T}_1|+ \sum_{q>1}q|T^*_q|+|T|} \right] \\
  &\leqslant 8\omega(\mathfrak{p})\left( \frac{3cN}{M\sqrt{M}}\right)^{a_1}\mathbf{E}\left[\left( \frac{3N}{\sqrt{M}}\right)^{-|\hat{T}1|} \left(\frac{c}{M} \right)^{ \sum_{q>1}q|T^*_q|+|T|} \right] .
\end{align}
where in the third line we used $a_1 = |T^*_1| = |S_1|+|\hat{T}_1|$. By construction, we have $\sum_{q>0} |T^*_q|+|T|=a$, therefore
\begin{equation}\sum_{q>1}q|T^*_q|+|T| = a - |T ^*_1|+\sum_{q>1}(q-1)|T^*_q|\geqslant a-a_1\end{equation}
and we have $(c/M)^{\sum_{q>1}q|T^*_q|+|T|} \leqslant (c/M)^{a-a_1}$. This finally yields 
$$F(\mathfrak{p}) \leqslant 8\omega(\mathfrak{p})\left(\frac{c}{M} \right)^{a-a_1}\left( \frac{3cN}{M\sqrt{M}}\right)^{a_1}\mathbf{E}\left[\left( \frac{3N}{\sqrt{M}}\right)^{-|\hat{T}1|}  \right].$$

In the next lemma, we bound the expectation on the right side. 
\begin{lem}
If $n$ is large enough,

\begin{equation}
\mathbf{E}\left[\left( \frac{3N}{\sqrt{M}}\right)^{-|\hat{T}1|}  \right] \leqslant 3.
\end{equation} 
\end{lem}

\begin{proof}
We have 
$$\mathbf{E}\left[\left( \frac{3N}{\sqrt{M}}\right)^{-|\hat{T}1|}  \right] = \sum_{\ell = 0}^\infty \mathbf{P}(|\hat{T}_1|=\ell) \left( \frac{\sqrt{M}}{3N}\right)^{\ell}.$$

Using the pigeonhole principle, on the event $\{|\hat{T}_1|=\ell \}$, at least $\lfloor \ell / 2 \rfloor$ couples of edges $(y,y')$ are ``mismatched", which means that $\sigma(\mathbf{e})=\mathbf{f}'$ or $\sigma(\mathbf{e}') = \mathbf{f}$. A (very) crude bound for the choice of those  $\lfloor \ell / 2 \rfloor$ couples is $(a^2)^{\lfloor \ell / 2 \rfloor}$. For each choice of those $\lfloor \ell/2\rfloor$ couples, the probability that they are indeed mismatched is at most $(1/(M-a))^{\lfloor \ell/2 \rfloor}$ which is smaller than $(2/\sqrt{M})^\ell$ if $n$ is large enough. In the end, we get 
$$\mathbf{P}(|\hat{T}_1| = \ell) \leqslant a^\ell \left(\frac{2}{\sqrt{M}} \right)^\ell$$
Finally, as $a\leqslant N$, we have 
$$\mathbf{E}\left[\left( \frac{3t}{\sqrt{M}}\right)^{-|\hat{T}1|}  \right] \leqslant  \sum_{\ell = 0}^\infty  \left( \frac{2a}{3N}\right)^{\ell} \leqslant \sum_{\ell = 0}^\infty \left( \frac{2}{3}\right)^{\ell}$$
which ends the proof of the lemma.
\end{proof}

We finally get the desired bound, that is 
\begin{equation}\label{tech:consistent}
|F(\mathfrak{p})| \leqslant 24 \cdot \omega(\mathfrak{p})\left( \frac{c}{M} \right)^{a} \left(\frac{N}{\sqrt{M}} \right)^{a_1}
\end{equation}

\subsection{Proof, part II: some edges are not consistent.}

We now suppose some edges are not consistent: for example, there might be in $\mathfrak{p}$ two edges having the form $y=(\mathbf{e}, \mathbf{f})$ and $y'=(\mathbf{e}, \mathbf{f}')$ with $\mathbf{f} \neq \mathbf{f}'$. Without loss of generality we can suppose $y=y_1$ and $y' = y_a$. The contributions of those two edges in the product has the form $B(y)^w B'(y)^z B(y')^{w'}B'(y')^{z'}$. Note that $B'(y)B'(y')$ is always zero. From this, we see that we can't have $z$ and $z'$ be both non zero. Without loss of generality, we suppose that $z'=0$.

\subsubsection{First case: $z\neq 0$.}

Here, we immediately have 
\begin{equation}
B(y)^w B'(y)^z B(y')^{w'} = B'(y)B(y)^{w}\left(  - \frac{1}{M}\right)^{w'}.
\end{equation}
This expression does not longer rely upon $y'$. Hence, in this case, we have 
\begin{equation}F(\mathfrak{p}) = \frac{1}{(d_\mathbf{e}^+)^{w'}}\left(  - \frac{1}{M}\right)^{w'} F(\mathfrak{q})\end{equation}
where the proto-path $\mathfrak{q}$ is the proto-path $\mathfrak{p}$ without all the $w'$ instances of the $y'$ edge. This new proto-path $\mathfrak{q}$ has length $N-w'$, has $a'=a-1$ distinct edges before $p$, and its number of simple, consistent edges before $p$ is greater than $a_1$. 

\subsubsection{Second case: $z=0$.}

The product is now reduced to $B(y)^w B(y')^{w'}$. After a short development we find that 
\begin{equation}B(y)^w B(y')^{w'} = B(y)^w \left( - \frac{1}{M}\right)^{w'}  + \left( - \frac{1}{M}\right)^{w'}B(y')^{w} - \left( - \frac{1}{M}\right)^{w'+w}
\end{equation}
Hence, $F(\mathfrak{p})$ splits into three parts: 
\begin{align}
F(\mathfrak{p}) &=  \left( - \frac{1}{d_\mathbf{e}^+M}\right)^{w'}F(\mathfrak{q}) +\left( - \frac{1}{d_\mathbf{e}^+ M}\right)^{w}F(\mathfrak{q}') - \left( - \frac{1}{d_\mathbf{e}^+M}\right)^{w'+w}F(\mathfrak{q}''). 
\end{align}

All the three new proto-paths $\mathfrak{q}, \mathfrak{q}', \mathfrak{q}''$ now have 
\begin{itemize}
\item length $N-w', t-w$ and $N-w'-w$,
\item at most $a-1$ distinct edges,
\item less inconsistent edges than $\mathfrak{p}$.
\end{itemize}

\subsubsection{Iteration of the procedure.}

We repeat the procedure as many times as needed to get rid of every inconsistent edge. Each step gives rise to at most $3$ terms having the form 
$$\pm \left(\frac{1}{M}\right)^\alpha \omega(\mathfrak{p}) \mathbf{E}\left[\prod_{i=1}^{a-1}B(y_i)^{w_i} B(y_i) \right]$$
or 
$$\pm \left(\frac{1}{M}\right)^\alpha \omega(\mathfrak{p}) \mathbf{E}\left[\prod_{i=2}^{a}B(y_i)^{w_i} B(y_i) \right]$$
or
$$\pm \left(\frac{1}{M}\right)^\alpha \omega(\mathfrak{p}) \mathbf{E}\left[\prod_{i=2}^{a-1}B(y_i)^{w_i} B(y_i) \right]$$
where $\alpha$ is either $w_a,w_1$ or $w_a+w_1$.

Now, we repeat the procedure for each term. Each step removes one inconsistent edge, so there are no more than $3^b$ steps, and in the end we get at most $3^b$ terms. In each one of the final $3^b$ terms, all edges are consistent so we can apply \eqref{tech:consistent}. The number of simple, consistent edges of those new proto-paths is greater than $a_1$ but still smaller than $N$. Hence, applying \eqref{tech:consistent} to each term, we can bound $|F(\mathfrak{p})|$ with at most $3^b$ terms having the form
$$ 24\cdot \omega(\mathfrak{p})\left( \frac{c}{M} \right)^{a} \left(\frac{N}{\sqrt{M}} \right)^{a_1}$$
which yields the final desired result \eqref{tech:equation1}.

\bibliography{bibli}

\newcommand{\etalchar}[1]{$^{#1}$}
\begin{thebibliography}{dlHRV93}

\bibitem[Alo86]{Alon1986}
Noga Alon.
\newblock Eigenvalues and expanders.
\newblock {\em Combinatorica}, 6(2):83--96, 1986.

\bibitem[BC12]{bordenave_chafai_survey}
Charles Bordenave and Djalil Chafaï.
\newblock Around the circular law.
\newblock {\em Probab. Surveys}, 9:1--89, 2012.

\bibitem[BCC08a]{BCC}
C.~{Bordenave}, P.~{Caputo}, and D.~{Chafai}.
\newblock {Circular Law Theorem for Random Markov Matrices}.
\newblock {\em ArXiv e-prints}, August 2008.

\bibitem[BCC08b]{BCC_revex}
C.~{Bordenave}, P.~{Caputo}, and D.~{Chafai}.
\newblock {Spectrum of large random reversible Markov chains: two examples}.
\newblock {\em ArXiv e-prints}, November 2008.

\bibitem[BCC09]{BCC_rev}
C.~{Bordenave}, P.~{Caputo}, and D.~{Chafa{\"i}}.
\newblock {Spectrum of large random reversible Markov chains: Heavy-tailed
  weights on the complete graph}.
\newblock {\em ArXiv e-prints}, March 2009.

\bibitem[BCCP16]{BCCP}
C.~{Bordenave}, P.~{Caputo}, D.~{Chafa{\"i}}, and D.~{Piras}.
\newblock {Spectrum of large random Markov chains: heavy-tailed weights on the
  oriented complete graph}.
\newblock {\em ArXiv e-prints}, October 2016.

\bibitem[BCS15]{caputo_salez_bordenave}
C.~{Bordenave}, P.~{Caputo}, and J.~{Salez}.
\newblock {Random walk on sparse random digraphs}.
\newblock {\em ArXiv e-prints}, August 2015.

\bibitem[BCZ17]{cook_dreg2}
A.~{Basak}, N.~{Cook}, and O.~{Zeitouni}.
\newblock {Circular law for the sum of random permutation matrices}.
\newblock {\em ArXiv e-prints}, May 2017.

\bibitem[BDH18]{dumitriu_bireg}
G.~{Brito}, I.~{Dumitriu}, and K.~D. {Harris}.
\newblock {Spectral gap in random bipartite biregular graphs and applications}.
\newblock {\em ArXiv e-prints}, April 2018.

\bibitem[BF60]{bauer}
F.L. Bauer and C.T. Fike.
\newblock Norms and exclusion theorems.
\newblock {\em Numerische Mathematik}, 2:137--141, 1960.

\bibitem[BLM15]{bordenave_lelarge_massoulie}
C.~{Bordenave}, M.~{Lelarge}, and L.~{Massouli{\'e}}.
\newblock {Non-backtracking spectrum of random graphs: community detection and
  non-regular Ramanujan graphs}.
\newblock {\em ArXiv e-prints}, January 2015.

\bibitem[Bol01]{bollobas-rg}
Béla Bollobás.
\newblock {\em Random graphs}.
\newblock Cambridge studies in advanced mathematics. Cambridge university
  press, Cambridge, New York (N. Y.), Melbourne, 2001.
\newblock Paru précédemment à : London ; Orlando (FL) ; Sydney : Academic
  Press, 1985.

\bibitem[{Bor}15]{bordenave2015}
C.~{Bordenave}.
\newblock {A new proof of Friedman's second eigenvalue Theorem and its
  extension to random lifts}.
\newblock {\em ArXiv e-prints}, February 2015.

\bibitem[BQZ18]{bordenave_qiu}
C.~{Bordenave}, Y.~{Qiu}, and Y.~{Zhang}.
\newblock {Spectral gap of sparse bistochastic matrices with exchangeable rows
  with application to shuffle-and-fold maps}.
\newblock {\em ArXiv e-prints}, May 2018.

\bibitem[CF04]{cooper-frieze}
C~Cooper and A~Frieze.
\newblock The size of the largest strongly connected component of a random
  digraph with a given degree sequence.
\newblock {\em Combinatorics, Probability and Computing}, 13(3):319 -- --338, 5
  2004.

\bibitem[Coo11]{cooper2011random}
Colin Cooper.
\newblock Random walks, interacting particles, dynamic networks: Randomness can
  be helpful.
\newblock In {\em International Colloquium on Structural Information and
  Communication Complexity}, pages 1--14. Springer, 2011.

\bibitem[{Coo}15]{cook-srrd}
N.~A. {Cook}.
\newblock {The circular law for signed random regular digraphs}.
\newblock {\em ArXiv e-prints}, August 2015.

\bibitem[{Coo}17]{cook_rrd}
N.~A. {Cook}.
\newblock {The circular law for random regular digraphs}.
\newblock {\em ArXiv e-prints}, March 2017.

\bibitem[Dia96]{diaconis1996cutoff}
Persi Diaconis.
\newblock The cutoff phenomenon in finite markov chains.
\newblock {\em Proceedings of the National Academy of Sciences},
  93(4):1659--1664, 1996.

\bibitem[dlHRV93]{MR1231179}
Pierre de~la Harpe, A.~Guyan Robertson, and Alain Valette.
\newblock On the spectrum of the sum of generators for a finitely generated
  group.
\newblock {\em Israel J. Math.}, 81(1-2):65--96, 1993.

\bibitem[DSV03]{ramanujan-book}
Giuliana~P. Davidoff, Peter Sarnak, and Alain Valette.
\newblock {\em Elementary number theory, group theory, and Ramanujan graphs}.
\newblock London Mathematical Society student texts. Cambridge University
  Press, Cambridge (UK), New York, 2003.

\bibitem[Fil91]{fill1991eigenvalue}
James~Allen Fill.
\newblock Eigenvalue bounds on convergence to stationarity for nonreversible
  markov chains, with an application to the exclusion process.
\newblock {\em The annals of applied probability}, pages 62--87, 1991.

\bibitem[FK81]{komlos}
Zolt{\'a}n F{\"u}redi and J{\'a}nos Koml{\'o}s.
\newblock The eigenvalues of random symmetric matrices.
\newblock {\em Combinatorica}, 1(3):233--241, 1981.

\bibitem[Fri04]{friedman}
Joel Friedman.
\newblock A proof of alon's second eigenvalue conjecture and related problems.
\newblock {\em CoRR}, cs.DM/0405020, 2004.

\bibitem[HLW06]{hoory2006expander}
Shlomo Hoory, Nathan Linial, and Avi Wigderson.
\newblock Expander graphs and their applications.
\newblock {\em Bulletin of the American Mathematical Society}, 43(4):439--561,
  2006.

\bibitem[Kes59]{kesten59}
Harry Kesten.
\newblock Symmetric random walks on groups.
\newblock {\em Trans. Amer. Math. Soc.}, 92:336--354, 1959.

\bibitem[LP16]{cutoff_ramanujan}
Eyal Lubetzky and Yuval Peres.
\newblock Cutoff on all ramanujan graphs.
\newblock {\em Geometric and Functional Analysis}, 26(4):1190--1216, 2016.

\bibitem[LPW09]{peres}
David~Asher Levin, Yuval Peres, and Elizabeth~Lee Wilmer.
\newblock {\em Markov chains and mixing times}.
\newblock Providence, R.I. American Mathematical Society, 2009.
\newblock With a chapter on coupling from the past by James G. Propp and David
  B. Wilson.

\bibitem[LS{\etalchar{+}}10]{cutoff_dreg}
Eyal Lubetzky, Allan Sly, et~al.
\newblock Cutoff phenomena for random walks on random regular graphs.
\newblock {\em Duke Mathematical Journal}, 153(3):475--510, 2010.

\bibitem[{Mas}13]{massoulie_rama}
L.~{Massoulie}.
\newblock {Community detection thresholds and the weak Ramanujan property}.
\newblock {\em ArXiv e-prints}, November 2013.

\bibitem[MT{\etalchar{+}}06]{montenegro}
Ravi Montenegro, Prasad Tetali, et~al.
\newblock Mathematical aspects of mixing times in markov chains.
\newblock {\em Foundations and Trends{\textregistered} in Theoretical Computer
  Science}, 1(3):237--354, 2006.

\bibitem[Nil91]{NILLI1991207}
A.~Nilli.
\newblock On the second eigenvalue of a graph.
\newblock {\em Discrete Mathematics}, 91(2):207 -- 210, 1991.

\bibitem[NSW01]{newman}
M.~E.~J. {Newman}, S.~H. {Strogatz}, and D.~J. {Watts}.
\newblock {Random graphs with arbitrary degree distributions and their
  applications}.
\newblock {\em journal}, 64(2):026118, August 2001.

\bibitem[{Par}18]{2018arXiv180408028P}
O.~{Parzanchevski}.
\newblock {Ramanujan Graphs and Digraphs}.
\newblock {\em ArXiv e-prints}, April 2018.

\end{thebibliography}

\end{document}